\documentclass[a4paper]{amsart}
\usepackage{amsfonts}
\usepackage{amsmath,latexsym,amssymb,amsfonts, amsthm}
\usepackage{esint}
\usepackage{cite}
\usepackage{graphicx}
\usepackage{amscd}
\usepackage{color}
\usepackage{bm}             
\usepackage{enumerate}
\usepackage[dvips]{epsfig}
\usepackage{psfrag}
\usepackage{epsfig}
\usepackage{comment}

\usepackage[
  hmarginratio={1:1},
  vmarginratio={1:1},
  textwidth=16.7cm,
  textheight=22.5cm,
  heightrounded,
]{geometry}

\usepackage{tikz}
\usepackage{graphicx,color}
\usepackage[colorlinks]{hyperref}
\hypersetup{linkcolor=blue,citecolor=blue,filecolor=black,urlcolor=blue}

\newcommand{\beq}{\begin{equation}}
\newcommand{\eeq}{\end{equation}}

\newtheorem{thm}{Theorem}

\newtheorem{proposition}[thm]{Proposition}

\newtheorem{cor}[thm]{Corollary}
\newtheorem{lem}[thm]{Lemma}
\newtheorem{lemma}[thm]{Lemma}

\theoremstyle{definition}

\newtheorem{remark}[thm]{Remark}
\newtheorem{definition}[thm]{Definition}

\numberwithin{thm}{section}

\numberwithin{equation}{section}

\newcommand{\R}{\mathbb{R}}
\newcommand{\cG}{\mathcal{G}}

\newcommand{\x}{\times}
\newcommand{\cE}{\mathcal{E}}

\newcommand{\pa}{\partial}



\title[The NLS equation on products of $\R^N$ and compact metric graphs]{The nonlinear Schr\"odinger equation on
products of $\R^N$ and compact metric graphs}
\author{Nicola Soave, Gianmaria Verzini and Lorenzo Villata}

\address{Nicola Soave and Lorenzo Villata \newline\indent Dipartimento di Matematica ``G. Peano'' \newline\indent Universit\`a degli Studi di Torino \newline\indent Via Carlo Alberto 10, 10123, Torino, Italy} \email{nicola.soave@unito.it, lorenzo.villata@unito.it}

\address{Gianmaria Verzini \newline\indent Dipartimento di Matematica \newline\indent Politecnico di Milano \newline\indent Piazza Leonardo da Vinci 32, 20133, Milano, Italy} \email{gianmaria.verzini@polito.it}

\begin{document}

\begin{abstract}
We study the stationary focusing nonlinear Schr\"odinger equation on the product $\R^N \x \cG$ of the Euclidean space with a compact metric graph, in the mass-constrained variational setting. Such a product is a hybrid structure of a new type: all its faces are $(N+1)$-dimensional and are glued along interfaces of codimension one, so that the energy space is a genuine Sobolev space, while both the metric and the topology of the graph enter the variational problem. We first develop the functional framework, giving two equivalent descriptions of $H^1(\R^N \x \cG)$, introducing partial rearrangements in each of the two variables together with the corresponding P\'olya--Szeg\H{o} inequalities, and proving Gagliardo--Nirenberg inequalities in a localized form, with a comparison of the optimal constants with those of $\R^{N+1}$ and of the half-space. We then study the mass-constrained problem. Ground states exist for every mass when
\(2<p<2_*:=2+4/(N+1)\). At the critical exponent \(p=2_*\),
they exist below a graph-dependent threshold lying between one half
of the Euclidean critical mass and the full Euclidean critical mass.
The latter value is attained when the graph admits a cycle covering,
whereas the threshold is exactly halved in the presence of a terminal
edge. In both cases, the threshold is sharp. For $2_*<p<2+4/N$ global minimizers do not exist, but we prove the existence of local minimizers below a further mass threshold, for which we give an explicit lower bound. Finally, we describe the dimensional crossover: below a critical mass the minimizers do not depend on the graph variable, and we characterize the threshold below which the semi-trivial solution is a local minimizer in terms of the first nonzero eigenvalue of the Kirchhoff Laplacian on $\cG$.
\end{abstract}

\maketitle

\section{Introduction}

Let $\cG$ be a compact, connected metric graph, and let $N \geq 1$. In this paper we study the existence and the qualitative properties of standing waves for the focusing nonlinear Schr\"odinger (NLS) equation on the higher dimensional graph-like structure given by the product $\R^N \x \cG$:
\begin{equation}
	\label{intro NLS}
	i\partial_t \Psi + \Delta_{x,y}\Psi + |\Psi|^{p-2}\Psi=0, \qquad (x,y) \in \R^N \x \cG,
\end{equation}
where $x$ denotes the variable in $\R^N$ and $y$ the variable in $\cG$, and $\Delta_{x,y}=\Delta_x+\partial_y^2$ is understood edgewise, with continuity and Kirchhoff conditions on the interfaces $\R^N \x \{v\}$, $v$ a vertex of $\cG$. Looking for solutions of the form $\Psi(t,x,y)=e^{i\omega t}u(x,y)$ leads to the stationary problem
\begin{equation}
	\label{intro stationary}
	-\Delta_x u-\partial_y^2u+\omega u=|u|^{p-2}u \qquad \text{on } \R^N \x \cG,
\end{equation}
and it is well known that a natural way to produce solutions of \eqref{intro stationary}, with the frequency $\omega$ arising as a Lagrange multiplier, consists in minimizing the NLS energy
\begin{equation}\label{NLS energy}
E_p\left(u, \R^N \x \cG\right)=\frac{1}{2}\left\|\nabla_{x,y}u\right\|_{L^2(\R^N \x \cG)}^2-\frac{1}{p}\|u\|_{L^p(\R^N \x \cG)}^p
\end{equation}
under the mass constraint $\|u\|^2_{L^2(\R^N \x \cG)}=\mu$. The mass is a conserved quantity of \eqref{intro NLS}, and prescribing it is both physically meaningful and, from the point of view of the dynamics, the natural way to obtain orbitally stable solutions. Following a well-established terminology, we call \emph{ground state of mass $\mu$} any minimizer of $E_p(\cdot,\R^N \x \cG)$ on the constraint, and we denote by $\cE_p(\mu, \R^N \x \cG)$ the corresponding infimum.

The interest of the domain $\R^N \x \cG$ is twofold. On the one hand, products of this kind are natural models for \emph{waveguides with a branched cross section}, or, dually, for branched structures which are thick in $N$ directions and thin in the remaining one: the 
\emph{total lenght} $l(\cG)$ of the graph measures the size of the transversal structure, and the topology of $\cG$ encodes how its branches are connected. On the other hand, and this is the point of view we emphasize here, $\R^N \x \cG$ is a \emph{hybrid domain of a new type}. It is not a manifold, because of the vertices of $\cG$ of degree different from two; it is not a metric graph, because each edge contributes a full $(N+1)$-dimensional face; and it is not a hybrid in the sense of mixed dimensionality, because all its faces $\R^N \x e$ have the same dimension. What makes it a genuinely new object is the way the pieces are glued: finitely many copies of $\R^N \x [0,\ell_e]$ are identified along the $N$-dimensional interfaces $\R^N \x \{v\}$, so that the singular set is of codimension one inside each face. As we shall see, this is exactly what allows us to keep working within a Sobolev-space framework, while at the same time producing a nontrivial interplay between the topology of $\cG$, its total length, and the variational structure of the problem.

\subsection*{Relations with the existing literature}

The problem we consider lies at the crossroads of three lines of research.

\subsubsection*{NLS on metric graphs.} The study of \eqref{intro NLS} on a metric graph, corresponding to the case in which the factor $\R^N$ is absent (i.e. it collapses to a zero dimensional vector space), is by now a well-developed subject, motivated by the analysis of the propagation of waves in branched structures and networks \cite{Meh, Noja, branched, SMSS, BoCa}, and systematically initiated, in the mass-constrained variational setting on non-compact graphs, by Adami, Serra and Tilli \cite{2015, threshold, critical}. Their results revealed a phenomenon which has no counterpart in $\R^n$: whether ground states exist or not depends not only on the mass $\mu$ and on the power $p$, but also on the topology of the graph, the crucial notion being the possibility of covering $\cG$ by cycles, see \cite{threshold}. Since then, the literature has grown in several directions: compact and noncompact graphs, also with localized nonlinearities \cite{compact, CacDovSer, localized_subcritical, localized_critical}, periodic graphs and lattices \cite{square_grid, periodic, DoTe19, Pank, PeSc}, combined nonlinearities \cite{combined, LiZhLi, SoaVil}, the $L^2$-supercritical regime \cite{supercritical, compact_supercritical, localized_supercritical}, potentials \cite{attractive_potential}; we refer to \cite{intro_graphs, kuchment} for the linear theory. Two features of this literature will play a role here: the rearrangement techniques on graphs introduced in \cite{2015} and refined in \cite{threshold}, which we shall extend to the product setting, and the analysis of the stability of the constant solution on a compact graph carried out in \cite{CacDovSer}, of which our last section is the analogue.

\subsubsection*{NLS on product spaces.} The case in which the compact factor is a manifold, that is $\R^N \x M^k$ with $M^k$ a $k$-dimensional compact Riemannian manifold, was studied by Terracini, Tzvetkov and Visciglia in \cite{TerTzvVis}. Beside the existence of ground states in the range in which the problem is $L^2$-subcritical with respect to the dimension $N+k$, they discovered a \emph{dimensional crossover} driven by the mass: for small masses the ground state does not depend on the variable of $M^k$, and the problem behaves as the NLS equation on $\R^N$, whereas for large masses the ground state is genuinely $M^k$-dependent and the full $(N+k)$-dimensional nature of the domain is seen. The threshold between the two regimes is a critical mass, whose existence is proved in \cite[Theorem 1.3]{TerTzvVis}. This phenomenon is precisely the one we reproduce, in a sharper form, on $\R^N \x \cG$. Local minimizers on product spaces, in the regime which is 
$L^2$-supercritical for $\R^N \x M^k$, but both $L^2$-subcritical for $\R^N$ and Sobolev subcritical for $\R^N \x M^k$, have been recently investigated by Pierotti, Verzini and Yu \cite{PieVerYu}; a related mechanism, in which the confinement acts only in some directions, appears in \cite{BelBouJeaVis, MolRieVer}. A dimensional crossover of a different nature, produced by the metric rather than by the mass, is the one of periodic graphs \cite{square_grid, honeycomb, spatial_grid, SoaVil}, where a continuum of critical exponents interpolates between the one- and the two-dimensional ones. 

\subsubsection*{NLS on hybrids.} Domains of \emph{mixed dimensionality} were introduced, in the linear setting, by Exner and \v{S}eba \cite{ExnerSeba}, who classified the self-adjoint realizations of the Laplacian on the \emph{hybrid plane}, a half-line attached to a plane at a single point. The nonlinear problem on such a structure has been addressed very recently by Adami, Boni, Carlone and Tentarelli \cite{hybrid, hybrid_plane}, who proved the first existence and nonexistence results for NLS ground states on a mixed-dimensional manifold. It is worth stressing how different that situation is from ours, and why we nevertheless regard both as hybrid problems. In \cite{hybrid_plane} the two components have different dimensions and meet at a point: a state is a \emph{pair} $(u,v) \in H^1(\R^+) \x L^2(\R^2)$, the coupling between the components is realized through contact interactions and is measured by three parameters, and the very definition of the energy domain requires the singular decomposition of the planar component dictated by the theory of self-adjoint extensions, since a two-dimensional Dirac interaction is not controlled by the gradient term. Consequently, existence is governed by the interplay of the coupling constants, and the compactness threshold is the energy of the one-dimensional soliton. In our setting, instead, the gluing takes place along interfaces of codimension one, the state is a single function, the coupling is the Kirchhoff condition, and the energy space is a genuine Sobolev space, as we show in Section \ref{sez spazio}. What survives of the hybrid nature is the fact that the domain is not a manifold and that its dimensional features are not uniform in scale: at scales much larger than $l(\cG)$ the space looks $N$-dimensional, at scales much smaller than $l(\cG)$ it looks $(N+1)$-dimensional, and the two regimes are selected by the mass.

\subsubsection*{NLS on open books.} A closely related, and independently motivated, notion of higher dimensional quantum graph has been considered very recently under the name of \emph{open books}: finitely many $n$-dimensional ``pages'', glued along $(n-1)$-dimensional ``bindings''. Akduman and Kuchment \cite{openbook} set up the linear theory, equipping such varieties with the Laplace-Beltrami operator on the pages and with junction conditions at the bindings, and characterizing which of these conditions produce a self-adjoint operator, in complete analogy with the vertex conditions of quantum graphs. Although their analysis restricts to pages given by compact manifolds, this construction can be adapted to our setting without modifications. In particular, $\R^N \x \cG$ is an open book whose pages are the slabs $\R^N \x e$, $e$ an edge of $\cG$, and whose bindings are the copies $\R^N \x \{v\}$ of $\R^N$ attached to the vertices, the Kirchhoff condition being one of the admissible self-adjoint junction conditions.

Nonlinear problems on such structures have been investigated in a series of recent papers by Le Coz and Shakarov, whose results are closely related to ours. In \cite{fracturedstrip} they consider a strip $\R \x [0,L]$ with Neumann conditions and a $\delta$ interaction supported on the transversal segment $\{0\} \x [0,L]$, and they prove, among other results, that in the case of attractive delta interaction, for every $\tilde{m}>0$ fixed, the energy minimizer of mass $m = \tilde m L$ does not depend on the transversal variable when $L$ is small, and does depend on it when $L$ is large. This is the framework closest to the present paper, since the minimization is performed under a mass constraint; the case without the delta interaction, which is excluded there, 
corresponds to our Theorem \ref{thm 1.3} with $N=1$ and $\cG$ a segment. Let us point out that \cite{fracturedstrip} produces two distinct thresholds, one for each of the two regimes, which are not shown to coincide, whereas Theorem \ref{thm 1.3} provides a single critical mass. In \cite{LeCozShakarov} the same authors turn to two dimensional open books, whose pages are isometric to rectangles and whose bindings are intervals; they develop the corresponding functional setting and prove the existence of \emph{action} ground states, that is of minimizers of the action on the Nehari manifold. For \emph{graph-based} books $\cG \x [0,L]$ they exhibit a sharp threshold $L_{min}$ separating transversally trivial from genuinely two dimensional ground states, thereby justifying metric graphs as effective one dimensional models of thin two dimensional networks. Their transition and ours are two facets of the same phenomenon, but in this last setting the two products are transposed: there the graph is the \emph{extended} factor and the transversal direction is a shrinking interval, so that the limiting object of the reduction is the metric graph itself, whereas here the graph is the \emph{compact} factor and the extended directions are those of $\R^N$, so that the limiting object is the NLS equation on $\R^N$, and of the graph only the total length survives. Besides working with energy rather than action ground states, and in arbitrary dimension $N \geq 1$, we face a difficulty which is specific to our setting: the arguments of \cite{fracturedstrip, LeCozShakarov} ultimately rest on the fact that, on an interval with Neumann conditions, $\partial_y u$ satisfies Dirichlet conditions and is therefore controlled by $\partial_{yy} u$ through a Poincaré inequality, whereas on a graph the transversal derivative is not continuous at the vertices and satisfies only the Kirchhoff balance. This is precisely why we first settle the case of a segment and then transfer the result to a general compact graph by means of the $y$-decreasing rearrangement.

\subsection*{The functional setting, and the tools}

A substantial part of this paper is devoted to building the
analytical framework for variational problems on $\R^N \x \cG$, which we believe
to be of independent interest for other nonlinear problems on such domains. In
Section \ref{sez spazio} we introduce the Sobolev space $H^1(\R^N \x \cG)$ in two
natural ways: edgewise, requiring $u_e \in H^1(\R^N \x e)$ for every edge $e$
together with the compatibility of the traces at the vertices, and fiberwise,
requiring $u(\cdot,y) \in H^1(\R^N)$ and $u(x,\cdot) \in H^1(\cG)$ for a.e. $y$
and a.e. $x$; then we prove that they coincide, in Proposition \ref{prop: eq def}:
the first definition makes the vertex conditions transparent, the second one makes it possible
to argue one variable at a time, and both are used constantly in the sequel. Since we will be interested in ground states and local minimizers, we restrict to real valued functions; it is clear that our definitions and results extend to the complex valued case with minor changes. Section \ref{sez riarrangiamenti} introduces three \emph{partial rearrangements}:
the $y$-decreasing and the $y$-symmetric ones, acting on the graph variable and
built upon the rearrangements on metric graphs of \cite{2015}, and the
$x$-symmetric one, which is the Schwarz rearrangement in $x$. In the spirit of
Steiner symmetrization we show that all of them are compatible with
$H^1(\R^N \x \cG)$ (the delicate point being the compatibility of the traces), preserve every $L^p$ norm, and decrease \emph{both} components of the
Dirichlet energy (Propositions \ref{propr y-riarrangiamento decr},
\ref{propr y-riarrangiamento simm} and \ref{propr x-riarrangiamento}); the
Polya-Szego inequalities in the ``transversal" variable rely on a characterization
of $H^1$ in terms of the heat semigroup, in the spirit of
\cite[Theorem 7.10]{Lieb_Loss}, whose analogue on the circle we prove in Lemma
\ref{lemma heat kernel su S1}. As a consequence, minimizing sequences for the NLS energy can always be assumed to be radially decreasing in $x$. Moreover, when $\cG$ can be covered by cycles, they can be compared with their symmetric rearrangement in $y$: this is what lets the topology of $\cG$ enter the picture. Finally, in Section \ref{sez background} we establish the Gagliardo-Nirenberg inequalities on $\R^N \x \cG$, the two relevant critical
exponents being $2_*:=2+4/(N+1)$, the $L^2$-critical exponent of $\R^{N+1}$, and
$2+4/N$, that of $\R^N$ (which is always smaller than the $N+1$ Sobolev critical exponent, 
$2^*:=2+4/(N-1)$). We prove the inequality in a localized form
(Proposition \ref{prop GN localizzata}). This gives a
Gagliardo-Nirenberg counterpart of Lions' vanishing lemma, and it is what rules
out the vanishing of minimizing sequences. We then rewrite the inequality with an equivalent
norm, so that the optimal constant $K_{\R^N \x \cG,p,B}$ (see \eqref{GN con costanti} for the precise definition) becomes invariant under
dilations of $\cG$, and we compare it with those of $\R^{N+1}$, of the half-space
$\R^{N+1}_+$, and of the model products $\R^N \x I$ and $\R^N \x S_1$, $I$ a
segment and $S_1$ a circle (Proposition \ref{stime masse critiche}): a terminal
edge, or a covering by cycles, decides which model constant bounds
$K_{\R^N \x \cG,p,B}$ from above and from below, the exact analogue in our
setting of the topological dichotomy of \cite{threshold}. At the \(L^2\)-critical exponent, we also determine the asymptotic
behavior of the optimal constants on the model products
\(\mathbb R^N\times S_1\) and \(\mathbb R^N\times I\) as
\(B\to+\infty\) (Proposition \ref{p:asymptotic-GN-constants} and Corollary \ref{cor: GN opt crit}). Combined with the previous comparisons, this
identifies the limiting optimal constant on graphs admitting a cycle
covering and on graphs containing a terminal edge.

\subsection*{Existence and properties of ground and bound states}

With these tools at hand, we carry out a rather complete study of ground states and local minimizers when $\cG$ is compact. We describe the results following the three regimes of the exponent $p$. We search for local or global minimizers of $E_p(\cdot, \R^N \x \cG)$ defined in \eqref{NLS energy} in 
\begin{equation*}
    H^1_{\mu}(\R^N \x \cG):=\{u \in  H^1(\R^N \x \cG) \text{ such that } \|u\|_{L^2(\R^N \x \cG)}^2=\mu\};
\end{equation*}
in particular, we investigate the existence of a function $u$ that realizes
\begin{equation*}
    \mathcal{E}_{p}(\mu, \R^N \x \cG):=\inf\limits_{u \in H^1_\mu(\R^N \x \cG)}E_p(u, \R^N \x \cG).
\end{equation*}

\subsubsection*{The subcritical and critical regimes.} The main existence results in this ranges can be summarized as follows.

\begin{thm}\label{caso sottocritico e critico}
 Let $\cG$ be a compact metric graph, $\mu>0$, $N\geq 1$, and $p \in \left(2,2+4/(N+1)\right)$. Then, there exists a ground state of mass $\mu$ for $E_p(\cdot, \R^N \x \cG)$. \\
For the critical case $p=2+4/(N+1)=2_*$, there exist positive thresholds 
\[
\mu_{\R^{N+1}_+,2_*} \le \mu_{\mathrm{ex}}^*
\bigl(\mathbb R^N\times \cG,2_*\bigr) \le \mu_{\R^{N+1},2_*},
\]
defined in terms of optimal Gagliardo-Nirenberg constants (see Remark \ref{g.s. in RN+1} and Eq. \eqref{def thr criti 926} below) such that the following holds:
\begin{itemize}
\item[($i$)] If $\cG$ has a terminal edge, then
\[
\begin{cases}
\mathcal{E}_{2_*}\left(\mu, \R^N \x \cG\right) > -\infty \qquad \text{and is attained if }\mu \in (0,\mu_{\R^{N+1}_+,2_*})\\
\mathcal{E}_{2_*}\left(\mu, \R^N \x \cG\right) = -\infty \qquad \text{if }\mu >\mu_{\R^{N+1}_+,2_*}.
\end{cases}
\]
\item[($ii$)] If $\cG$ has no terminal edge, then
\[
\begin{cases}
\mathcal{E}_{2_*}\left(\mu, \R^N \x \cG\right) > -\infty \qquad \text{and is attained if }\mu \in (0,\mu_{\mathrm{ex}}^*(\mathbb R^N\times \cG,2_*))\\
\mathcal{E}_{2_*}\left(\mu, \R^N \x \cG\right) = -\infty \qquad \text{if }\mu >\mu_{\R^{N+1},2_*}.
\end{cases}
\]
Furthermore, if $\cG$ can be covered by cycles, then $\mu_{\mathrm{ex}}^*(\mathbb R^N\times \cG,2_*)= \mu_{\R^{N+1},2_*}$.
\end{itemize}
\end{thm}

From now on, we often say that 
\begin{equation}\label{(H)}
\text{$\cG$ satisfies condition (H) if $\cG$ can be covered by cycles},
\end{equation}
where a cycle in $\cG$ is a closed or infinite path with almost no self-intersections. This assumption plays a major role in the analysis of the NLS equation on metric graphs, we refer for instance to \cite{2015}. Actually, in \cite{2015} this condition is stated differently; this equivalent formulation can be found in \cite{threshold}.

We point out that, assuming that $\cG$ has a terminal edge or admits a covering by cycles, the threshold $\mu_{\mathrm{ex}}^*(\mathbb R^N\times \cG,2_*)$ is sharp. Nonetheless, also in these cases the existence of ground states in the limit case $\mu=\mu_{ex}^*\bigl(\mathbb R^N\times \cG,2_*\bigr)$ remains open. We also note that there exist compact metric graphs with no terminal edges that do not admit a cycle covering. A simple example is given by two circles connected by a single segment. More generally, this is the case for any compact graph consisting of two compact subgraphs without pendants joined by a single edge. 

The previous theorem establishes that, for $p \in (2, 2_*)$, the energy is bounded from below on every mass sphere, and ground states always exist. At the critical exponent $p=2_*$ the situation changes, and the outcome depends on the mass through the optimal Gagliardo-Nirenberg constant.

The genuinely new phenomenon, in the presence of the graph factor, concerns the dependence of the ground state on $y$. A distinguished family of solutions of \eqref{intro stationary} is given by the \emph{semi-trivial} ones, i.e. those which do not depend on $y$: these are exactly the rescaled solitons $Z_{\mu/l(\cG),p}$ with mass $\mu /l(\cG)$ of the NLS equation on $\R^N$, regarded as functions on $\R^N \x \cG$. We prove that they are the ground states precisely for small masses:
\begin{thm}\label{thm 1.3}
Let $\cG$ be a compact metric graph, $\mu>0$, $N \in  \mathbb{N}\geq 1$, 
$p \in \left(2,2+4/(N+1)\right]$, and let $u_\mu$ be a ground state of 
$\cE_p(\mu, \R^N \times \cG)$ 
(we assume that such a ground state exists, see Theorem 
\ref{caso sottocritico e critico}). Then, there exists a critical mass 
$\mu_2\left(\R^N \x \cG,p\right)>0$ such that	
	\begin{equation}
		\label{tesi thm 1.3}
		\mu < \mu_2\left(\R^N \x \cG,p\right) \implies \partial_y u_\mu \equiv 0, \qquad 
		\mu>\mu_2\left(\R^N \x \cG,p\right) \implies  \partial_y u_\mu \not\equiv 0.
	\end{equation}
\end{thm}

This extends \cite[Theorem 1.3]{TerTzvVis} to the graph setting, and it also covers the critical case $p=2_*$, which is not contained in \cite{TerTzvVis}. The proof follows the strategy of \cite{TerTzvVis} only up to a point: after reducing, by a scaling, to a family of functionals $E_{p,\lambda}$ in which the transversal derivative carries a large weight $\lambda$, we have to show that the transversal derivative of the corresponding ground states does not merely vanish in the limit, but vanishes identically for $\lambda$ large. Here the graph structure requires a genuinely different argument: we first prove the statement when $\cG$ is a segment, exploiting the $W^{2,2}$ regularity of the minimizers and the Neumann expansion in the transversal variable (Lemmas \ref{u in W2,2} and \ref{lemma 3.6}), and then we transfer it to a general compact graph by means of the $y$-decreasing rearrangement of Section \ref{sez riarrangiamenti}.

\subsubsection*{The regime $2_*<p<2+4/N$.} In this range the problem is $L^2$-supercritical for $\R^N \x \cG$ but still subcritical for $\R^N$: global minimizers do not exist for any mass, but the semi-trivial solutions do, and it is natural to look for local minimizers, in the spirit of \cite{PieVerYu} (where the case of the product of $\R^N$ with a compact manifold is investigated). In Section \ref{supercritical case} we prove that local minimizers exist below an explicit threshold:

\begin{thm}
	\label{minimo locale intro}
	Let $\cG$ be a compact metric graph, $N \geq 1$, $p \in \left(2+4/(N+1), 2+4/N\right)$. Then, there exists a critical mass $\mu^*_{ex}(\R^N \x \cG, p)>0$ such that, for every $0<\mu<\mu_{ex}^*(\R^N \x \cG, p)$, the functional $E_p(\cdot, \R^N \x \cG)$ admits a local minimum $u_{\mu}$ in $H^1_{\mu}(\R^N \x \cG)$.
\end{thm}

A more accurate statement (with a quantitative estimate for $\mu^*_{ex}$) is given in Theorem \ref{minimo locale} below.

Our approach here is close to the one of \cite{BelBouJeaVis}, and slightly different from the one of \cite{PieVerYu}: instead of a truncation argument, we isolate a sublevel of the Dirichlet energy in which the energy is bounded from below, and we optimize the resulting threshold with respect to the size of the sublevel. This has the advantage of producing an explicit lower bound for $\mu^*_{ex}$ in terms of the optimal Gagliardo-Nirenberg constant, of $l(\cG)$ and of $N$, and of showing how the threshold scales under dilations of $\cG$. We then show that these local minimizers exhibit the same behaviour as the ground states of the subcritical regime: they are semi-trivial below a critical mass, and genuinely $y$-dependent above it.

\begin{thm}
	\label{thm 1.3 sopracritico}
	Let $\cG$ be a compact metric graph, $N \geq 1$, $2+4/(N+1) < p < 2+4/N$. Then, there exists a critical mass $0<\mu_2\left(\R^N \x \cG,p\right)\leq \mu^*_{ex}\left(\R^N \x \cG,p\right)$ such that:
	\begin{itemize}
		\item[($i$)] if $0<\mu<\mu_2\left(\R^N \x \cG,p\right)$, every local minimizer $u$ found in Theorem \ref{minimo locale intro} satisfies $\|\pa_y u\|_2^2= 0$;
		\item[($ii$)] if $\mu_2(\R^N \x \cG,p)<\mu<\mu^*_{ex}(\R^N \x \cG,p)$, then the local minimizer of mass $\mu$ found in Theorem \ref{minimo locale intro} satisfies $\|\pa_y u\|_2^2>0$.
	\end{itemize}
\end{thm}

\subsubsection*{Local minimality of the semi-trivial solution.} The critical mass $\mu_2$ of Theorems \ref{thm 1.3} and \ref{thm 1.3 sopracritico} is defined by a comparison of energy levels, and is therefore not easy to compute. In Section \ref{sez minimalita} we characterize a second, and more explicit, critical mass, namely the largest mass below which the semi-trivial solution $Z_{\mu/l(\cG),p}$ is a local minimizer. Denoting by $\alpha_1(\cG)$ the first nonzero eigenvalue of the Kirchhoff Laplacian on $\cG$ (see Definition \ref{primo autovettore} below) and $\beta_1(p, N, 1)$ the first eigenvalue of $E_p''(Z_{1,p}, \R^N)$ (see Definition \ref{primo autovalore E''}), we prove:

\begin{thm}\label{thm: loc min semitriv}
		Let $\cG$ be a compact metric graph, $N \geq 1$, $p \in \left(2, 2+4/N\right)$, and let $Z_{\mu/l(\cG), p}$ be the semi-trivial solution. Then, the critical mass
		\[
		\mu_1(\R^N \x \cG,p):=l(\cG)\left(\frac{\alpha_1(\cG)}{E_p'\left(Z_{1, p}, \R^N\right)\left(Z_{1, p}\right)-\beta_1(p, N, 1)}\right)^{\frac{4-N(p-2)}{2(p-2)}}
		\]
		is such that
	\begin{itemize}
		\item[($i$)] if $\mu<\mu_1(\R^N \x \cG,p)$, then $Z_{\mu/l(\cG), p}$ is a local minimizer of $E_p$ in $H^1_{\mu}(\R^N \x \cG)$;
		\item[($ii$)] if $\mu>\mu_1(\R^N \x \cG,p)$, then $Z_{\mu/l(\cG), p}$ is not a local minimizer of $E_p$ in $H^1_{\mu}(\R^N \x \cG)$.
	\end{itemize}
\end{thm}

The interest of this formula is that it decouples the two ingredients: the graph enters only through its total length and through $\alpha_1(\cG)$, while the remaining quantities depend only on the soliton of the NLS equation on $\R^N$. In particular, among graphs of fixed length, the larger $\alpha_1(\cG)$, the wider the range of masses for which the semi-trivial solution is stable, which is the exact analogue of what happens for the constant solution on a compact graph \cite{CacDovSer}.

\subsection*{Organization of the paper and notation}

Sections \ref{sez spazio}--\ref{sez background} contain the functional
setting described above. Section \ref{sez sottocritico} deals with the
subcritical and critical regimes and with the dimensional crossover of the
ground states, Section \ref{supercritical case} with the local minimizers in the
range $(2_*, 2+4/N)$, and Section \ref{sez minimalita} with the local minimality
of the semi-trivial solution.

Throughout the paper, $\cG$ is a connected metric graph with edge set $\mathfrak{E}$ and vertex set $\mathcal{V}$, and $l(\cG)$ denotes its total length; $\cG$ is compact from Section \ref{sez background} on. We write $\|\cdot\|_p$ and $\|\cdot\|_{H^1}$ for the norms of $L^p(\R^N \x \cG)$ and $H^1(\R^N \x \cG)$, specifying the domain whenever it is different, and $\|\nabla_{x,y}u\|_2^2=\|\nabla_xu\|_2^2+\|\partial_yu\|_2^2$. Positive constants denoted by $C$ may change from line to line.

\section{Definition of the function space \texorpdfstring{$H^1(\R^N\x\cG)$}{H1(RN X G)}}\label{sez spazio}

In this section we give two different definitions of the Sobolev space $H^1(\R^N\x\cG)$, and we show their equivalence. We always denote by $x$ the variable in $\R^N$ and by $y$ the variable in $\cG$. Recall that a graph $\cG$ is uniquely identified by the sets of edges and vertices, denoted by $\mathfrak{E}$ and $\mathcal{V}$, respectively. In particular, each edge $e$ is associated with a length $l_e \in (0, +\infty]$. If $l_e$ is finite, $e$ is identified with a segment $[0, l_e]$, and is incident to two vertices in $\mathcal{V}$, which represent its endpoints (these may coincide, in which case $e$ is a self-loop). If $l_e=+\infty$, $e$ is identified with a half-line $[0, +\infty)$, and is incident to a single vertex in $\mathcal{V}$. Following \cite[Section 2]{2015}, a function $v: \mathcal{G} \rightarrow \R$ is identified with the family of functions $\{v_e\}_{e \in \mathfrak{E}}$, where $v_e$ is the restriction of $v$ on the edge $e$. Endowing each edge with the $1$-dimensional Lebesgue measure, for $1 \leq p <+\infty$ we define the $L^p(\cG)$-norm
	\[
	\|v\|_{L^p(\cG)}^p:=\sum\limits_{e \in \mathfrak{E}}\|v_e\|_{L^p(e)}^p,
	\]
	and define $L^p(\cG)$ spaces accordingly. Similarly, we define the $H^1(\cG)$-norm
	\[
	\|v\|_{H^1(\cG)}^2:=\sum\limits_{e \in \mathfrak{E}}\|v_e\|^2_{H^1(e)}=\sum\limits_{e \in \mathfrak{E}}\left(\|v_e\|^2_{L^2(e)}+\|v_e'\|^2_{L^2(e)}\right).
	\]
	By the standard one-dimensional Sobolev embedding, if $v$ has finite $H^1(\cG)$-norm of $v$, then each restriction $v_e$ admits a continuous representative on $e$. It is therefore natural to require, in addition, continuity at the vertices of $\cG$ in the definition of $H^1(\cG)$. More precisely, we say that $v \in H^1(\cG)$ if:
\begin{itemize}
	\item[$(i)$] for each edge $e$, $v_e \in H^1(e)$;
	\item[$(ii)$] $\sum_{e \in \mathfrak{E}} \|v_e\|_{H^1(e)}^2<+\infty$;
	\item [$(iii)$] $v$ is continuous at every vertex.
\end{itemize} 
Similarly, a function $u$ defined on $\R^N \times \cG$ can be identified with a vector of functions $\{u_e\}_{e \in \mathfrak{E}}$, where each $u_e$ is defined on $\R^N \times e$. Endowing $\R^N \times e$ (which can be seen as a subset of $\R^{N+1}$) with the $(N+1)$-dimensional Lebesgue measure, one can define $L^p$ spaces over $\R^N \times \cG$ in a natural way, with norm
\[
\|u\|_{L^p(\R^N \times \cG)}^p := \sum_{e \in \mathfrak{E}} \|u_e\|_{L^p(\R^N \times e)}^p.
\]

\begin{definition}
\label{def Verzini}
We define the Sobolev space $H^1(\R^N \x \cG)$ as the set of functions $u$ with domain $\R^N \x \cG$ such that:
\begin{itemize}
\item[$(i)$] $u \in L^2(\R^N \x \cG)$;
\item[$(ii)$] $u_e \in H^1(\mathbb{R}^N \x e)$ for every $e \in \mathfrak{E}$;
\item[$(iii)$] $\sum_{e\in \mathfrak{E}}\|u_e\|_{H^1(\R^N \x e)}^2 <+\infty$;
\item[$(iv)$] $\forall e_1, e_2 
\in \mathfrak{E}$ meeting at a vertex $v \in \mathcal{V}$, we have
\begin{equation}
\label{trace operator}
T_{\mathbb{R}^N \x \{v\}}(u_{e_1})=T_{\mathbb{R}^N \x \{v\}}(u_{e_2}),
\end{equation}
\end{itemize}
where $T_{\mathbb{R}^N \x \{v\}}:H^1(\mathbb{R}^N \x e) \rightarrow L^2(\R^N \x \{v\}) \simeq L^2(\R^N)$ denotes the trace operator. 

We equip $H^1(\R^N \x \cG)$ with the natural scalar product
\begin{equation}
\label{definition scalar product}
(u,v)_{H^1}=\sum_{e\in\mathfrak{E}}(u_e,v_e)_{H^1(\mathbb{R}^N\x e)}. 
\end{equation}
The induced norm is 
\begin{equation*}
    \|u\|_{H^1(\R^N \x \cG)}^2:=\sum\limits_{e \in \mathfrak{E}}\|u_e\|_{H^1(\R^N \x e)}^2=\|u\|^2_{L^2(\R^N \x \cG)}+\|\nabla_x u\|^2_{L^2(\R^N \x \cG)}+\|\partial_y u\|^2_{L^2(\R^N \x \cG)}.
\end{equation*}
\end{definition}
Observe that the continuity condition at the vertices, required in the definition of $H^1(\cG)$, is replaced with a compatibility of the traces at the vertices.
\begin{remark}
	Through standard arguments, it is straightforward to show that $H^1(\R^N \x \cG)$, endowed with the scalar product defined in \eqref{definition scalar product}, is a Hilbert space.
\end{remark}

Now we provide an alternative definition of $H^1(\R^N \x \cG)$.
\begin{definition}\label{def Soave}
\label{seconda def H1}
We define the Sobolev space $H^1(\R^N \x \cG)$ as the set of functions $u$ defined on $\R^N \x \cG$ such that:
\begin{itemize}
 \item[$(i)$] $u(\cdot, y) \in H^1(\R^N)$ for almost every $y \in \cG$;
\item[$(ii)$] $\int_{\cG} \|u(\cdot, y)\|_{H^1(\R^N)}^2\,dy<+\infty$;
\item[$(iii)$] $u(x, \cdot) \in H^1(\cG)$ for almost every $x \in \R^N$;
\item[$(iv)$] $\int_{\R^N} \|u(x, \cdot)\|_{H^1(\cG)}^2\,dx<+\infty$.
\end{itemize}
\end{definition}
\begin{proposition}\label{prop: eq def}
	Definitions \ref{def Verzini} and \ref{def Soave} for $H^1(\R^N \x \cG)$ are equivalent.
\end{proposition}
\begin{proof}
Assume that $u$ satisfies Definition \ref{def Verzini}. By points $(i)$ and $(ii)$, for every edge $e$ and almost every $y \in e$ we have $u(\cdot, y)=u_e(\cdot,y) \in L^2(\R^N)$ and $\left(\nabla_x u\right)(\cdot,y)=\left(\nabla_x u_e\right)(\cdot,y) \in L^2(\R^N)$; moreover $\nabla_x\left(u(\cdot,y)\right)=\left(\nabla_xu\right)(\cdot,y)$ a.e. on $\R^N$, for a.e. $y$. Hence, by Fubini's Theorem,
	\begin{equation}\label{finitezza norma 2 gradiente x}
		\int_{\cG}\|u(\cdot, y)\|_{H^1(\R^N)}^2\,dy=\|u\|_{L^2(\R^N \x \cG)}^2+\|\nabla_x u\|_{L^2(\R^N \x \cG)}^2<+\infty,
	\end{equation}
	which is points $(i)$ and $(ii)$ of Definition \ref{def Soave}; point $(iv)$ follows in the same way, with the $y$-derivatives in place of the $x$-gradients. As for point $(iii)$, arguing as above in the variable $x$ we get $u_e(x,\cdot) \in H^1(e)$ for every $e \in \mathfrak{E}$ and almost every $x \in \R^N$, so that each $u_e(x,\cdot)$ is continuous, and it only remains to check the continuity at the vertices required in the definition of $H^1(\cG)$ (see \cite{2015}). If $v \in \mathcal{V}$ and $e_1,e_2 \in \mathfrak{E}$ are incident at $v$, then point $(iv)$ of Definition \ref{def Verzini} gives
	\[
	\int_{\R^N}\left(T_{\R^N \x \{v\}}\left(u_{e_1}\right)-T_{\R^N \x \{v\}}\left(u_{e_2}\right)\right)^2\,dx=0,
	\]
	whence $u_{e_1}(x,v)=u_{e_2}(x,v)$ for almost every $x \in \R^N$, and $u(x,\cdot) \in H^1(\cG)$ for almost every $x$.

	Conversely, assume that $u$ satisfies Definition \ref{def Soave}. Point $(i)$ of Definition \ref{def Verzini} follows from $\|u\|^2_{L^2(\R^N \x \cG)}=\int_{\cG}\|u(\cdot,y)\|^2_{L^2(\R^N)}dy \leq \int_{\cG}\|u(\cdot,y)\|^2_{H^1(\R^N)}dy<+\infty$. By points $(ii)$ and $(iv)$ of Definition \ref{def Soave}, both $\nabla_xu_e$ and $\partial_yu_e$ belong to $L^2(\R^N \x e)$ for every $e$, and, reading \eqref{finitezza norma 2 gradiente x} backwards,
	\[
	\sum\limits_{e \in \mathfrak{E}}\|u_e\|^2_{H^1(\R^N \x e)}=\int_{\cG}\|u(\cdot,y)\|^2_{H^1(\R^N)}\,dy+\int_{\R^N}\|\partial_yu(x,\cdot)\|^2_{L^2(\cG)}\,dx<+\infty,
	\]
	which gives points $(ii)$ and $(iii)$. Finally, if $v \in \mathcal{V}$ and $e_1,e_2$ are incident at $v$, the continuity of $u(x,\cdot)$ for almost every $x$ yields
	\[
	\int_{\R^N}\left(T_{\R^N \x \{v\}}\left(u_{e_1}\right)-T_{\R^N \x \{v\}}\left(u_{e_2}\right)\right)^2\,dx=\int_{\R^N}\left(u_{e_1}(x,v)-u_{e_2}(x,v)\right)^2\,dx=0,
	\]
	that is, point $(iv)$.
\end{proof}

\begin{remark}\label{rmk open books}
	The space $H^1(\R^N \x \cG)$ is the form domain of the Kirchhoff realization of the Laplacian on the open book $\R^N \x \cG$, studied in \cite{openbook}; for $N=1$, Definition \ref{def Verzini} coincides with the space $H^1_D$ of \cite{LeCozShakarov}. In both references the Sobolev spaces are defined edgewise, with no compatibility at the bindings, which is imposed afterwards as a restriction, whereas here the compatibility of the traces is part of the definition and the balance of the fluxes is encoded in the weak formulation. The equivalence with the fiberwise Definition \ref{def Soave} seems instead to have no counterpart in that literature, and it is what allows us to argue one variable at a time in the sequel.
\end{remark}

\section{Partial rearrangements in \texorpdfstring{$H^1(\R^N\x\cG)$}{H1(RN X G)}}\label{sez riarrangiamenti}

The rearrangements introduced below are partial rearrangements: they act on one factor of the product $\R^N \times \cG$, leaving the other variable fixed. They are inspired by classical Steiner/Schwarz symmetrizations, but the graph factor requires using the monotone or symmetric rearrangements on metric graphs developed in \cite{2015}. The main point of this section is to verify that these operations are compatible with the product Sobolev space $H^1(\R^N \times \cG)$, including the trace conditions at the vertices, and that they do not increase either component of the Dirichlet energy.

Firstly, we define the $y$-decreasing rearrangement of a real-valued function $u \in H^1(\R^N \x \cG)$, and we prove its fundamental properties: the equimeasurability and the Polya-Szego type inequalities.

\begin{definition}
\label{def y decreasing rearr}
Let $u$ be a nonnegative function in $H^1(\R^N\x\cG)$. Letting $l(\cG)$ be the total length of $\cG$, consider 
\[
\cG^*:=
\begin{cases}
    [0, l(\cG)] \qquad \text{ if } l(\cG)<+\infty; \\
    [0, +\infty) \qquad \text{ if }l(\cG)=+\infty.
\end{cases} 
\]
We define the $y$-decreasing rearrangement of $u$ as the function $u^{*y}: \R^N \x \cG^* \rightarrow \R$ such that:
\begin{equation*}
		\text{for }y \in \cG^*,  x \in \R^N \text{ such that } u(x, \cdot) \in H^1(\cG), \qquad u^{*y}(x, y):= (u(x, \cdot))^{*}(y),
\end{equation*}
where $f^* \in L^2(\cG^*)$ is the decreasing rearrangement of the function $f \in H^1(\cG)$, as defined in \cite[Section 3]{2015}.
\end{definition}
\begin{remark}
\label{remark sulla misurabilit'a}
The function $u^{*y}$ is measurable. Indeed, let $\rho(x,t):=\mathcal{L}^1\left(\{y \in \cG: \, u(x,y)>t\}\right)$, which is measurable in $x$ for every fixed $t$ by Tonelli's theorem, and nonincreasing in $t$ for every fixed $x$. By the definition of the decreasing rearrangement given in \cite{2015}, and by the monotonicity of $\rho(x,\cdot)$, which allows one to restrict to rational levels,
	\[
	\left\{(x,y) \in \R^N \x \cG^*: \  u^{*y}(x,y) < t\right\}=\left\{(x,y): \exists\, 0 \leq s<t \, : \, \rho(x,s)\leq y\right\}=\bigcup\limits_{\substack{0 \leq s<t \\ s \in \mathbb{Q}}}\left\{(x,y): \rho(x,s)\leq y\right\},
	\]
	which is a countable union of measurable sets.

\end{remark}
The following lemmas are fundamental to proving the main properties of the $y$-decreasing rearrangement. We first state the Hardy--Littlewood inequality for rearrangements on graphs. We omit the proof, which follows step by step the arguments presented in \cite[Theorem 3.4]{Lieb_Loss}.
\begin{lemma}
\label{Hardy-Littlewood lemma}
Let $f,g$ be nonnegative functions in $H^1(\cG)$, and $f^*, g^* \in H^1(\cG^*)$ be the decreasing rearrangements of $f, g$. Then
\begin{equation}
\label{teorema base dei riarrangiamenti}
\int_{\cG}f(y)g(y)\,dy \leq \int_{\cG^*}f^*(y)g^*(y)\,dy   \qquad  \forall f,g \in H^1(\cG).
\end{equation}
\end{lemma}
Secondly, we extend \cite[Theorem 7.10]{Lieb_Loss} to our setting. \begin{definition}
For any $v \in L^2(\R^N)$, $t>0$, let
\[
I_t^{\R^N}(v):=\frac{1}{t}\left(\|v\|^2_{L^2(\R^N)}-\iint_{\R^{2N}}K_t(x_1-x_2)v(x_1)v(x_2)\, dx_1\, dx_2\right),
\] 
where $K_t(x)$ denotes the heat kernel:
\[
K_t(x):=\frac{1}{\left(4\pi t\right)^{\frac{N}{2}}}e^{-\frac{|x|^2}{4t}}.
\]
This functional is defined in \cite[Section 7.10]{Lieb_Loss}. Now, let $I_t: L^2(\R^N \x \cG) \longrightarrow \R$ be such that, for every $u \in L^2(\R^N \x \cG)$, $t>0$:
\[
\begin{split}
I_t(u)&:= \frac{1}{t}\left(\int_{\mathcal{G}}\int_{\R^N} |u(x,y)|^2\,dx\,dy - \int_{\cG}\iint_{\R^{2N}}K_t(x_1-x_2)u(x_1,y)u(x_2,y)\, dx_1\, dx_2\,dy\right) 
=\int_{\cG}I_t^{\R^N}(u(\cdot, y))\, dy.
\end{split}
\]
\end{definition}
\begin{remark}\label{remark su It}
	For each $v \in L^2(\R^N)$ fixed, $I_t^{\R^N}(v)$ is nonnegative (actually positive, unless $v \equiv 0$) and decreasing in $t$, since it can be expressed as
	\[
	I_t^{\R^N}(v)=\int_{\R^N}\left(\frac{1-e^{-4\pi^2|k|^2t}}{t}\left|\mathcal{F}v(k)\right|^2\right)\,dk
	\]
	where $\mathcal{F} v$ is the Fourier transform of $v$ (see \cite[7.10, equation (4)]{Lieb_Loss}). This implies that, for every $u \in L^2(\R^N \x \cG)$, the function $t \mapsto I_t(u)$ is positive and decreasing for $t \in (0, +\infty)$.
\end{remark}
\begin{lemma}
\label{lem 7.10}
Let $u \in L^2(\R^N \x \cG)$. Then, 
\[
u(\cdot, y) \in H^1(\R^N) \text{ for almost every } y \in \cG \qquad \text{and} \qquad \int_{\cG}\|\nabla_x u(\cdot, y)\|_{L^2(\R^N)}^2\,dy<+\infty
\]
if and only if $I_t(u)$ is uniformly bounded in $t \in (0, +\infty)$. In this case, $\nabla_xu \in L^2$ (the gradient is understood in the weak sense), and
\[
 \lim\limits_{t \rightarrow 0^+}  I_t(u)=\|\nabla_x u\|_{L^2(\R^N\x\cG)}^2.
\]
\end{lemma}
\begin{proof}
At first, we prove the direct implication. Since $u(\cdot, y) \in H^1(\R^N)$ for almost every $y \in \cG$, we can apply \cite[Theorem 7.10]{Lieb_Loss} and deduce that $I^{\R^N}_t(u(\cdot, y))$ is uniformly bounded in $t \in (0, +\infty)$. Moreover,
\[
\sup\limits_{t>0}I^{\R^N}_t(u(\cdot, y))=\lim\limits_{t \rightarrow 0^+} I^{\R^N}_t(u(\cdot, y))= \|\nabla_x u(\cdot, y)\|_{L^2(\R^N)}^2.
\]
Then, for every $t \in (0, +\infty)$,
\[
I_t(u)=\int_{\cG}I^{\R^N}_t(u(\cdot, y))\,dy\leq \int_{\cG}\|\nabla_x u(\cdot, y)\|_{L^2(\R^N)}^2\,dy=\|\nabla_x u\|_{L^2(\R^N \x \cG)}^2,
\]
hence $I_t$ is uniformly bounded in $t \in (0, +\infty)$.
Moreover, by the Monotone Convergence Theorem,
\[
\|\nabla_x u\|_{L^2(\R^N \x \cG)}^2=\int_{\cG}\|\nabla_x u(\cdot, y)\|_{L^2(\R^N)}^2\,dy=\int_{\cG}\lim\limits_{t\rightarrow 0}I_t^{\R^N}(u(\cdot, y))\,dy=\lim\limits_{t\rightarrow 0}\int_{\cG}I_t^{\R^N}(u(\cdot, y))\,dy=\lim\limits_{t\rightarrow 0}I_t(u),
\]
which is the desired result. \\
Now, let us manage the opposite implication. For almost every $y \in \cG$, the function $I_t^{\R^N}(u(\cdot, y))$ is positive and decreasing in $t$ by Remark \ref{remark su It}. By the Monotone Convergence Theorem,
\[
\begin{split}
\int_{\cG}\sup\limits_{t \in (0, +\infty)}I_t^{\R^N}(u(\cdot, y))\,dy &= \int_{\cG}\lim\limits_{t \rightarrow 0} I_t^{\R^N}(u(\cdot, y))\,dy=\lim\limits_{t \rightarrow 0}\int_{\cG}I_t^{\R^N}(u(\cdot, y))\,dy\\
&= \lim\limits_{t \rightarrow 0}I_t(u) =\sup\limits_{t \in (0, +\infty)}I_t(u)\leq C, 
\end{split}
\]
for some $C>0$, since $I_t(u)$ is uniformly bounded. Thus,
\[
\text{for almost every } y \in \cG, \qquad\sup\limits_{t \in (0, +\infty)}I_t^{\R^N}(u(\cdot, y)) <+\infty.
\]
Then, by \cite[Theorem 7.10]{Lieb_Loss}, for almost every $y \in \cG$,
\[
u(\cdot, y) \in H^1(\R^N), \text{ and }\|\nabla_x u(\cdot, y)\|_{L^2(\R^N)}^2=\lim\limits_{t \rightarrow 0} I_t^{\R^N}(u(\cdot, y)).
\]
Integrating over $\cG$, and extracting the limit by the Monotone Convergence Theorem:
\[
\|\nabla_x u\|_{L^2(\R^N \x \cG)}^2=\int_{\cG}\|\nabla_x u(\cdot, y)\|_{L^2(\R^N)}^2\,dy=\int_{\cG}\lim\limits_{t \rightarrow 0}I_t^{\R^N}(u(\cdot, y))\, dy=\lim\limits_{t \rightarrow 0}\int_{\cG}I_t^{\R^N}u(\cdot, y))\, dy=\lim\limits_{t \rightarrow 0}I_t(u),
\]
which is the desired result.
\end{proof}

\begin{proposition}
\label{propr y-riarrangiamento decr}
Let $u \in H^1(\R^N\x\cG)$ be a nonnegative function. Then, $u^{*y} \in H^1(\R^N \x \cG^*)$, and:
\begin{itemize}
    \item $\|u^{*y}\|_{L^p(\R^N\x\cG^*)}^p=\|u\|_{L^p(\R^N\x\cG)}^p$, for every $p \geq 1$;
    \item $\|\partial_y u^{*y}\|_{L^2(\R^N\x\cG^*)}^2 \leq \|\partial_y u\|_{L^2(\R^N\x\cG)}^2$;
    \item $\|\nabla_x u^{*y}\|_{L^2(\R^N\x\cG^*)}^2 \leq \|\nabla_x u\|_{L^2(\R^N\x\cG)}^2$.
\end{itemize}
\end{proposition}
\begin{proof}
We prove the first point. For every $p \geq 1$,
\begin{equation*}
\begin{split}
    \|u^{*y}\|_{L^p(\R^N\x\cG^*)}^p &= \int_{\R^N} \|u^{*y}(x, \cdot)\|_{L^p(\cG^*)}^p\, dx = \int_{\R^N} \|(u(x, \cdot))^*\|_{L^p(\cG^*)}^p\, dx\\
    &=\int_{\R^N} \|u(x, \cdot)\|_{L^p(\cG)}^p\, dx= \|u\|_{L^p(\R^N \x \cG)}^p.
\end{split}
\end{equation*}
We used the fact that the $p$-norms of functions in $H^1(\mathcal{G})$ are preserved, when passing to the decreasing rearrangement, see \cite[Section 3]{2015}. Concerning the second point, we assert that  
\[
\text{for almost every } x \in \R^N, \qquad u^{*y}(x, \cdot) \in H^1(\cG^*),
\]
since the decreasing rearrangement of a function in $H^1(\cG)$ is in $H^1(\cG^*)$
(recall that $u(x, \cdot) \in H^1(\cG)$ for almost every $x \in \R^N$ by Definition \ref{seconda def H1}). Moreover, 
\begin{equation}
\label{calcoli polya-szego y-rearr}
\begin{split}
\|\partial_y u^{*y}\|_{L^2(\R^N\x\cG^*)}^2&=\int_{\R^N}\|\partial_y u^{*y}(x, \cdot)\|_{L^2(\cG^*)}^2\, dx=\int_{\R^N}\|\partial_y \left(u(x, \cdot)\right)^*\|_{L^2(\cG^*)}^2\, dx \\
& \leq \int_{\R^N}\|\partial_y u(x, \cdot)\|_{L^2(\cG)}^2\, dx=\|\partial_y u\|_{L^2(\R^N\x\cG)}^2.
\end{split}
\end{equation}
The inequality is due to the fact that the Dirichlet integral of a function in $H^1(\mathcal{G})$ decreases when passing to the decreasing rearrangement, as stated in \cite[Section 3]{2015}. \\
The proof of the last inequality is more involved, and relies on the proof of \cite[Lemma 7.17]{Lieb_Loss}. Combining Lemma \ref{lem 7.10} and Fubini's Theorem:
\begin{equation}
\label{disug polya szego}
\begin{split}
&\|\nabla_xu\|_{L^2(\R^N \x \cG)}^2= \lim\limits_{t \rightarrow 0}I_t(u)\\
&= \lim\limits_{t\rightarrow 0}\frac{1}{t} \left( \|u\|_{L^2(\R^N \times \cG)}^2 -\iint_{\R^{2N}}\left(\int_{\cG}u(x_1,y)u(x_2,y)\,dy\right)K_t(x_1-x_2)\, dx_1\, dx_2\right) \\
 &\geq\lim\limits_{t\rightarrow 0}\frac{1}{t} \left(\|u^{*y}\|_{L^2(\R^N \times \cG^*)}^2-\iint_{\R^{2N}}\left(\int_{\cG^*}u^{*y}(x_1,y)u^{*y}(x_2,y)\,dy\right)K_t(x_1-x_2)\, dx_1\, dx_2\right) \\
 &= \lim\limits_{t\rightarrow 0}I_t(u^{*y}).
\end{split}
\end{equation}
The inequality results from applying \eqref{teorema base dei riarrangiamenti} on $u(x_1, \cdot)$, $u(x_2, \cdot)$ (which belong to $H^1(\cG)$ for almost every $x_1, x_2 \in \R^N$), combined with the fact that the $2$-norm of $u$ is maintained when passing to $u^{*y}$. This estimate and the monotonicity of $t \mapsto I_t(u^{*y})$ show that $I_t(u^{*y})$ is uniformly bounded for $t \in (0, +\infty)$. Thus, by Lemma \ref{lem 7.10}, 
\[
u^{*y}(\cdot, y) \in H^1(\R^N) \text{ for a.e. } y \in \cG^*, \,  \qquad \text{and}\qquad \lim\limits_{t \rightarrow 0^+} I_t(u^{*y})=\|\nabla_x u^{*y}\|_{L^2(\R^N \times \cG^*)}^2.
\]
This, combined with \eqref{disug polya szego}, gives the third point of the thesis. Finally, by Definition \ref{seconda def H1}, $u^{*y} \in H^1(\R^N \x\cG^*)$.
\end{proof}
Similarly, we are able to define the $y$-symmetric rearrangement of a function $u \in H^1(\R^N\x\cG)$.
\begin{remark}
	\label{circle}
	We denote by $S_1:=\partial B_1(0) \subset \R^2$ the circle of radius $1$, while $S_{\rho}$ is the circle of radius $\rho$ (notice that they can be seen as metric graphs). Often, it is convenient to identify $S_{\rho}$ with the segment $\left[-\pi \rho, \pi \rho\right]$, and establish that $f \in H^1(S_{\rho})$ if $f \in H^1\left(\left[-\pi \rho, \pi \rho\right]\right)$ and $f\left(-\pi \rho\right)=f\left(\pi \rho\right)$.
\end{remark}

\begin{definition}
	\label{def y symmetric rearr}
Let $u \in H^1(\R^N \x \cG)$ be a nonnegative function, and let 
\[
\widehat{\cG}:=
   \begin{cases}
      S_{\frac{l(\cG)}{2\pi}}\qquad &\text{ if } l(\cG)<+\infty;\\
        \mathbb{R} \qquad &\text{ if } l(\cG)=+\infty.
   \end{cases}
\]
The $y$-symmetric rearrangement of $u$ is 
the function $\widehat{u}^y: \R^N \x \widehat{\cG} \rightarrow \R$ defined by
\[
   \text{if } x \in \R^N, y \in \widehat{\cG} \text{ and } u(x, \cdot) \in H^1(\cG), \qquad    \widehat{u}^y(x, y):= \widehat{(u(x, \cdot))}(y),
\]
where $\widehat{f} \in H^1(\widehat{\cG})$ is the symmetric rearrangement of the function $f \in H^1(\cG)$ (the definition is given in \cite[Section 3]{2015}).
\end{definition}
\begin{remark}
	\label{remark symm rearr}
For the same arguments of Remark \ref{remark sulla misurabilit'a}, the function $\widehat{u}^y$ is measurable. Moreover, the Hardy-Littlewood inequality remains valid for symmetric rearrangements:
\[
\forall f,g \in H^1(\cG), \qquad \int_{\cG}f(y)g(y)\,dy \leq \int_{\widehat{\cG}}\widehat{f}(y)\widehat{g}(y)\,dy.
\]
The proof of this fact is identical to that of \cite[Theorem 3.4]{Lieb_Loss}, and is therefore omitted.
\end{remark}
\begin{proposition}
\label{propr y-riarrangiamento simm}
Let $u \in H^1(\R^N \x \cG)$ be a nonnegative function. Then, $\widehat{u}^y \in H^1(\R^N \x \widehat{\cG})$, and
\begin{itemize}
\item[($i$)] $\|\widehat{u}^{y}\|_{L^p(\R^N\x\widehat{\cG})}^p=\|u\|_{L^p(\R^N\x\cG)}^p$, for every $p \geq 1$;
\item[($ii$)] $\|\nabla_x \widehat{u}^{y}\|_{L^2(\R^N\x\widehat{\cG})}^2 \leq \|\nabla_x u\|_{L^2(\R^N\x\cG)}^2$;
\item[($iii$)] $\|\partial_y \widehat{u}^{y}\|_{L^2(\R^N\x\widehat{\cG})}^2 \leq \|\partial_y u\|_{L^2(\R^N\x\cG)}^2$, if, for almost every $x \in \R^N$, $u(x, \cdot)$ attains at least twice almost every value in its image.
\end{itemize}
\end{proposition}
\begin{proof}
The first two points can be proved by the same arguments presented in Proposition \ref{propr y-riarrangiamento decr}, without modifications. In particular, the second point relies on the Hardy-Littlewood inequality, which is discussed in Remark \ref{remark symm rearr}. The last point follows from the fact that if $f \in H^1(\cG)$ reaches at least twice almost every value in its image, then
\[
\widehat{f} \in H^1(\widehat{\cG}) \qquad \text{ and }\qquad  \|(\widehat{f})'\|_{L^2(\widehat{\cG})}^2\leq \left\|f'\right\|_{L^2(\cG)}^2,
\]
see \cite[Proposition 3.1]{2015}. Then, replacing the $y$-decreasing rearrangement with the $y$-symmetric rearrangement in \eqref{calcoli polya-szego y-rearr}, we obtain the third point of the thesis.
\end{proof}

\begin{remark}
\label{remark sulla condiz (H)}
As observed in the proof of \cite[Theorem 2.3]{2015}, the additional condition of the third point of Proposition \ref{propr y-riarrangiamento simm} is ensured when $\cG$ can be covered by cycles.
\end{remark}

Now we introduce the $x$-symmetric rearrangement. Thanks to its properties, we will seek global and local minimizers which are radial and decreasing in the $x$-variable.

\begin{definition}
\label{def x riarr}
Let $u \in H^1(\R^N\x\cG)$ be a nonnegative function. Then its $x$-symmetric rearrangement is $u^{*x}:\R^N \x \cG \rightarrow \R$ defined by
\begin{equation*}
    \text{for } x \in \R^N, y \in \cG \text{ such that } u(\cdot, y) \in H^1(\R^N), \qquad u^{*x}(x,y):=u\left((\cdot, y)\right)^{*}(x),
\end{equation*}
where $f^* \in H^1(\R^N)$ denotes the symmetric rearrangement of $f \in H^1(\R^N)$, as defined in \cite[Section 3.3]{Lieb_Loss}. 
\end{definition}

\begin{remark}
	Arguing similarly to Remark \ref{remark sulla misurabilit'a}, $u^{*x}$ is measurable on $\R^N \x \cG$. Moreover,
	\[
	u^{*x}(\cdot, y) \in H^1(\R^N) \text{ for almost every } y \in \cG.
	\]
\end{remark}

The following lemma is the generalization of \cite[Theorem 7.10]{Lieb_Loss} for functions defined on the circle $S_1$ instead of $\mathbb{R}$: it will be used to prove the fundamental properties of the $x$-symmetric rearrangement.
\begin{lemma}
\label{lemma heat kernel su S1}
Let $f$ be a function in $L^2(S_1)$, and let
\[
K_t^{S_1}(y):=\frac{1}{2\pi}\sum\limits_{k=-\infty}^{+\infty}e^{-k^2t}e^{iky}=\frac{1}{2\pi}+\frac{1}{\pi}\sum\limits_{k=1}^{+\infty}e^{-k^2t}cos(ky).
\]
be the heat kernel of the unit circle (see, e.g., \cite[Exercise 10.18]{Heat_kernels}). Then, $f \in H^1\left(S_1\right)$ if and only if 
\[
J_t^{S_1}(f):=\frac{1}{t}\left(\|f\|^2_{L^2(S_1)}-\int_{-\pi}^{\pi}\int_{-\pi}^{\pi}K^{S_1}_t(y_1-y_2)f(y_1)f(y_2)\, dy_1\, dy_2\right),
\]
is uniformly bounded in $t$. In this case, $J_t^{S_1}(f)$ is decreasing in $t$, and
\begin{equation*}
    \|f'\|_{L^2(S_1)}^2=\lim\limits_{t \rightarrow 0^+} J^{S_1}_t(f).
\end{equation*}
\end{lemma}
\begin{proof}
We retrace the proof of \cite[Theorem 7.10]{Lieb_Loss}. Note that 
\[
f(t,y):=\int_{-\pi}^{\pi}K^{S_1}_t(s-y)f(s)\, ds
\] 
solves the heat equation with initial datum $f(0,y)=f(y)$. Let $f_k$ be the $k$-th Fourier coefficient of $f$:
\[
f_k=(f, e_{k})_{L^2(S_1)}, \quad \text{where} \quad e_k:= \frac{e^{iky}}{\sqrt{2\pi}} \qquad k \in \mathbb{Z},
\]
being $(f,g)_{L^2(S_1)}:=\int_{-\pi}^{\pi}f(y)\overline{g(y)}\,dy$ the scalar product in $L^2(S_1)$. Then, 
\[
\forall k \in \mathbb{Z}, \qquad f(t, \cdot)_k=e^{-k^2t}f_k.
\]
Using Plancherel's identity, $J_t^{S_1}(f)$ can be written as
\begin{equation*}
\begin{split}
    J_t^{S_1}(f)&=\frac{1}{t}\left((f,f)_{L^2(S_1)}-(f(t, \cdot),f)_{L^2(S_1)}\right) =\frac{1}{t}\left(\sum\limits_{k=-\infty}^{+\infty}|f_k|^2-\sum\limits_{k=-\infty}^{+\infty}e^{-k^2t}|f_k|^2\right)\\
    &=\sum\limits_{k=-\infty}^{+\infty}\left(\frac{1-e^{-k^2t}}{t}\right)|f_k|^2 \leq \sum\limits_{k=-\infty}^{+\infty} k^2|f_k|^2.
\end{split}
\end{equation*}
By the Monotone Convergence Theorem,
\begin{equation*}
    \lim\limits_{t \rightarrow 0^+}J_t^{S_1}(f)=\sum\limits_{k=-\infty}^{+\infty}\lim\limits_{t \rightarrow 0^+}\left(\frac{1-e^{-k^2t}}{t}\right)|f_k|^2=\sum\limits_{k=-\infty}^{+\infty}k^2|f_k|^2=\|f'\|_{L^2(S_1)}^2.
\end{equation*}
This implies that $f' \in L^2(S_1)$ if and only if $J_t^{S_1}(f)$ is uniformly bounded in $t$, completing the proof.
\end{proof}

This allows us to prove the counterpart of Lemma \ref{lem 7.10}.
\begin{lemma}
\label{lem 7.10 cerchio}
Let $u \in L^2(\R^N \x S_1)$. Then, 
\[
u(x,\cdot) \in H^1(S_1) \text{ for almost every } x \in \R^N \qquad \text{ and } \qquad \int_{\R^N}\|\partial_y u(x, \cdot)\|_{L^2(S_1)}^2\, dx < +\infty
\]
if and only if 
\[
J_t(u):= \frac{1}{t}\left(\int_{S_1}\int_{\R^N} |u(x,y)|^2\,dx\,dy - \int_{\R^N}\int_{S_1}\int_{S_1}K^{S_1}_t(y_1-y_2)u(x,y_1)u(x,y_2)\, dy_1\, dy_2\,dx\right)
\]
is uniformly bounded in $t \in (0, +\infty)$. If this is the case, $J_t(u)$ is positive and decreasing in $t$, the weak derivative $\partial_y u$ is a function in $L^2$, and
\[
 \lim\limits_{t \rightarrow 0} J_t(u)=\|\partial_y u\|_{L^2\left(\R^N \x S_1\right)}^2
\]
\end{lemma}
The proof is identical to that of Lemma \ref{lem 7.10} (thanks to Lemma \ref{lemma heat kernel su S1}), and is therefore omitted.

\begin{proposition}
\label{propr x-riarrangiamento}
Let $u \in H^1(\R^N_x\x\cG_y)$ be nonnegative. Then, the $x$-symmetric rearrangement is a function in $H^1(\R^N \x \cG)$ such that:
\begin{itemize}     
\item[($i$)] $\|u^{*x}\|_{L^p(\R^N \x \cG)}^p=\|u\|_{L^p(\R^N \x \cG)}^p$, for every $p \geq 1$;
\item[($ii$)] $\|\nabla_x u^{*x}\|_{L^2(\R^N \x \cG)}^2 \leq \|\nabla_x u\|_{L^2(\R^N \x \cG)}^2$;
\item[($iii$)] $\|\partial_y u^{*x}\|_{L^2(\R^N \x \cG)}^2 \leq \|\partial_y u\|_{L^2(\R^N \x \cG)}^2$.
\end{itemize}
\end{proposition}
\begin{proof}
The proof of the first two points is analogous to that of the first two points of Proposition \ref{propr y-riarrangiamento decr}, and is omitted: it is based on the fact that the $p$-norms are preserved when passing to the $x$-rearrangement (see \cite[Section 3.3]{Lieb_Loss}), while the Dirichlet integral decreases (see \cite[Lemma 7.17]{Lieb_Loss}). We give the details of the proof of the third point. Notice that
	\begin{equation}
		\label{decomposizione norma H1}
		\|\partial_y u\|^2_{L^2(\R^N\x\cG)}=\int_{\R^N}\sum\limits_{e \in \mathfrak{E}}\|\partial_yu_e(x, \cdot)\|_{L^2(e)}^2\,dx= \sum\limits_{e \in \mathfrak{E}}\int_{\R^N}\|\partial_yu_e(x, \cdot)\|_{L^2(e)}^2\,dx,
	\end{equation}
	so that it suffices to prove that, for every $e \in \mathfrak{E}$,
	\begin{equation}
		\label{prima cosa da dimostrare in x-riarr}
		u_e^{*x}(x, \cdot) \in H^1(e) \quad \text{for almost every } x \in \R^N,
	\end{equation}
	and
	\begin{equation}
		\label{seconda cosa da dimostrare in x-riarr}
		\int_{\R^N}\|\partial_yu_e(x, \cdot)\|_{L^2(e)}^2\,dx\geq \int_{\R^N}\|\partial_yu_e^{*x}(x, \cdot)\|_{L^2(e)}^2\,dx.
	\end{equation}
	Fix $e \in \mathfrak{E}$. If $e$ is a half-line we decompose it into countably many intervals of finite length and argue on each of them separately, so that we may assume $l(e)<+\infty$ and, after a trivial rescaling, $l(e)=\pi$, identifying $e$ with $[0,\pi]$. Let $\tilde{u}_e \in H^1(\R^N\x S_1)$ be the symmetric extension of $u_e$, that is $\tilde u_e(x,y):=u_e(x,|y|)$ for $y \in [-\pi,\pi]$ (see Remark \ref{circle}); since $u_e \in H^1(\R^N \x e)$ and $\tilde u_e$ is even in $y$, Lemma \ref{lem 7.10 cerchio} applies to $\tilde u_e$ and gives
	\[
	\|\partial_yu_e\|_{L^2(\R^N \x e)}^2=\frac{1}{2}\|\partial_y\tilde{u}_e\|_{L^2(\R^N \x S_1)}^2=\lim\limits_{t \rightarrow 0}\frac{1}{2}J_t(\tilde{u}_e).
	\]
	Now, arguing exactly as in \eqref{disug polya szego}, with $J_t$, $K^{S_1}_t$ and the roles of the two variables interchanged --- that is, using Fubini's Theorem to integrate first in $x$, the invariance of the $2$-norm and the Hardy--Littlewood inequality \cite[Theorem 3.4]{Lieb_Loss} in the variable $x$, and the positivity of $K^{S_1}_t$ --- we obtain
	\[
	\|\partial_yu_e\|_{L^2(\R^N \x e)}^2 \geq \lim\limits_{t \rightarrow 0}\frac{1}{2}J_t(\tilde{u}_e^{*x}).
	\]
	In particular $J_t(\tilde{u}_e^{*x})$ is uniformly bounded in $t$, so that, by Lemma \ref{lem 7.10 cerchio}, $\tilde{u}^{*x}_e(x, \cdot) \in H^1(S_1)$ for almost every $x \in \R^N$, which is \eqref{prima cosa da dimostrare in x-riarr} because $\tilde{u}^{*x}_e$ is even in $y$, and $\partial_y \tilde{u}_e^{*x} \in L^2(\R^N \x S_1)$ with
	\[
	\|\partial_y u_e^{*x}\|^2_{L^2(\R^N \x e)}=\frac{1}{2}\|\partial_y \tilde{u}^{*x}_e \|_{L^2(\R^N \x S_1)}^2 =\lim\limits_{t \rightarrow 0}\frac{1}{2}J_t(\tilde{u}_e^{*x})\leq \|\partial_yu_e\|_{L^2(\R^N\x e)}^2,
	\]
	that is \eqref{seconda cosa da dimostrare in x-riarr}. Together with \eqref{decomposizione norma H1}, this proves the third point.

	Finally, $u^{*x} \in H^1(\R^N \x \cG)$ by Definition \ref{def Verzini}. Point $(i)$ follows from $\|u^{*x}\|_2^2=\|u\|_2^2$. As for points $(ii)$ and $(iii)$, by \eqref{prima cosa da dimostrare in x-riarr}, \eqref{seconda cosa da dimostrare in x-riarr} and by the same conclusions for the weak gradient $\nabla_x u_e^{*x}$,
	\[
	\forall e \in \mathfrak{E}, \ u_e^{*x} \in H^1(\R^N \x e) \qquad \text{ and } \qquad \sum\limits_{e\in \mathfrak{E}}\|u_e^{*x}\|_{H^1(\R^N \x e)}^2 \leq\sum\limits_{e\in \mathfrak{E}}\|u_e\|_{H^1(\R^N \x e)}^2<+\infty.
	\]
	Concerning point $(iv)$, let $v \in \mathcal{V}$ and let $e$ be an edge containing $v$. Since $u_e \in H^1(\R^N \x e)$, the map $y \mapsto u_e(\cdot\,,y)$ belongs to $C\big(\overline{e}; L^2(\R^N)\big)$, so that $u_e(\cdot\,,y) \to T_{\R^N \x \{v\}}(u_e)$ in $L^2(\R^N)$ as $y \to v$; by the previous point the same holds for $u_e^{*x}$. Hence, by the nonexpansivity of the Schwarz rearrangement \cite[Theorem 3.5]{Lieb_Loss}, for any $\{y_n\} \subset e$ with $y_n \to v$,
	\[
	\begin{split}
	\big\|u_e^{*x}(\cdot\,,y_n)-\big(T_{\R^N \x \{v\}}(u_e)\big)^{*}\big\|_{L^2(\R^N)} &=\big\|u_e(\cdot\,,y_n)^{*}-\big(T_{\R^N \x \{v\}}(u_e)\big)^{*}\big\|_{L^2(\R^N)}\\
	& \leq \big\|u_e(\cdot\,,y_n)-T_{\R^N \x \{v\}}(u_e)\big\|_{L^2(\R^N)} \rightarrow 0,
	\end{split}
	\]
	whence $T_{\R^N \x \{v\}}\big(u_e^{*x}\big)=\big(T_{\R^N \x \{v\}}(u_e)\big)^{*}$. As $T_{\R^N \x \{v\}}(u_{e_1})=T_{\R^N \x \{v\}}(u_{e_2})$ whenever $e_1,e_2$ are incident at $v$, this gives point $(iv)$.
\end{proof}
\begin{remark}
	\label{rem ground states sono radially decreasing}
Thanks to the Polya-Szego type inequalities of the $x$-symmetric rearrangements, we could always assume that the minimizing sequences for the nonlinear Schr\"odinger equation on $\R^N \x \cG$ are radially decreasing in the $x$-variable.
\end{remark}
\section{Background on the NLSE on \texorpdfstring{$\R^N\x\cG$}{RN x G}, where 
\texorpdfstring{$\cG$}{G} is a compact metric graph}\label{sez background}
This section is devoted to introducing the fundamental properties of the nonlinear Schr\"odinger energy and the Gagliardo-Nirenberg inequality.

\subsection{Basic facts on the NLS energy on \texorpdfstring{$\R^N\x\cG$}{RN x G}}
For $u \in H^1(\R^N\x\cG)$, let
\begin{equation}
\label{schroedinger energy}
\begin{split}
    E_p\left(u, \R^N \x \cG\right):&=\frac{1}{2}\|\nabla_x u\|_{L^2(\R^N \x \cG)}^2+\frac{1}{2}\|\partial_yu\|_{L^2(\R^N \x \cG)}^2-\frac{1}{p}\|u\|_{L^p(\R^N \x \cG)}^p\\
    &=\frac{1}{2}\|\nabla_{x,y} u\|_{L^2(\R^N \x \cG)}^2-\frac{1}{p}\|u\|_{L^p(\R^N \x \cG)}^p
\end{split}
\end{equation}
be the nonlinear Schr\"odinger energy, with power-type nonlinearity. We search for local or global minimizers of $E_p(\cdot, \R^N \x \cG)$ in 
\begin{equation*}
    H^1_{\mu}\left(\R^N \x \cG\right):=\{u \in  H^1(\R^N \x \cG) \text{ such that } \|u\|_{L^2(\R^N \x \cG)}^2=\mu\};
\end{equation*}
in particular, we investigate the existence of a function $u$ that realizes
\begin{equation*}
    \mathcal{E}_{p}\left(\mu, \R^N \x \cG\right):=\inf\limits_{u \in H^1_\mu(\R^N \x \cG)}E_p(u, \R^N \x \cG).
\end{equation*}
If $u$ is a ground state of mass $\mu$ for $E_p(\cdot, \R^N \x \cG)$, then we can assume that $u\geq 0$ (since $E_p(u, \R^N \x \cG)=E_p(|u|, \R^N \x \cG)$), and $u$ is radially decreasing in the $x$-variable, by Remark \ref{rem ground states sono radially decreasing}. Moreover, $u \in H^1(\R^N \times e)$ is a solution of the stationary nonlinear Schr\"odinger equation:
\begin{equation}
	\label{schroedinger equation}
	\begin{cases}
		-\Delta_x u_e-\partial_y^2u_e+\omega u_e=|u_e|^{p-2}u_e \qquad &\text{in $\R^N \times e$, $\forall e \in \mathfrak{E}$} \\
		\sum\limits_{e \text{ incident in }v}\partial_y u_e(x,v)=0 \qquad &\forall v \in \mathcal{V}, x \in \mathbb{R}^N, 
		\end{cases}
	\end{equation}
	for some Lagrange multiplier $\omega=\omega(\mu) \in \mathbb{R}$, where $\partial_y u_e(x,v)$ is the outer derivative in the $y$ direction with respect to the vertex $v$. Namely, $u$ is a classical solution of the nonlinear Schr\"odinger equation on each $\R^N \times e$, and additional continuity and Kirchhoff conditions at the vertices of $\cG$ hold. This is obtained through standard arguments, analogous to those employed in \cite[Proposition 3.3]{2015} for metric graphs. 
\begin{remark}
\label{soluzione triviale}
A particularly relevant class of solutions is given by those that are independent on the variable \(y\), which we shall henceforth call \emph{semi-trivial solutions}.
With this ansatz, a solution $u(x,y)=Z(x) \in H^1(\R^N)$ satisfies
\begin{equation}
	\label{schroedinger equation R^N}
-\Delta_x Z+\omega Z=|Z|^{p-2}Z \qquad \text{ on }\R^N,
\end{equation}
and is such that 
\[
\|Z\|_{L^2(\R^N \x \cG)}^2=\mu \quad \implies \quad  \|Z\|_{L^2(\R^N)}^2=\frac{\mu}{l(\cG)}.
\]
It is well known (see \cite{Kwong}) that, for every $p \in (2, 2^*) \setminus \{2+4/N\}$ (in particular, for every $p \in (2,2+4/N)$) and for every $\mu>0$, equation \eqref{schroedinger equation R^N} together with the mass constraint $\|Z\|_{L^2(\R^N)}^2=\mu/l(\cG)$ admits a unique positive solution for some $\omega>0$, solution which is radially symmetric and decreasing. We denote such a solution by $Z_{\mu/l(\cG),p}$. For $p \in (2,2+4/N)$, the function $Z_{\mu/l(\cG),p}$ is also a mass-constrained global minimizer of 
\[
E_p\left(u, \R^N\right):=\frac{1}{2}\|\nabla_x u\|_{L^2(\R^N)}^2-\frac{1}{p}\|u\|_{L^p(\R^N)}^p 
\]
on the $L^2$-sphere
\[ 
H^1_{\frac{\mu}{l(\cG)}}\left(\R^N\right):=\left\{u \in H^1(\R^N) \text{ such that } \|u\|_{L^2(\R^N)}^2=\frac{\mu}{l(\cG)}\right\}.
\]
For each $\nu>0$, let 
\[
\cE_p(\nu, \R^N):=\inf\limits_{u \in H^1_{\nu}(\R^N)}E_p(u, \R^N)<0.
\]
Then, by standard Pohozaev identities and scaling arguments,
\begin{equation}
\label{identita di Pohozaev}
\left\|\nabla_x Z_{\nu,p}\right\|_{L^2(\R^N)}^2=\frac{N(p-2)}{2p}\left\|Z_{\nu,p}\right\|_{L^p(\R^N)}^p \implies \left\|\nabla_x Z_{\nu,p}\right\|_{L^2(\R^N)}^2=-\frac{2N(p-2)}{4-N(p-2)}\cE_p(\nu, \R^N).
\end{equation}
Moreover, by a straightforward computation,
	\begin{equation}\label{da Znup a Z1p}
	Z_{\nu, p}(x) = \nu^{\frac{2}{4-N(p-2)}}Z_{1,p}\left(\nu^{\frac{p-2}{4-N(p-2)}}x\right).
	\end{equation}
In particular,
\begin{equation}
	\label{scalamento norme}
\|\nabla_xZ_{\nu, p}\|_{L^2(\R^N)}^2=\nu^{1+\frac{2(p-2)}{4-N(p-2)}}\|\nabla_xZ_{1, p}\|_{L^2(\R^N)}^2, \qquad \|Z_{\nu, p}\|_{L^p(\R^N)}^p=\nu^{1+\frac{2(p-2)}{4-N(p-2)}}\|Z_{1, p}\|_{L^p(\R^N)}^p,
\end{equation}
and
\begin{equation}
	\label{scalamento energia}
	\cE_p(\nu, \R^N)=\nu^{1+\frac{2(p-2)}{4-N(p-2)}}\cE_p(1, \R^N).
\end{equation}
These identities will be useful in the following. We will refer to $Z_{\mu/l(\cG),p}$ both as an element of $H^1_{\mu/l(\cG)}(\R^N)$ and of $H^1_{\mu}(\R^N \x \cG)$, without changing the notation.
\end{remark}

Before discussing the Gagliardo-Nirenberg inequality, we introduce the following useful scaling.
\begin{remark}
	\label{remark scaling}
Let $\cG$ be a metric graph. For every $L>0$, we denote by $L\cG$ the dilation of the metric graph $\cG$ by a factor $L$; with this, we mean that each edge is scaled by a factor $L$ in the following way: if $e$ is identified with the closed interval $[0,\ell_e]$, where $\ell_e$ is the length of $e$, then the scaled edge $Le$ is identified with $[0,L\ell_e]$. It is clear that this operation respects the gluing of the new vertices. At this point we can also consider the following scaling of functions in $H^1(\R^N \x \cG)$: for $u\simeq \{u_e\}_{e \mathfrak{E}} \in H^1(\R^N \x \cG)$ and $p >2$, we pose
		\begin{equation}\label{scaling}
		u_L:=\{u_{Le}\}_{e \in \mathfrak{E}}, \quad \text{where} \quad u_{Le}(x,y) :=L^{-\frac{2}{p-2}}u_e\left(\frac{x}{L}, \frac{y}{L}\right).
	\end{equation}
	Note that this definition depends on the particular choice of $p$, but we do not stress this dependence to ease the notation. Direct computations lead to
	\begin{itemize}
		\item[($i$)] $\|u_L\|_{L^2(\R^N \x L\cG)}^2=L^{-\frac{4-(N+1)(p-2)}{p-2}}\|u\|_{L^2(\R^N \x \cG)}^2$
		\item[($ii$)] $\|u_L\|_{L^p(\R^N \x L\cG)}^p=L^{-\frac{4-(N-1)(p-2)}{p-2}}\|u\|_{L^p(\R^N \x \cG)}^p$;
		\item[($iii$)] $\|\nabla_{x,y}u_L\|_{L^2(\R^N \x L\cG)}^2=L^{-\frac{4-(N-1)(p-2)}{p-2}} \|\nabla_{x,y} u\|_{L^2(\R^N \x \cG)}^2$.
	\end{itemize}
	Thus, we also have
	\[
	E_p(u_L, \R^N \x L\cG)=L^{- \frac{4-(N-1)(p-2)}{p-2}}E_p(u, \R^N \x \cG).
	\]
\end{remark}

\subsection{Gagliardo-Nirenberg inequality}

The fundamental tool for the investigation of energy minimizers on $H^1_{\mu}\left(\R^N \x \cG\right)$ is the Gagliardo-Nirenberg inequality. We prove it in a localized form, which is the one we shall actually use: it controls the $L^p$ norm of $u$ by its global $H^1$ norm and by the $L^2$ norm of $u$ on a \emph{single} slab $Q \x \cG$, $Q$ a unit cube of $\R^N$. The global inequality follows at once, see Corollary \ref{proposition gn} below.

Throughout, we fix once and for all a decomposition of $\R^N$ into open unit cubes $\{Q_n\}_{n \in \mathbb{N}}$ centred at the points of $\mathbb{Z}^N$, so that
\begin{equation}
	\label{decomposition R^N}
	\R^N \x \cG = \bigcup\limits_{n \in \mathbb{N}}\overline{Q_n}\x\cG, \qquad \left|Q_n \cap Q_m\right|=0 \quad \text{for } n \neq m.
\end{equation}
Accordingly, the family $\left\{Q_n \x e\right\}_{(n,e) \in \mathbb{N} \x \mathfrak{E}}$ is a countable family of pairwise disjoint boxes of $\R^{N+1}$ covering $\R^N \x \cG$ up to a set of measure zero, whence, for every $u \in H^1(\R^N \x \cG)$,
\begin{equation}
	\label{additivita}
	\sum\limits_{(n,e)}\|u\|_{L^2(Q_n \x e)}^2 = \|u\|_{L^2(\R^N \x \cG)}^2, \qquad \sum\limits_{(n,e)}\|u\|_{H^1(Q_n \x e)}^2=\|u\|_{H^1(\R^N \x \cG)}^2.
\end{equation}

We shall use the following elementary summation inequality, whose point is that it holds with constant $1$ for a \emph{countable} family of indices, and with no restriction other than $\alpha+\beta \geq 1$.

\begin{lemma}
	\label{lemma somma discreta}
	Let $\alpha, \beta \geq 0$ with $\alpha+\beta \geq 1$. Then, for every pair of sequences $\{a_j\}_{j \in J}, \{b_j\}_{j \in J}\in \ell^1$ of nonnegative numbers, with $J$ countable, $\{a_j^{\alpha}b_j^{\beta}\}_{j \in J} \in \ell^1$, and
	\begin{equation}
		\label{somma discreta}
		\sum\limits_{j \in J} a_j^{\alpha}b_j^{\beta} \leq \Big(\sum\limits_{j \in J} a_j\Big)^{\alpha}\Big(\sum\limits_{j \in J} b_j\Big)^{\beta}.
	\end{equation}
\end{lemma}
\begin{proof}
	Assume first $\alpha, \beta>0$, and set $\sigma:=\alpha+\beta \geq 1$. Since $\frac{1}{1/\alpha}+\frac{1}{1/\beta}=\alpha+\beta=\sigma$, the generalized H\"older inequality $\|fg\|_{\ell^{1/\sigma}} \leq \|f\|_{\ell^{1/\alpha}}\|g\|_{\ell^{1/\beta}}$ gives
	\[
	\big\|(a_j^{\alpha}b_j^{\beta})_j\big\|_{\ell^{1/\sigma}} \leq \big\|(a_j^{\alpha})_j\big\|_{\ell^{1/\alpha}}\big\|(b_j^{\beta})_j\big\|_{\ell^{1/\beta}}=\Big(\sum_j a_j\Big)^{\alpha}\Big(\sum\limits_{j} b_j\Big)^{\beta}.
	\]
	On the other hand $1/\sigma \leq 1$, and the map $q \mapsto \|\cdot\|_{\ell^q(J)}$ is nonincreasing. Therefore
	\[
	\sum\limits_{j \in J} a_j^{\alpha}b_j^{\beta}=\big\|(a_j^{\alpha}b_j^{\beta})_j\big\|_{\ell^{1}} \leq \big\|(a_j^{\alpha}b_j^{\beta})_j\big\|_{\ell^{1/\sigma}},
	\]
	and \eqref{somma discreta} follows. If $\beta=0$ then $\alpha \geq 1$, and \eqref{somma discreta} reduces to $\sum_j a_j^{\alpha} \leq \left(\sum_j a_j\right)^{\alpha}$, which is again the monotonicity of $q \mapsto \|\cdot\|_{\ell^q}$; the case $\alpha=0$ is symmetric.
\end{proof}

\begin{proposition}
	\label{prop GN localizzata}
	Let $\cG$ be a compact metric graph, $N \geq 1$, and let $p \in \left(2,2(N+1)/(N-1)\right)$ ($p \in (2, +\infty)$ if $N=1$). Set
	\[
	\theta(p):=\frac{(N+1)(p-2)}{2}, \qquad \delta(p):=\max\left\{0, \, 1-\frac{\theta(p)}{2}\right\},
	\]
	and, for $u \in H^1(\R^N \x \cG)$,
	\[
	S(u):=\sup\limits_{n \in \mathbb{N}}\|u\|_{L^2(Q_n \x \cG)}.
	\]
	Then there exists $C=C(N,p,\cG)>0$ such that
	\begin{equation}
		\label{GN localizzata}
		\|u\|_{L^p(\R^N \x \cG)}^p \leq C \|u\|^{\theta(p)}_{H^1(\R^N \x \cG)}\|u\|^{2\delta(p)}_{L^2(\R^N \x \cG)}S(u)^{p-\theta(p)-2\delta(p)} \qquad \forall u \in H^1(\R^N \x \cG).
	\end{equation}
	In particular, \eqref{GN localizzata} reads
	\[
	\begin{split}
	&\|u\|_{p}^p \leq C \|u\|^{\theta(p)}_{H^1}S(u)^{p-\theta(p)}\hphantom{\|u\|^{p-2}} \quad \text{if } p \ge 2+\frac{4}{N+1}, \\
	&\|u\|_{p}^p \leq C \|u\|^{\theta(p)}_{H^1}\|u\|_2^{2-\theta(p)}S(u)^{p-2} \quad \text{if } p < 2+\frac{4}{N+1}.
	\end{split}
	\]
\end{proposition}

\begin{proof}
	To shorten the notation we write $\theta:=\theta(p)$, $\delta:=\delta(p)$, and
	\[
	A_{n,e}:=\|u\|_{H^1(Q_n \x e)}^2, \qquad B_{n,e}:=\|u\|_{L^2(Q_n \x e)}^2, \qquad (n,e) \in \mathbb{N}\x\mathfrak{E}.
	\]
Each $Q_n \x e$ is a bounded box of $\R^{N+1}$, so that, by \cite[pp. 125-126]{MR109940} (see also, e.g., \cite[Sections 12.5, 13.3]{MR3726909}),
	\begin{equation}
		\label{GN locale}
		\|u\|_{L^p(Q_n\x e)}^p \leq C_0 \,A_{n,e}^{\frac{\theta}{2}}B_{n,e}^{\frac{p-\theta}{2}} \qquad \forall n \in \mathbb{N}, \ e \in \mathfrak{E},
	\end{equation}
	for some $C_0>0$ independent on $n$ and $e$: independence on $n$ follows from the invariance of \eqref{GN locale} under translations of $\R^N$, while the constant depends on $e$ only through its length $\ell_e$, and $\mathfrak{E}$ is a finite set, $\cG$ being compact. 
	
	Since $Q_n \x e \subseteq Q_n \x \cG$, we have $B_{n,e} \leq \|u\|_{L^2(Q_n \x \cG)}^2 \leq S(u)^2$ for every $(n,e)$. Moreover $\frac{p-\theta}{2}-\delta > 0$: if $\theta \geq 2$ this amounts to $p>\theta(p)$, which holds precisely because $p<2(N+1)/(N-1)$ (every $p>2$ if $N=1$), whereas if $\theta<2$ it amounts to $\frac{p-\theta}{2}-1+\frac{\theta}{2}=\frac{p-2}{2}> 0$. Consequently
	\begin{equation}
		\label{spezzamento}
		B_{n,e}^{\frac{p-\theta}{2}}=B_{n,e}^{\delta}\,B_{n,e}^{\frac{p-\theta}{2}-\delta} \leq B_{n,e}^{\delta}\,S(u)^{p-\theta-2\delta}.
	\end{equation}

Now, by definition $\frac{\theta}{2}+\delta=\max\left\{\frac{\theta}{2},1\right\}\geq 1$, so that Lemma \ref{lemma somma discreta} applies with $\alpha=\frac{\theta}{2}$, $\beta=\delta$ and $J=\mathbb{N}\x\mathfrak{E}$. Summing \eqref{GN locale} over $(n,e)$ and using \eqref{spezzamento}, \eqref{somma discreta} and \eqref{additivita}, we obtain
	\[
	\begin{split}
		\|u\|_{L^p(\R^N \x \cG)}^p &= \sum\limits_{(n,e)}\|u\|_{L^p(Q_n\x e)}^p \leq C_0\, S(u)^{p-\theta-2\delta}\sum\limits_{(n,e)}A_{n,e}^{\frac{\theta}{2}}B_{n,e}^{\delta} \\
		&\leq C_0\, S(u)^{p-\theta-2\delta}\Big(\sum\limits_{(n,e)}A_{n,e}\Big)^{\frac{\theta}{2}}\Big(\sum\limits_{(n,e)}B_{n,e}\Big)^{\delta} =C_0\,\|u\|^{\theta}_{H^1(\R^N \x \cG)}\|u\|^{2\delta}_{L^2(\R^N \x \cG)}S(u)^{p-\theta-2\delta},
	\end{split}
	\]
	which is the desired inequality.
\end{proof}

\begin{cor}
	\label{proposition gn}
	Let $\cG$ be a compact metric graph, $N \geq 1$. Then, for every $p \in \left(2,2(N+1)/(N-1)\right)$ ($p \in (2, +\infty)$ if $N=1$), there exists a constant $C>0$ such that
	\begin{equation}
		\label{GN}
		\|u\|_p^p \leq C \|u\|^{\theta(p)}_{H^1}\|u\|^{p-\theta(p)}_2 \qquad \forall u \in H^1(\R^N \x \cG), 
	\end{equation}
	where $\theta(p):=(N+1)(p-2)/2$.
\end{cor}
\begin{proof}
	It suffices to insert $S(u) \leq \|u\|_{L^2(\R^N \x \cG)}$ into \eqref{GN localizzata}, and to observe that $2\delta+\left(p-\theta-2\delta\right)=p-\theta$.
\end{proof}

\begin{remark}
	\label{GN localizzata radiale}
	Let $Q_0$ be the cube of the decomposition \eqref{decomposition R^N} centred at $0 \in \R^N$, and let $u$ be radially symmetric and decreasing in the variable $x$. Then it is not difficult to check that $S(u) = \|u\|_{L^2(Q_0 \x \cG)}$, namely $S(u)$ is achieved on the cube with center at $0$.
\end{remark}

\begin{remark}
	\label{vanishing}
	Proposition \ref{prop GN localizzata} is a Gagliardo-Nirenberg counterpart of Lions' vanishing lemma: if $\{u_n\} \subseteq H^1(\R^N \x \cG)$ is bounded and $S(u_n) \to 0$, then $u_n \to 0$ in $L^p(\R^N \x \cG)$ for every $p$ as above. A vanishing lemma of a similar flavour, on two dimensional open books, is \cite[Lemma 3.4]{LeCozShakarov}.
\end{remark}

 For future convenience, we reformulate the Gagliardo-Nirenberg inequality replacing the standard $H^1$ norm of $u$ with an equivalent norm: by \eqref{GN}, for every $B>0$ there exist an optimal constant $K_{\R^N \x \cG,p,B}>0$ such that
	\begin{equation}
		\label{GN con costanti}
		\|u\|_p^p \leq K_{\R^N \x \cG,p,B} \left(\|\nabla_{x,y}u\|_2^2+\frac{B}{l(\cG)^2}\|u\|_2^2\right)^{\frac{\theta(p)}{2}}\|u\|^{p-\theta(p)}_2 \qquad \forall u \in H^1(\R^N \x \cG).
	\end{equation}
	The optimal constant is characterized as
	\[
	K_{\R^N \x \cG,p,B}:=
	\sup\limits_{u \in H^1(\R^N \x \cG) \setminus \{0\}} Q_{p,B}(u, \R^N \x \cG),
	\]
	where 
	\beq\label{def Q GN}
	Q_{p,B}(u, \R^N \x \cG):=\frac{\|u\|_p^p}{\left(\|\nabla_{x,y}u\|_2^2+\frac{B}{l(\cG)^2}\|u\|_2^2\right)^{\frac{\theta(p)}{2}} \|u\|_2^{p-\theta(p)}}. 
	\eeq
	This new version of the Gagliardo-Nirenberg inequality has the advantage that the optimal constant for \eqref{GN con costanti} is invariant under dilatations of $\cG$ by a factor $L>0$, namely
	\begin{equation}\label{invarianza costante GN}
		K_{\R^N \x \cG,p,B}=K_{\R^N \x L\cG,p,B} \qquad \forall L>0,
	\end{equation}
	by a trivial application of the transformation \eqref{scaling}.

\begin{remark}\label{def KRN+1}
	On $\R^{N+1}$ and $\R^{N+1}_+:=\R^N \x [0, +\infty)$, the Gagliardo-Nirenberg inequality takes the following forms:
	\[
	\|u\|_p^p \leq K_{\R^{N+1},p} \|\nabla_{x,y}u\|_2^{\theta(p)}\left\|u\right\|_2^{p-\theta(p)}, \qquad \|u\|_p^p \leq K_{\R^{N+1}_+,p} \|\nabla_{x,y}u\|_2^{\theta(p)}\left\|u\right\|_2^{p-\theta(p)},
	\]
	for every $u \in H^1(\R^{N+1})$ and $u \in H^1(\R^{N+1}_+)$, respectively (note that on the right hand side we have the $L^2$ norm of the gradient instead of the full $H^1$ norm of $u$). The validity of the inequality on $\R^{N+1}_+$ can be proved by extending each $u \in H^1(\R^{N+1}_+)$ on the whole space $\R^{N+1}$ by even symmetry with respect to $\{x_{N+1}=0\}$, and then applying the inequality on $\R^{N+1}$. The values $K_{\R^{N+1},p}$ and $K_{\R^{N+1}_+,p}$ are the optimal constants for the inequalities, defined by
	\[
	K_{\R^{N+1},p}:=\sup\limits_{u \in H^1(\R^{N+1})\setminus\{0\}} Q_p(u, \R^{N+1})
	\qquad 	K_{\R^{N+1}_+,p}:=\sup\limits_{u \in H^1(\R^{N+1}_+)\setminus\{0\}} Q_p(u, \R^{N+1}_+),
	\]
	where
	\[\begin{split}
		Q_p\left(u, \Omega\right):=\frac{\|u\|_{L^p(\Omega)}^p}{\left\|\nabla_{x,y}u\right\|_{L^2(\Omega)}^{\theta(p)}\|u\|_{L^2(\Omega)}^{p-\theta(p)}}.
	\end{split}
	\]
	Since one can choose a radially symmetric (and decreasing) maximizing sequence for $K_{\R^{N+1},p}$, arguing by even symmetry with respect to $\{x_{N+1}=0\}$ it is not difficult to check that $
	2^{\frac{p}{2}-1}K_{\R^{N+1},p}=K_{\R^{N+1}_+,p}$.	In a similar way, if \(I\) is segment of length $\pi$, and \(S_1\) is the unit circle, both regarded as metric graphs, we have $
	2^{\frac{p}{2}-1}K_{\R^N \x S_1,p,4B}=K_{\R^N \x I,p,B}$.
\end{remark}

In the next proposition, we compare the optimal constants on general graphs with the prototypical cases of $\R^N \times S_1$ and $\R^N \times I$. We recall that assumption (H) was introduced in \eqref{(H)}.		
\begin{proposition}
	\label{stime masse critiche}
	     Let $\cG$ be a compact metric graph. Then, for every $p \in (2, 2(N+1)/(N-1))$ ($p \in (2, +\infty)$ if $N=1$),
			\begin{equation}\label{confronto costanti g.n. dall'alto}
				K_{\R^N \x \cG,p,B} \leq K_{\R^N \x I,p,B}, \qquad \text{ and } \qquad K_{\R^N \x \cG,p,B} \leq K_{\R^N \x S_1,p,B} \text{ if $\cG$ satisfies condition $(H)$.}
			\end{equation}
		Moreover, if $p \geq 2_*=2+\frac{4}{N+1}$, 
			\begin{equation}\label{confronto costanti g.n. dal basso}
				K_{\R^{N+1},p} \leq K_{\R^N \x \cG,p,B}, \qquad \text{ and } \qquad 	K_{\R^{N+1}_+,p} \leq K_{\R^N \x \cG,p,B} \text{ if $\cG$ contains a terminal edge.}
			\end{equation}
\end{proposition}

\begin{proof}
	\textit{First step: proof of \eqref{confronto costanti g.n. dall'alto}.} Let $\{u_n\} \subseteq H^1(\R^N \x \cG)\setminus \{0\}$ be a maximizing sequence for $Q_{p,B}(\cdot, \R^N \x \cG)$. By the properties of the $y$-decreasing rearrangement stated in Proposition \ref{propr y-riarrangiamento decr}, the rearranged sequence $\{u_n^{*y}\}\subseteq H^1(\R^N \x \cG^{*})$ satisfies
			\[\begin{split}
				K_{\R^N \x \cG,p,B}&=\lim\limits_n Q_{p,B}\left(u_n, \R^N \x \cG\right) \leq \limsup\limits_n Q_{p,B}\left(u_n^{*y}, \R^N \x \cG^*\right) 
				\leq K_{\R^N \x \cG^*, p, B}=K_{\R^N \x I, p, B},
			\end{split}\]
	which proves the first inequality. Assume now that $\cG$ admits a cycle covering. Let $\left\{\widehat{u_n}^y\right\}\subseteq H^1(\R^N \x S_1)$ denote the $y$-symmetric rearrangement of $\{u_n\}$. By Proposition \ref{propr y-riarrangiamento simm} and Remark \ref{remark sulla condiz (H)},
		\[\begin{split}
				K_{\R^N \x \cG, p, B}=\lim\limits_n Q_{p, B}\left(u_n, \R^N \x \cG\right) \leq \limsup\limits_n Q_{p, B}\left(\widehat{u_n}^y, \R^N \x \widehat{\cG}\right)
			   \leq K_{\R^N \x \widehat{\cG}, p, B}=K_{\R^N \x S_1, p, B},
		\end{split}\]
	which is the desired result. \\
	\textit{Second step: proof of the first inequality in \eqref{confronto costanti g.n. dal basso}.} We argue by a standard embedding argument. Let $\{u_n\} \subseteq C^{\infty}_c(\R^{N+1})$ be a maximizing sequence for $Q_p(\cdot, \R^{N+1})$. Since the quotient $Q_p(v, \R^{N+1})$ is homogeneous, we may normalize 
		\[
		\|u_n\|_{L^2(\R^{N+1})}^2=1 \qquad \forall n \in \mathbb{N}.
		\]
	Fix any edge $e \in \mathfrak{E}$. For each $n$ choose $L_n>0$ sufficiently small so that the rescaled function $(u_n)_{L_n}$, defined in \eqref{scaling}, is supported in $\R^N \x e$. Hence $(u_n)_{L_n}$ can be regarded as an element of both $H^1(\R^{N+1})$ and $H^1(\R^N\times\mathcal G)$. Choosing $L_n$ smaller if necessary, we may further assume that
		\[
		L_n \rightarrow 0, \qquad  \|(u_n)_{L_n}\|_{L^p(\R^{N+1})}^p \rightarrow +\infty, \qquad  \left\|\nabla_{x,y}(u_n)_{L_n}\right\|_{L^2(\R^{N+1})}^2 \rightarrow +\infty.
		\] 
    By Remark \ref{remark scaling},	
		\[
			\left\|\left(u_n\right)_{L_n}\right\|_{L^2\left(\R^{N+1}\right)}^2=L_n^{-\frac{4-(N+1)(p-2)}{p-2}}\left\|u_n\right\|_{L^2\left(\R^{N+1}\right)}^2,
		\]	
	and therefore
	    \[
			\left\|\nabla_{x,y}\left(u_n\right)_{L_n}\right\|_{L^2\left(\R^{N+1}\right)}^2+\frac{B}{l(\cG)^2}\left\|\left(u_n\right)_{L_n}\right\|_{L^2\left(\R^{N+1}\right)}^2 \asymp \left\|\nabla_{x,y}\left(u_n\right)_{L_n}\right\|_{L^2\left(\R^{N+1}\right)}^2 
		\]
	as $n \to \infty$. Indeed, $\|\nabla_{x,y}(u_n)_{L_n}\|_{L^2(\R^{N+1})}^2 \rightarrow +\infty$, while $\left\|u_n\right\|_{L^2\left(\R^{N+1}\right)} \equiv 1$ and $L_n^{{-(4-(N+1)(p-2)) / (p-2)}}$ remains bounded, since $L_n \rightarrow 0$ and $-(4-(N+1)(p-2))/(p-2) \geq 0$.\\
	Since $Q_p(v_L, \R^{N+1})=Q_p(v, \R^{N+1})$ for every $L>0$, the sequence $\{(u_n)_{L_n}\} \subseteq H^1(\R^n \x \cG)$ is still maximizing. Consequently,
		\[\begin{split}
			K_{\R^{N+1},p} &=\lim\limits_{n} \frac{\|(u_n)_{L_n}\|_p^p}{\left\|\nabla_{x,y}(u_n)_{L_n}\right\|_2^{\theta(p)}\|(u_n)_{L_n}\|_2^{p-\theta(p)}}\\ 
			&\leq \limsup \limits_{n} \frac{\|(u_n)_{L_n}\|_p^p}{\left(\left\|\nabla_{x,y}(u_n)_{L_n}\right\|_2^2+\frac{B}{\left(l(\cG)\right)^2}\|(u_n)_{L_n}\|_2^2\right)^{\frac{\theta(p)}{2}}\|(u_n)_{L_n}\|_2^{p-\theta(p)}}\\
			& = \limsup\limits_n Q_{p, B}\left((u_n)_{L_n}, \R^N \x \cG\right)\leq K_{\R^N \x \cG,p,B},
	      \end{split}\]
     which proves the first inequality. \\
     \textit{Third step: proof of the the second inequality in \eqref{confronto costanti g.n. dal basso}.} If $\cG$ contains a terminal edge $\bar e$, let $\{u_n\} \subseteq C^{\infty}_c(\R^{N+1}_+)$ be a maximizing sequence for $Q_p(\cdot, \R^{N+1}_+)$. Choosing $L_n \rightarrow 0$ sufficiently small, the rescaled functions $(u_n)_{L_n}$ are supported in $\R^N \x \bar e \subseteq \R^N \x \cG$. Repeating verbatim the argument of the previous step, we obtain the second inequality in \eqref{confronto costanti g.n. dal basso}.
\end{proof}
			\begin{remark}
			Obtaining explicit upper bounds for $K_{\R^N \x \cG, p, B}$ seems to be a challenging problem. By Proposition \ref{stime masse critiche}, $K_{\R^N \x \cG, p, B}$ is bounded above by $K_{\R^N \x I, p, B}$, or by $K_{\R^N \x S_1, p, B}$ if $\cG$ satisfies $(H)$, and Remark \ref{def KRN+1} relates the two; the problem thus reduces to estimating $K_{\R^N \x S_1, p, B}$ for a given $B$. A straightforward interpolation argument yields
	\[
	K_{\R^N \x S_1, p, B}\leq \left(K_{\R^N \x S_1, 2^*, B}\right)^{\frac{\theta(p)}{2^*}}, \qquad 2^*:=\frac{2(N+1)}{N-1},
	\]
	where at the Sobolev exponent $p=2^*$ one has $\theta(2^*)=2^*$ and \eqref{GN con costanti} reduces to the Sobolev inequality on $\R^N \x \cG$; moreover, by \cite[Theorem 7.2]{hebeynonlinear} there exists $B=B(\R^N \x S_1)$ with $K_{\R^N \x S_1, 2^*, B}=K(N+1,2)^{2^*}$, $K(N+1,2)$ being the optimal Sobolev constant of $\R^{N+1}$. Estimating such a $B$, however, appears to be nontrivial. For manifolds of the form $\R \x S^k$, $k \geq 2$, one has the explicit value $B(\R \x S^k)=\frac{k-1}{4k}Scal_{(\R\x S^k,g)}$ \cite[Theorem 7.7]{hebeynonlinear}, but the proof relies on $\R \x S^k$ being conformally flat, of constant scalar curvature \emph{and} simply connected, and the last property fails for $\R^N \x S_1$. Note also that the above formula, if extended to $\R^N \x S_1$, would give $B=0$, the product metric being flat: this is a further indication that it does not extend to our setting, although one should keep in mind that $K_{\R^N \x \cG,p,B}$ itself could diverge as $B \to 0$, so that no conclusion on the thresholds of Section \ref{supercritical case} can be drawn directly. Finding an explicit upper bound for $B$ remains an open problem.
\end{remark}

In the $L^2$-critical case $p=2_*=2+4/(N+1)$, the asymptotic behavior of the optimal
constants on the model products as \(B\to+\infty\) can be determined
exactly.

\begin{proposition}\label{p:asymptotic-GN-constants}
Let \(I\) be a segment of length \(\pi\), and let \(S_1\) be the unit
circle. Then
\[
\lim_{B\to+\infty}
K_{\mathbb R^N\times S_1,2_*,B}
=
K_{\mathbb R^{N+1},2_*} \quad \text{and} \quad 
\lim_{B\to+\infty}
K_{\mathbb R^N\times I,2_*,B}
=
K_{\mathbb R^{N+1}_+,2_*}
=
2^{\frac{2}{N+1}}K_{\mathbb R^{N+1},2_*}.
\]
\end{proposition}
The following discussion can in fact be adapted, with minor modifications, to all $p \geq 2_*$. However, since we will only apply these results in the case $p=2_*$, we restrict our attention to this case.

\begin{proof}
We divide the proof into several steps. The first three deal with $K_{\mathbb R^N\times S_1,2_*,B}$, while the remaining one concerns $K_{\mathbb R^N\times I,2_*,B}$. Recall that the constant \(K_{\mathbb R^N\times \cG,p,B}\) is the supremum of $Q_{2_*,B}(u,\mathbb R^N\times \cG)$ defined in \eqref{def Q GN}. Moreover, at \(p=2_*\), we have $\theta(2_*)=2$ (where $\theta(p)$ is defined in Proposition \ref{prop GN localizzata}).

\smallskip
\noindent
\emph{Step 1) Reduction to cylinders with diverging radius.} As shown in \eqref{invarianza costante GN},
\[
K_{\mathbb R^N\times S_1,2_*,B}
=
K_{\mathbb R^N\times S_r,2_*,B}
\qquad\text{for every }r>0.
\]
We now choose
\[
r_B:=\frac{\sqrt B}{2\pi} \quad \implies \quad \frac{B}{\ell(S_{r_B})^2}
=
\frac{B}{(2\pi r_B)^2}
=1.
\]
Therefore,
\[
K_{\mathbb R^N\times S_1,2_*,B}
=
\mathcal K_{r_B}, \quad \text{where} \quad 
\mathcal K_r
:=
\sup_{v\in H^1(\mathbb R^N\times S_r)\setminus\{0\}}
\mathcal Q_r(v)
\]
and
\[
\mathcal Q_r(v)
:=
\frac{\|v\|_{L^{2_*}(\mathbb R^N\times S_r)}^{2_*}}
{\left(
 \|\nabla_{x,y}v\|_{L^2(\mathbb R^N\times S_r)}^2
 +\|v\|_{L^2(\mathbb R^N\times S_r)}^2
 \right)
 \|v\|_{L^2(\mathbb R^N\times S_r)}^{2_*-2}}.
\]
Since \(r_B\to+\infty\) as \(B\to+\infty\), it remains to prove that
\beq\label{lim K 129}
\lim_{r\to+\infty}\mathcal K_r
=
K_{\mathbb R^{N+1},2_*};
\eeq
indeed, once that this estimate is proved, we infer that
\[
\lim_{B \to +\infty} K_{\R^N \times S_1,2_*,B}= \lim_{r_B\to+\infty} \mathcal K_{r_B}
=
K_{\mathbb R^{N+1},2_*}, 
\]
which completes the first part.

\medskip
\noindent
\emph{Step 2) The lower bound.} As shown in (4.20), $\mathcal K_r = K_{\R^N \x S_1, p, 4\pi^2 r^2} \geq K_{\R^{N+1}, 2_*}$. In particular, 
\[
\liminf_{r\to+\infty}\mathcal K_r
\geq K_{\mathbb R^{N+1},2_*}.
\]
\medskip
\noindent
\emph{Step 3) The upper bound.}
For \(R>1\), let $\chi_R\in C_c^\infty(\mathbb R)$ be such that
\[
0\leq\chi_R\leq1,\qquad
\chi_R=1\ \text{on }[-R,R],\qquad
\chi_R=0\ \text{outside }[-R-1,R+1],
\]
and such that the two transition parts of \(\chi_R\) are translates
of two fixed profiles. For \(q\geq1\), set
\[
A_q(R):=\int_{\mathbb R}|\chi_R(t)|^q\,dt,
\qquad
J(R):=\int_{\mathbb R}|\chi_R'(t)|^2\,dt.
\]
By construction, $A_q(R)=2R+O(1)$ and $J(R)=O(1)$
 as $R\to+\infty$,
so that in particular,
\beq\label{1291}
\frac{A_2(R)}{A_{2_*}(R)}\longrightarrow1
\qquad\text{and}\qquad
J(R)\leq A_2(R)
\eeq
for all sufficiently large \(R\). Assume that $2(R+1)<2\pi r$. For \(a\in S_r\), define the periodic cutoff
\[
\chi_{R,a}(y)
:=
\sum_{k\in\mathbb Z}
\chi_R(y-a+2\pi r k).
\]
The assumption \(2(R+1)<2\pi r\) ensures that the supports of the
terms in this sum are pairwise disjoint, and each support has size smaller than $2\pi r$. Then, $\chi_{R,a}$ can be seen as an element in $H^1(S_r)$. Now, given
\(v\in H^1(\mathbb R^N\times S_r)\), let
\[
v_a(x,y):=\chi_{R,a}(y)v(x,y).
\]
The function \(v_a\) is supported in an arc strictly shorter than the
whole circle. By opening this arc and extending \(v_a\) by zero, we
may regard it as a function in \(H^1(\mathbb R^{N+1})\). Therefore,
the sharp Euclidean Gagliardo--Nirenberg inequality gives
\[
\|v_a\|_{L^{2_*}(\mathbb R^{N+1})}^{2_*}
\leq
K_{\mathbb R^{N+1},2_*}
\|\nabla v_a\|_{L^2(\mathbb R^{N+1})}^2
\|v_a\|_{L^2(\mathbb R^{N+1})}^{2_*-2} \le K_{\mathbb R^{N+1},2_*}
\|\nabla v_a\|_{L^2(\mathbb R^{N+1})}^2\|v\|_2^{2_*-2},
\]
where we used the fact that \(0\leq\chi_{R,a}\leq1\). Integrating the previous inequality with respect to \(a\in S_r\), we
obtain
\[
\int_{S_r}\|v_a\|_{2_*}^{2_*}\,da
\leq
K_{\mathbb R^{N+1},2_*}
\|v\|_2^{2_*-2}
\int_{S_r}\|\nabla v_a\|_2^2\,da.
\]
For every \(y\in S_r\), translation invariance in the variable \(a\)
gives
\[
\int_{S_r}|\chi_{R,a}(y)|^q\,da=A_q(R) \quad \implies \quad 
\int_{S_r}\|v_a\|_{2_*}^{2_*}\,da
=
A_{2_*}(R)\|v\|_{2_*}^{2_*}.
\]
We next compute the Dirichlet term. For the derivatives in the
\(x\)-variables,
\[
\int_{S_r}
\|\chi_{R,a}\nabla_xv\|_2^2\,da
=
A_2(R)\|\nabla_xv\|_2^2.
\]
Moreover, $\partial_yv_a
=
\chi_{R,a}\partial_yv
+
(\partial_y\chi_{R,a})v$. The mixed term vanishes after integration with respect to \(a\),
because
\[
\int_{S_r}
\chi_{R,a}(y)\partial_y\chi_{R,a}(y)\,da
=
\frac12\int_{\mathbb R}(\chi_R^2)'(t)\,dt
=0.
\]
It follows that
\[
\int_{S_r}\|\partial_yv_a\|_2^2\,da
=
A_2(R)\|\partial_yv\|_2^2
+
J(R)\|v\|_2^2.
\]
Combining the identities above, we obtain
\[
\int_{S_r}\|\nabla v_a\|_2^2\,da
=
A_2(R)\|\nabla_{x,y}v\|_2^2
+
J(R)\|v\|_2^2.
\]
Therefore, recalling also that
\(J(R)\leq A_2(R)\) for \(R\) sufficiently large, 
\[
A_{2_*}(R)\|v\|_{2_*}^{2_*}
\leq
K_{\mathbb R^{N+1},2_*}A_2(R)
\left(
 \|\nabla_{x,y}v\|_2^2+\|v\|_2^2
\right)
\|v\|_2^{2_*-2}.
\]
It follows that for every $v\in H^1(\R^N \x \cG)$,
\[
\mathcal Q_r(v)
\leq
K_{\mathbb R^{N+1},2_*}
\frac{A_2(R)}{A_{2_*}(R)} \quad \implies \quad 
\mathcal K_r
\leq
K_{\mathbb R^{N+1},2_*}
\frac{A_2(R)}{A_{2_*}(R)}
\]
whenever \(2(R+1)<2\pi r\). We may now choose, for instance, $R=\pi r/2$. For all sufficiently large \(r\), this choice satisfies the required
condition, and \(R\to+\infty\) as \(r\to+\infty\). Consequently, by \eqref{1291},
\[
\limsup_{r\to+\infty}\mathcal K_r
\leq
K_{\mathbb R^{N+1},2_*}.
\]
Together with the lower bound, this proves that \eqref{lim K 129}, and completes the proof of the asymptotic estimate of $K_{\mathbb R^N\times S_1,2_*,B}$.

\medskip
\noindent
\emph{Step 4) The case of the interval.} By Remark \ref{def KRN+1}, we have
\[
K_{\mathbb R^N\times I,2_*,B}
=
2^{\frac{2_*}{2}-1}
K_{\mathbb R^N\times S_1,2_*,B} = 2^\frac{2}{N+1} K_{\mathbb R^N\times S_1,2_*,B}.
\]
Thus, the first part of the proposition yields
\[
\lim_{B\to+\infty}
K_{\mathbb R^N\times I,2_*,B}
=
2^{\frac{2}{N+1}}K_{\mathbb R^{N+1},2_*}.
\]
Finally, the even-reflection argument recalled in Remark
\ref{def KRN+1} gives $K_{\mathbb R^{N+1}_+,2_*}
=
2^{\frac{2}{N+1}}K_{\mathbb R^{N+1},2_*}$,
and the conclusion follows.
\end{proof}

Observing the the map $B \mapsto K_{\mathbb R^N\times \cG,2_*,B}$
is nonincreasing, and recalling Proposition \ref{stime masse critiche}, we deduce a useful corollary which will be the crucial ingredient in the determination of the optimal threshold for existence of ground states in the $L^2$-critical case.

\begin{cor}\label{cor: GN opt crit}
For every compact metric graph \(\cG\), the limit
\[
K^\infty_{\mathbb R^N\times \cG,2_*}
:=
\lim_{B\to+\infty}
K_{\mathbb R^N\times \cG,2_*,B}
\]
exists. Moreover, $K_{\mathbb R^{N+1},2_*}
\leq
K^\infty_{\mathbb R^N\times \cG,2_*}
\leq
K_{\mathbb R^{N+1}_+,2_*}$, and
\[
\begin{cases}
K^\infty_{\mathbb R^N\times \cG,2_*}
=
K_{\mathbb R^{N+1},2_*} & \text{if \(\cG\) satisfies condition \((H)\)}\\
K^\infty_{\mathbb R^N\times \cG,2_*}
=
K_{\mathbb R^{N+1}_+,2_*}
=
2^{\frac{2}{N+1}}K_{\mathbb R^{N+1},2_*} & \text{if \(\cG\) contains a terminal edge}.
\end{cases}
\]
\end{cor}

\section{The subcritical and critical cases}\label{sez sottocritico}

The purpose of this section is to prove Theorems \ref{caso sottocritico e critico} and \ref{thm 1.3}: we study existence and properties of ground states for the nonlinear Schr\"odinger energy $E_p(\cdot, \R^N \x \cG)$, defined in \eqref{schroedinger energy}, when $p \in (2, 2+4/(N+1)]$, and when they are independent on the variable of the metric graph $\cG$. This investigation is inspired by \cite{TerTzvVis}, which regards the NLS equation on product spaces $\R^N\x M^k$, where $M^k$ is a $k$-dimensional compact Riemannian manifold. In particular, in this section we aim at generalizing to this new setting \cite[Theorems 1.1 and 1.3]{TerTzvVis}, including moreover the critical case $p=2+4/(N+1)$. Results of a similar nature, for strips and for two dimensional open books, are obtained in \cite{fracturedstrip, LeCozShakarov}; we refer to the introduction for a detailed comparison.

\subsection{Existence of ground states in the subcritical case.}

\begin{proof}[Proof of Theorem \ref{caso sottocritico e critico} for $p<2_*$] 
Let $\{u_n\} \subseteq H^1_{\mu}(\R^N \x \cG)$ be a nonnegative $x$-radially decreasing minimizing sequence for $\mathcal{E}_p(\mu, \R^N \x \cG)$. We split the proof into a series of steps, as in \cite[Theorem 1.1]{TerTzvVis}. \\
\textit{First step: up to a subsequence, $u_n \rightharpoonup u$ weakly in $H^1(\R^N \x \cG)$}. By the Gagliardo-Nirenberg inequality \eqref{GN}
\[
E_p(u_n, \R^N \x \cG)\geq \frac{1}{2}\|\nabla_{x,y}u_n\|_2^2-\frac{C}{p}\|u_n\|_{H^1}^{\theta(p)}\mu^{\frac{p-\theta(p)}{2}},
\]
and, since $\theta(p)<2$, the sequence $\{u_n\}$ is bounded. \\
\textit{Second step: $u \not\equiv 0$.} Since $Z_{\mu/l(\cG),p}\in H^1_{\mu}(\R^N \x \cG)$,
\[
0>E_p\left(Z_{\frac{\mu}{l(\cG)},p}, \R^N \x \cG\right) \geq \cE_p(\mu, \R^N \x \cG)=\lim\limits_n \left(\frac{1}{2}\|\nabla_{x,y}u_n\|_2^2-\frac{1}{p}\|u_n\|_p^p\right)\geq -\liminf\limits_n\frac{1}{p}\|u_n\|_p^p,
\]
whence it follows that $\|u_n\|_p^p \geq C>0$ for every $n$. Thus, combining \eqref{GN localizzata} (see also Remark \ref{GN localizzata radiale}) with the boundedness of $\{u_n\}$ in $H^1$,
\[
C \leq \|u_n\|_{L^p(\R^N \x \cG)}^p \leq C\|u_n\|_{H^1(\R^N \x \cG)}^{\theta(p)} \|u_n\|_{L^2(\R^N \x \cG)}^{2-\theta(p)} \|u_n\|_{L^2(Q_0\x\cG)}^{p-2} \leq C \|u_n\|_{L^2(Q_0\x\cG)}^{p-2},
\]
which implies that $\|u_n\|_{L^2(Q_0\x\cG)}^2 \geq C>0$ for every $n$, where we recall that $Q_0 \subseteq \R^N$ denotes the cube in $\R^N$ centered at $0$. Since $u_n \rightarrow u$ in $L^2(Q_0 \x \cG)$, the previous lower bound ensures that $u \not\equiv 0$. \\
\textit{Third step: the function $\mu \rightarrow \cE_p(\mu, \R^N \x \cG)/\mu$ is strictly decreasing.} This fact is proved in \cite[Theorem 1.1, fifth step]{TerTzvVis} in the case when $\R^N \times \cG$ is replaced by the product of $\R^N$ with a compact Riemannian manifold. The arguments therein carry over verbatim to the present setting, and the proof is therefore omitted.\\
\textit{Fourth step: $\|u\|_2^2=\mu$.}
By the Brezis-Lieb Lemma,
\[
\|u_n\|_p^p = \|u_n-u\|_p^p+\|u\|_p^p + o(1).
\]
Moreover, since $\nabla_{x,y}u_n \rightharpoonup \nabla_{x,y}u$ weakly in $L^2(\R^N \times \cG)$,
\[
\|\nabla_{x,y}u_n\|_2^2 = \|\nabla_{x,y}u_n-\nabla_{x,y}u\|_2^2+\|\nabla_{x,y}u\|_2^2 + o(1).
\]
Therefore, 
\[
E_p(u_n, \R^N \x \cG)=E_p(u_n-u, \R^N \x \cG)+E_p(u, \R^N \x \cG)+o(1).
\]
Assume by contradiction that $\|u\|_2^2=\tau\in(0, \mu)$. Then, $\|u_n-u\|_2^2\rightarrow \mu-\tau$, and, by the third step,
\[
\begin{split}	
\cE_p(\mu, \R^N \x \cG)&=\lim\limits_n E_p(u_n, \R^N \x \cG)\geq\limsup\limits_n E_p(u_n-u, \R^N \x \cG)+E_p(u, \R^N \x \cG)\\
&\geq \cE_p(\mu-\tau, \R^N \x \cG)+\cE_p(\tau, \R^N \x \cG)>\frac{\mu-\tau}{\mu}\cE_p(\mu, \R^N \x \cG)+\frac{\tau}{\mu}\cE_p(\mu, \R^N \x \cG)\\
&=\cE_p(\mu, \R^N \x \cG),
\end{split}
\] 
a contradiction. Therefore, $\|u\|_2^2=\mu$. \\
\textit{Fifth step: $u$ is the desired ground state.} Since $\nabla_{x,y}u_n \rightharpoonup \nabla_{x,y}u$,
\[
\|\nabla_{x,y}u\|_2^2 \leq \liminf\limits_n \|\nabla_{x,y}u_n\|_2^2.
\]
Moreover, by the Gagliardo-Nirenberg inequality \eqref{GN},
\[
\|u_n-u\|_p^p \leq C \|u_n-u\|^{\theta(p)}_{H^1}\|u_n-u\|^{p-\theta(p)}_2 \leq C \|u_n-u\|^{p-\theta(p)}_2 \rightarrow 0.
\]
Therefore
\[
\cE_p(\mu, \R^N \x\cG) \leq E_p(u, \R^N \x\cG) \leq  \liminf\limits_n E_p(u_n, \R^N \x\cG)=\cE_p(\mu, \R^N \x\cG),
\]
and $u$ is the desired ground state. Moreover, since up to a subsequence $u_n \to u$ a.e. in $\R^N \times \cG$, $u$ is nonnegative and radially decreasing with respect to the $x$ variable. The positivity follows from the strong maximum principle.
\end{proof}

\subsection{Existence of ground states in the critical case.}

We start with some preliminary remarks.

\begin{remark}\label{g.s. in RN+1}
	Consider the nonlinear Schr\"odinger energy $E_p(u,\R^{N+1})$ on $\R^{N+1}$,
	and let $H^1_{\mu}(\R^{N+1})$ and $\cE_p(\mu, \R^{N+1})$ be defined in an obvious way. 
	When searching for mass-constrained minimizers in $H^1_{\mu}(\R^{N+1})$, it is well known that
	\[
	\cE_p(\mu,\R^{N+1}) = \begin{cases} \in \R_- & \text{if $p \in (2,2_*)$, for every $\mu$}\\ -\infty & \text{if $p >2_*$, for every $\mu$}, \end{cases} \quad \text{and} \quad 
	\cE_{2_*}\left(\mu, \R^{N+1}\right)=\begin{cases}
		0 \qquad &\text{ if } \mu \leq \mu_{\R^{N+1}, 2_*} \\
		-\infty \qquad &\text{ if } \mu > \mu_{\R^{N+1}, 2_*},
	\end{cases}
	\]
	where
	\[
	\mu_{\R^{N+1}, 2_*}:= \left(\frac{2_*}{2K_{\R^{N+1},2_*}}\right)^{\frac{N+1}{2}}
	\]
	and $K_{\R^{N+1},2_*}$ is defined in Remark \ref{def KRN+1}. Moreover, the mass-constrained ground state of $E_{2_*}(\cdot, \R^{N+1})$ is attained if and only if $\mu=\mu_{\R^{N+1}, 2_*}$. Similar considerations hold for $\R^{N+1}_+$, with $\mu_{\R^{N+1}_+, 2_*}= \mu_{\R^{N+1}, 2_*}/2$.
\end{remark}

In accordance to the previous remark, we define
\beq\label{def thr criti 926}
\mu_{\mathrm{ex}}^*
\bigl(\mathbb R^N\times \cG,2_*\bigr)
:=
\left(
\frac{2_*}
{2K^\infty}
\right)^{\frac{N+1}{2}},
\eeq
where $K^\infty:= K^\infty_{\mathbb R^N\times \cG,2_*}$ has been defined in Corollary \ref{cor: GN opt crit}. By Corollary \ref{cor: GN opt crit}, we have
\[
\mu_{\R^{N+1}_+,2_*} \le \mu_{\mathrm{ex}}^*
(\mathbb R^N\times \cG,2_*) \le \mu_{\R^{N+1},2_*},
\]
as stated in Theorem \ref{caso sottocritico e critico}. The upper bound is attained if \(\cG\) satisfies condition \((H)\), whereas the lower bound is attained if \(\cG\) contains a terminal edge.

Note that, by definition \eqref{def thr criti 926}, we have
\beq\label{1293}
\frac{2K^\infty}{2_*}\mu^{\frac{2}{N+1}}<1 \qquad \forall \mu \in (0,
\mu_{\mathrm{ex}}^*
(\mathbb R^N\times \cG,2_*)).
\eeq

\begin{proof}[Proof of Theorem \ref{caso sottocritico e critico} for $p=2_*$]\label{esistenza caso massa sottocritica} \emph{Step 1) Existence of ground states for $\mu<\mu_{\mathrm{ex}}^*
(\mathbb R^N\times \cG,2_*))$.} For any such $\mu$, estimate \eqref{1293} holds. Recalling that $K^\infty= \lim_{B \to \infty} K_{\mathbb R^N\times \cG,2_*,B}$,
we may then fix \(B \gg 1\) sufficiently large so that, setting $K:=K_{\mathbb R^N\times \cG,2_*,B}$, one has
\beq\label{1292}
\frac{2K}{2_*}\mu^{\frac{2}{N+1}}<1.
\eeq
Consider a nonnegative $x$-radially decreasing minimizing sequence $\{u_n\}\subseteq H^1_{\mu}(\R^N \x \cG)$ for $E_{2_*}(\cdot, \R^N \x \cG)$. Then
	\begin{equation}\label{stima E_p}
		\begin{split}
			E_{2_*}(u_n)&=\frac{1}{2}\|\nabla_{x,y}u_n\|_2^2+\frac{B}{2l(\cG)^2}\|u_n\|_2^2-\frac{1}{2_*}\|u_n\|_{2_*}^{2_*}-\frac{B}{2l(\cG)^2}\|u_n\|_2^2 \\
			&=\frac{1}{2}\left(\|\nabla_{x,y}u_n\|_2^2+ \frac{B}{l(\cG)^2} \mu\right)\left(1-\frac{2}{2_*}Q_{2_*, B}(u_n, \R^N \x \cG)\mu^{\frac{2}{N+1}}\right)-\frac{B\mu}{2l(\cG)^2}
			\\
			&\geq\frac{1}{2}\left(\|\nabla_{x,y}u_n\|_2^2+ \frac{B}{l(\cG)^2} \mu\right)\left(1-\frac{2}{2_*}K\mu^{\frac{2}{N+1}}\right)-\frac{B\mu}{2l(\cG)^2}.
		\end{split}
	\end{equation}
	By \eqref{1292}, this implies that $\cE_p(\mu, \R^N \x \cG)>-\infty$, and $\{u_n\}$ is bounded in $H^1\left(\R^N \x \cG\right)$. Up to extracting a subsequence, $\{u_n\}$ admits a weak limit $u \in H^1(\R^N \x \cG)$. At this point, the rest of the proof follows exactly as in the subcritical case. The second step yields $u \not\equiv 0$, whereas the third and fourth steps imply $\|u\|_2^2=\mu$ (clearly, the strict monotonicity of $\mathcal{E}_{2_*}(\mu,\R^N \x \cG)/\mu$ with respect to $\mu$ holds only below the critical threshold). Finally, by the fifth step, $u$ is the desired ground state.

\medskip

\noindent\emph{Step 2) Non-existence of ground states.} By Remark \ref{g.s. in RN+1}, since $\mu>\mu_{\R^{N+1}, 2_*}$ there exists a function $u \in H^1_{\mu}(\R^{N+1})$ with compact support such that $E_{2_*}(u, \R^{N+1})<0$. Moreover, for $L \rightarrow 0$, the scaling $u_L$ defined in \eqref{scaling} is such that:
	\begin{itemize}
		\item[($i$)] $\|u_L\|_2^2 = \mu$ for each $L$;
		\item[($ii$)] $E_{2_*}(u_L, \R^{N+1})=L^{-2}E_{2_*}(u, \R^{N+1}) \rightarrow -\infty$;
		\item[($iii$)] the support of $u_L$ gets smaller as $L \rightarrow 0$. 
	\end{itemize}
	Fix any edge $e$ in $\cG$. For $L$ sufficiently small, the support of $u_L$ can be embedded into $\R^N \x e \subseteq \R^N \x \cG$, and hence 
	\begin{equation}\label{passaggi cE -infty}
		\cE_{2_*}(\mu, \R^N \x \cG)\leq \liminf\limits_{L \to 0^+}E_{2_*}(u_L, \R^N \x \cG)=\liminf\limits_{L \to 0^+}L^{-2}E_{2_*}(u, \R^{N+1})\rightarrow -\infty,
	\end{equation}
	which is the desired result. \\
	Moreover, if $\cG$ contains a terminal edge and $\mu>\mu_{\R^{N+1}_+, 2_*}$, by Remark \ref{g.s. in RN+1} there exists a function $u \in H^1_{\mu}(\R^{N+1}_+)$ with negative energy and with compact support. Applying the scaling \eqref{scaling} with $L \rightarrow 0$, the support of $u_L$ is contained in $\R^N \x e$, where $e$ is the terminal edge of $\cG$. Hence $u_L$ can be regarded as an element in $H^1_\mu(\R^N \x \cG)$, and \eqref{passaggi cE -infty} still holds.
\end{proof}

\subsection{Dependence of ground states on the variable \texorpdfstring{$y$}{y}}

In this subsection we aim at generalizing \cite[Theorem 1.3]{TerTzvVis} to our setting, 
including the case $p=2+4/(N+1)$. We prove that for sufficiently small masses $\mu$ the ground 
state of $E_p(\cdot, \R^N \x \cG)$ in $H^1_{\mu}(\R^N \x \cG)$ is actually the semi-trivial 
solution $Z_{\mu/l(\cG),p}$ presented in Remark \ref{soluzione triviale}. 
To this end, we introduce the following auxiliary functional.
\begin{definition}
For every $u \in H^1(\R^N \x \cG)$, we define
\[
E_{p,\lambda}(u, \R^N \x \cG):=\frac{1}{2}\|\nabla_x u\|_2^2+\frac{\lambda}{2}\|\partial_y u\|_2^2-\frac{1}{p}\|u\|_p^p,
\]
with ground state energy level of unitary mass
\[
\mathcal{E}_{p, \lambda}(\R^N \x \cG):=\inf\limits_{u \in H^1_1(\R^N \times \cG)} E_{p,\lambda}(u, \R^N \x \cG).
\]
\end{definition}
\begin{remark}
Since the function $Z_{1/l(\cG),p} \in H^1_1(\R^N \x \cG)$, for every $\lambda>0$
\begin{equation}
\label{energie minore}
\cE_{p, \lambda}(\R^N \x \cG) \leq E_{p, \lambda}\left(Z_{\frac{1}{l(\cG)},p}, \R^N \x \cG\right) = \int_{\cG}E_p\left(Z_{\frac{1}{l(\cG)},p}, \R^N\right)\, dy=l(\cG)\cE_p\left(\frac{1}{l(\cG)}, \R^N\right).
\end{equation}
\end{remark}
Theorem \ref{thm 1.3} is obtained as a direct consequence of the following:

\begin{proposition}\label{thm 3.1}
Let $\cG$ be a compact metric graph, $\mu>0$, $N \in \mathbb{N}$, $p \in (2,2+4/(N+1)]$. Then, there exists $\lambda^*_{\R^N \x \cG, p}>0$ such that
\begin{equation}
	\begin{split}
		\label{tesi thm 3.1}
		\lambda>\lambda^*_{\R^N \x \cG, p} \implies &\text{each ground state of mass $1$ for $E_{p, \lambda}(\cdot, \R^N \x \cG)$ does not depend} \\
			&\text{on the $y$ variable,} \\
		\lambda<\lambda^*_{\R^N \x \cG, p}\implies &\text{ground states of mass $1$ for $E_{p, \lambda}(\cdot, \R^N \x \cG)$ must depend on the $y$ variable.}
	\end{split}
\end{equation}
\end{proposition}

The proof is divided into several intermediate steps, and initially follows the strategy developed in \cite{TerTzvVis}. 

	\begin{remark}\label{successione un}
	For every sufficiently large $\lambda$, there exists a mass-one ground state $u_\lambda$ of $E_{p, \lambda}(\cdot, \R^N \x \cG)$. Indeed, for every $\nu>0$, define the scaling $T_{\nu}: H^1_{\nu}(\R^N \x \cG) \longrightarrow H^1_1(\R^N \x \cG)$ by
	\begin{equation}
		\label{trasformazione thm 3.1}
			u \longmapsto T_{\nu}(u(x,y)):=\nu^{-\frac{2}{4-N(p-2)}}u\left(\nu^{-\frac{p-2}{4-N(p-2)}}x, y\right).
	\end{equation}
	A direct computation shows that, for every $u \in H^1_{\nu}(\R^N \x \cG)$, 
	\begin{equation}\label{proprieta trasformazione thm 3.1}
	E_{p,\lambda_\nu}(T_{\nu}(u), \R^N \x \cG)=\nu^{\frac{N(p-2)-2p}{4-N(p-2)}}E_p(u, \R^N \x \cG), \qquad \text{ where } \qquad 	\lambda_{\nu}:=\nu^{-\frac{2(p-2)}{4-N(p-2)}}.
	\end{equation}
Therefore, $u_{\nu} \in H^1_{\nu}(\R^N \x \cG)$ is a ground state of $E_p(\cdot, \R^N \x \cG)$ with mass $\nu$ if and only if $T_{\nu}(u_{\nu}) \in H^1_1(\R^N \x \cG)$ is a ground state of $E_{p, \lambda_{\nu}}(\cdot, \R^N \x \cG)$ with mass $1$. \\
Now fix $\lambda>0$ and set 
	\[
	\mu:=\lambda^{-\frac{4-N(p-2)}{2(p-2)}}.
	\] 
	Let $u_{\mu}$ be a ground state of $E_p(\cdot, \R^N \x \cG)$ with mass $\mu$. Its existence follows from Theorem~\ref{caso sottocritico e critico} (in the critical case $p=2+4/(N+1)$, we assume that $\mu$ is sufficiently small, or equivalently, that $\lambda$ is sufficiently large). Then, $u_{\lambda}:=T_{\mu}(u_{\mu}) \in H^1_1(\R^N \x \cG)$ is a mass-one ground state of $E_{p, \lambda}(\cdot, \R^N \x \cG)$. 
\end{remark}

Consider a sequence $\lambda_n \rightarrow +\infty$, and for each $n$ let $u_n:=u_{\lambda_n}$ be a ground state of mass $1$ for $E_{p, \lambda_n}(\cdot, \R^N \x \cG)$ (see Remark \ref{successione un}). We can assume that each $u_n$ is positive and radially decreasing with respect to the $x$ variable. We aim at proving that
\begin{equation}
\label{tesi lemma 3.6}
\text{ for $n$ sufficiently large, } \qquad \partial_y u_n \equiv 0.
\end{equation}
In \cite{TerTzvVis}, this fact is stated in Lemma 3.6, and the proof relies on Lemmas 3.2-3.5. In our setting, these four preliminary propositions can be proved through the same arguments; therefore, we collect these results into a single statement and omit the proof.

\begin{lemma}\label{lem: prop v_n}
In the previous setting, the following statements hold:
\begin{itemize}
\item[($i$)] $\mathcal{E}_{p,\lambda_n}(\R^N \x \cG) \to l(\cG)\mathcal{E}_p\left(1/l(\cG), \R^N\right)$ and $\lambda_n \|\partial_y u_n\|_2^2 \to 0$ as $n \to \infty$.
\item[($ii$)] Each $u_n$ solves 
\begin{equation}
\label{eq diff thm 3.1}
	-\Delta_x u_n - \lambda_n\partial_y^2 u_n + \omega_n u_n=|u_n|^{p-2}u_n \qquad \text{in }\R^N \times \cG,
\end{equation}
for some $\omega_n \in (0,+\infty)$.
\item[($iii$)] $\omega_n \to \overline{\omega} \in (0,+\infty)$, where the limit is characterized by the equation
\begin{equation*}
\label{NLS omega barrato}
-\Delta_x Z_{\frac{1}{l(\cG)},p} + \overline{\omega} Z_{\frac{1}{l(\cG)},p}=|Z_{\frac{1}{l(\cG)},p}|^{p-2}Z_{\frac{1}{l(\cG)},p}.
\end{equation*}
\item[($iv$)] $u_n \to Z_{1/l(\cG),p}$ strongly in $H^1(\R^N \times \cG)$.
\end{itemize}
\end{lemma}

We recall that $Z_{1/l(\cG), p}$ denotes the semi-trivial ground state of mass $1/l(\cG)$ for $E_p(\cdot, \R^N)$, regarded as a function in $\R^N \times \cG$ (see Remark \ref{soluzione triviale}). From now on, the proof differs substantially from the one in \cite{TerTzvVis}. The idea is to show that not only $\|\partial_y u_n\|_2^2 \to 0$ as $n \to \infty$, but the equality holds for sufficiently large $n$. To this end, we first prove that this property holds in the particular case when $\cG=[0,L]$ is an interval. Afterwards, we cover the general case with a rearrangement argument. For the first step to succeed, we need the following preliminary result.

\begin{lem}\label{u in W2,2}
	Under the hypotheses of Proposition \ref{thm 3.1}, let $\cG=[0,L]$, with $L> 0$. Then $u_n \in W^{2,2}(\R^N \x [0,L])$ for every $n$.
\end{lem}
\begin{proof}
Since we fix the index $n$ throughout the proof, we omit it to simplify the notation. Moreover, without loss of generality, we assume that $L=1$. By standard regularity theory, $u \in L^\infty(\R^N \times [0,1])$ and hence
\[
-\Delta_x u -\lambda \pa_{yy} u = f \in L^2(\R^N \times [0,1]).
\]
Let $\{e_k: k \ge 0\}$ be the orthonormal basis of $L^2([0,1])$ consisting of eigenfunctions for the homogeneous Neumann Laplacian, that is $e_0 \equiv 1$ and $e_k(y)=\sqrt{2}\cos(k\pi y)$ for $k \geq 1$, with eigenvalues $\sigma_k:=(k\pi)^2$, and expand
\[
u(x,y) = \sum_{k=0}^\infty u_k(x) e_k(y), \qquad f(x,y) = \sum_{k=0}^\infty f_k(x) e_k(y).
\]
Taking $\varphi(x,y)=\psi(x)e_k(y)$, with $\psi \in C^\infty_c(\R^N)$, as test function in the equation for $u$, and integrating by parts in $y$ (recall that $e_k'(0)=e_k'(1)=0$), we obtain
\[
\int_{\R^N} \nabla u_k \cdot \nabla \psi\,dx +\lambda \sigma_k \int_{\R^N} u_k \psi\,dx = \int_{\R^N} f_k \psi\,dx,
\]
that is, each Fourier coefficient $u_k \in H^1(\R^N)$ solves a linear equation with right hand side $f_k$. Applying the Fourier transform in the $x$ variable,
\[
(|\xi|^2+\lambda \sigma_k)\mathcal{F}u_k(\xi)= \mathcal{F} f_k(\xi) \qquad \text{for a.e. } \xi \in \R^N,
\]
and therefore, for every $i,j$ and every $k$,
\[
\left|\xi_i\xi_j\mathcal{F} u_k(\xi)\right| \le |\mathcal{F} f_k(\xi)|, \qquad
\sigma_k\left|\mathcal{F} u_k(\xi)\right| \le \frac1\lambda\left|\mathcal{F} f_k(\xi)\right|, \qquad
|\xi_i|\sqrt{\sigma_k}\left|\mathcal{F} u_k(\xi)\right| \le \frac{1}{2\sqrt{\lambda}}\left|\mathcal{F} f_k(\xi)\right|,
\]
the three inequalities following respectively from $|\xi|^2 \leq |\xi|^2+\lambda\sigma_k$, from $\lambda\sigma_k \leq |\xi|^2+\lambda\sigma_k$, and from $2\sqrt{\lambda}\,|\xi|\sqrt{\sigma_k}\leq |\xi|^2+\lambda\sigma_k$. By Plancherel in $x$ and Parseval in $y$ we conclude that
\[
\|\partial_{x_ix_j}u\|_{L^2(\R^N \x \cG)} \leq \|f\|_{L^2(\R^N \x \cG)}, \qquad
\|\partial_{yy}u\|_{L^2(\R^N \x \cG)} \leq \frac{1}{\lambda}\|f\|_{L^2(\R^N \x \cG)},
\]
and
\[ 
\|\partial_{x_iy}u\|_{L^2(\R^N \x \cG)} \leq \frac{1}{2\sqrt{\lambda}}\|f\|_{L^2(\R^N \x \cG)},
\]
so that all second-order weak derivatives of $u$ belong to $L^2(\R^N \x \cG)$. Since $u \in H^1(\R^N \x \cG)$ by assumption, the thesis follows.
\end{proof}

\begin{lem}
	\label{lemma 3.6}
	Let $\cG$ be a compact metric graph, $N\geq 1$, $p \in (2,2+4/(N+1)]$, and let $u_n$ be defined as above. Then
	\[
\partial_y u_n \equiv 0 \qquad \text{in $\R^N \times \cG$ for sufficiently large $n$}.
\]
\end{lem}
\begin{proof}
We split the proof into three steps. The first two steps are devoted to proving the theorem for the case $\cG=[0,L]$, while the last step deals with general compact metric graphs $\cG$. \\
\textit{First step: definition and basic properties of $w_n$ when $\mathcal{G}=[0,L]$.} Let $w_n(x,y):=\partial_y u_n(x,y)$. By Lemma \ref{u in W2,2}, we have $w_n \in H^1(\R^N \x [0,L])$. Moreover, since $u_n$ satisfies the homogeneous Neumann condition at the vertices of the segment $[0,L]$ (as a consequence of the Kirchhoff condition), $w_n$ satisfies the homogeneous Dirichlet condition
\[
w(x,0)=w(x,L)=0 \qquad \forall x \in \R^N.
\]
Then, by the Poincar\'e inequality on the segment
\begin{equation*}
	\exists C>0 \text{ such that}\quad \|w_n(x, \cdot)\|_{L^2([0,L])}\leq C \|\partial_y w_n(x, \cdot)\|_{L^2([0,L])} \qquad \forall x \in \R^N, 
\end{equation*}
and, integrating over $x \in \R^N$, we deduce that 
\begin{equation}
	\label{poincare inequality}
	\exists C>0 \text{ such that } \|w_n\|_2\leq C\|\partial_y w_n\|_2.
\end{equation}
\textit{Second step: $w_n \equiv 0$ for sufficiently large $n$ when $\cG=[0,L]$.} Differentiating \eqref{eq diff thm 3.1} with respect to $y$, we deduce that $w_n$ is a weak solution of 
\[
	-\Delta_x w_n- \lambda_n \partial_y^2 w_n+\omega_n w_n-(p-1)|u_n|^{p-2}w_n=0.
\]
 Multiplying by $w_n$ and integrating over $\R^N \x [0,L]$, we obtain:
	\begin{equation*}
		\begin{split}
			\|\nabla_x w_n\|_2^2+\lambda_n \|\partial_y w_n\|_2^2+\omega_n\|w_n\|_2^2-(p-1)\int_{\R^N}\int_{[0,L]}|u_n|^{p-2}w_n^2\,dy\,dx =0,
		\end{split}
	\end{equation*}
where we used the homogeneous Dirichlet boundary condition and the fact that $w_n \in H^1(\R^N \times [0,L])$, by Lemma \ref{u in W2,2}. By the Poincar\'e inequality \eqref{poincare inequality}, 
\begin{equation*}
	\begin{split}
		0&=\|\nabla_x w_n\|_2^2+\|\partial_y w_n\|_2^2+\overline{\omega}\|w_n\|_2^2+(\lambda_n-1) \|\partial_y w_n\|_2^2+(\omega_n-\overline{\omega})\|w_n\|_2^2 \\
		&\qquad -(p-1)\int_{\R^N}\int_{[0,L]}|Z_{\frac{1}{l(\cG)},p}|^{p-2}w_n^2\,dy\,dx+(p-1)\int_{\R^N}\int_{[0,L]}\left(|Z_{\frac{1}{l(\cG)},p}|^{p-2}-|u_n|^{p-2}\right)w_n^2\,dy\,dx \\
		&\geq C\|w_n\|_{H^1}^2+C(\lambda_n-1)\|w_n\|_2^2+o(1)\|w_n\|_2^2-(p-1)\int_{\R^N}\int_{[0,L]}|Z_{\frac{1}{l(\cG)},p}|^{p-2}w_n^2\,dy\,dx\\
		&\qquad +(p-1)\int_{\R^N}\int_{[0,L]}\left(|Z_{\frac{1}{l(\cG)},p}|^{p-2}-|u_n|^{p-2}\right)w_n^2\,dy\,dx,
	\end{split}
\end{equation*}
	We estimate the last two integrals. For the first one, we proceed as follows:
	\[
	\int_{\R^N}\int_{[0,L]}|Z_{\frac{1}{l(\cG)},p}|^{p-2}w_n^2\,dy\,dx \leq \|Z_{\frac{1}{l(\cG)},p}\|_{\infty}^{p-2}\|w_n\|_2^2 \leq C\|w_n\|_2^2.
	\] 
	The second integral can be treated through the H\"older inequality, with conjugate exponents $\frac{2(N+1)}{(N-1)(p-2)}$ and $\frac{2(N+1)}{4N-p(N-1)}$ (if $N=1$, we consider $+\infty$ and $1$): from this, we deduce that
	\begin{equation*}
		\begin{split}
			\int_{\R^N}\int_{[0,L]}\left(|Z_{\frac{1}{l(\cG)},p}|^{p-2}-|u_n|^{p-2}\right)w_n^2\,dy\,dx &\leq \left\| |Z_{\frac{1}{l(\cG)},p}|^{p-2}-|u_n|^{p-2}\right\|_{\frac{2(N+1)}{(N-1)(p-2)}}\|w_n\|_{\frac{4(N+1)}{4N-p(N-1)}}^2 \\
			&= o(1)\|w_n\|_{\frac{4(N+1)}{4N-p(N-1)}}^2 =o(1) \|w_n\|_{H^1}^2, 
		\end{split}
	\end{equation*}
	where we used the Sobolev inequality, the fact that $u_n \to Z_{1/l(\cG),p}$ strongly in $H^1(\R^N \times \cG)$, and the fact that $2<\frac{4(N+1)}{4N-p(N-1)}<\frac{2(N+1)}{N-1}$. Combining the previous estimates, we obtain:
	\[
	0 \geq \|w_n\|_{H^1}^2\left(C+o(1)\right)+\|w_n\|_2^2\left(C(\lambda_n-1)-C\right) \geq \|w_n\|_2^2\left(C(\lambda_n-1)-C\right).
	\]
	Since $\lambda_n \rightarrow +\infty$, $\|w_n\|_2^2$ must be equal to zero, for $n$ sufficiently large. Therefore, \eqref{tesi lemma 3.6} is proved for $\cG=[0,L]$. 
	Moreover, as byproduct, we also infer that
	\begin{equation}
		\label{energie maggiore}
		\cE_{p, \lambda_n}\left(\R^N \x [0, L]\right) = L\cE_p\left(\frac{1}{L}, \R^N\right) \quad \text{for sufficiently large $n$}. 
	\end{equation}
\textit{Third step: conclusion of the proof.} Now, let $\cG$ be a general compact metric graph. By \eqref{energie minore}, the properties of rearrangements, and recalling that $u_n$ is the ground state of mass $1$ for $E_{p, \lambda_n}$, 
\[
l(\cG)\cE_p\left(\frac{1}{l(\cG)}, \R^N\right) \ge \cE_{p, \lambda_n}(\R^N \x \cG) = E_{p, \lambda_n}(u_n, \R^N \x \cG) \geq E_{p, \lambda_n}(u_n^{*y}, \R^N \x [0,l(\cG)]).
\]
On the other hand, since $u_n^{*y} \in H^1_1(\R^N \x [0,l(\cG)])$, if $n$ is so large that \eqref{energie maggiore} holds, we find 
\[
E_{p, \lambda_n}(u_n^{*y}, \R^N \x [0,l(\cG)]) \geq \cE_{p, \lambda_n}(\R^N \x [0, l(\cG)]) =  l(\cG)\cE_p\left(\frac{1}{l(\cG)}, \R^N\right).
	\]
	Therefore, for any such $n$ all the inequalities must be equalities; in particular $u_n^{*y}$ is a ground state for $\cE_{p, \lambda_n}(\R^N \x [0, l(\cG)])$, and by step 2 is independent on $y$. It follows that $u_n$ is also independent on the $y$-variable. Suppose, by contradiction, that this is not the case. Then there exist $x \in \R^N$, and $t \geq 0$ such that the superlevel set
		\[
		\{y \in \cG : u_n(x,y)>t\}
		\]
		is neither empty or equal to $\cG$, namely,
		\[
		\mathcal{L}^1\left(\{y \in \cG : u_n(x,y)>t\}\right) \in (0, l(\cG)),
		\]
		where $\mathcal{L}^1$ is the Lebesgue measure on $\cG$ (we can argue at a pointwise level since each $u_n$ is continuous). By equimeasurability,
		\[
		\mathcal{L}^1\left(\{y \in \cG^* : u_n^{*y}(x,y)>t\}\right) \in (0, l(\cG)).
		\]
		This is impossible, since $u_n^{*y}(x, y)$ is constant on $\cG^*$, for every fixed $x \in \R^N$. Therefore, $u_n$ is independent on the $y$-variable, completing the proof.
\end{proof}
At this point Proposition \ref{thm 3.1} follows from Lemma \ref{lemma 3.6}.
\begin{proof}[Proof of Proposition \ref{thm 3.1}] The argument follows exactly the proof of \cite[Theorem 3.1]{TerTzvVis}, so we only sketch the main steps. Define
		\[
		\lambda^*_{\R^N \x \cG,p}:=\inf\{\lambda>0 : \text{each ground state of mass $1$ of } E_{p, \lambda}(\cdot, \R^N \x \cG) \text{ is such that } \|\partial_y u\|_2^2=0\}.
		\]
	The proof proceeds in three steps. \\
	\textit{First step: $\lambda^*_{\R^N \x \cG,p}<+\infty$, and $\pa_y u_\lambda \equiv 0$ for every mass-one ground state of $E_{p,\lambda}(\cdot,\R^N \times \cG)$ if $\lambda>\lambda^*_{\R^N \x \cG,p}$.} This follows directly from Lemma~\ref{lemma 3.6}. \\
	\textit{Second step: $\lambda^*_{\R^N \x \cG,p}>0$.} Let $\rho(y) \in H^1(\cG)$ be a nonconstant strictly positive function on $\cG$ satisfying $\int_{\cG}\rho^2\, dy=1$, and let
	\[
	\psi(x,y):=\rho(y)^{\frac{4}{4-N(p-2)}}Z_{1,p}\left(\rho(y)^{\frac{2(p-2)}{4-N(p-2)}}x\right).
	\] 
	Then $\|\psi\|_2=1$ and
	\[
	\begin{split}
		\frac{1}{2}\|\nabla_x \psi\|_2^2-\frac{1}{p}\|\psi\|_p^p  &= \left(\int_{\cG} \rho^{2+\frac{4(p-2)}{4-N(p-2)}}\right) E_p(Z_1,\R^N)\\
		& <l(\cG)^{1-\left(1+\frac{2(p-2)}{4-N(p-2)}\right)} \cE_p\left(1, \R^N\right) = l(\cG) \cE_p\left(\frac{1}{l(\cG)}, \R^N\right),
		\end{split}
	\]
	where in the last step we used the Jensen inequality (the inequality is strict since $\rho$ is non-constant and $1+\frac{2(p-2)}{4-N(p-2)}>1$), the fact that $E_p(Z_1,\R^N)=\cE_p(1, \R^N)<0$, and \eqref{scalamento energia}. Consequently,
	\[
	\lim\limits_{\lambda \rightarrow 0^+}\cE_{p, \lambda}(\R^N \x \cG) \leq 	\lim\limits_{\lambda \rightarrow 0^+}E_{p, \lambda}(\psi, \R^N \x \cG) = \frac{1}{2}\|\nabla_x \psi\|_2^2-\frac{1}{p}\|\psi\|_p^p<l(\cG)\cE_p\left(\frac{1}{l(\cG)}, \R^N\right),
	\]
	and hence any mass-one ground state of
	$E_{p,\lambda}(\cdot,\R^N\times\cG)$ depends nontrivially on the
	$y$-variable for $\lambda$ sufficiently small. \\
	\textit{Third step: if $\lambda < \lambda^*_{\R^N \x \cG,p}$, then every mass-one ground state of $E_{p, \lambda}(\cdot, \R^N \x \cG)$ depends on the $y$-variable.} Assume by contradiction that there exists $\lambda_1 <\lambda^*_{\R^N \x \cG,p}$ and a mass-one ground state $v_1$ of $E_{p, \lambda_1}(\cdot, \R^N \x \cG)$ which is independent on the $y$-variable. Then $v_1$ is a semi-trivial solution and 
	\begin{equation}\label{uguaglianza energie}
	\cE_{p, \lambda_1}(\R^N \x \cG)=l(\cG)\cE_p\left(\frac{1}{l(\cG)}, \R^N\right).
	\end{equation}
	On the other hand, by the definition of $\lambda^*_{\R^N \x \cG,p}$, there exists $\lambda_2 \in (\lambda_1, \lambda^*_{\R^N \x \cG,p})$ and a mass-one ground state $v_2$ of $E_{p, \lambda_2}(\R^N \x \cG)$ which depends on the $y$-variable. Recalling \eqref{energie minore}, we obtain
	\[
	\cE_{p, \lambda_1}(\R^N \x \cG) \leq E_{p, \lambda_1}(v_2, \R^N \x \cG) < E_{p, \lambda_2}(v_2, \R^N \x \cG) \leq l(\cG)\cE_p\left(\frac{1}{l(\cG)}, \R^N\right),
	\]
	which contradicts \eqref{uguaglianza energie}.
\end{proof}

\begin{proof}[Proof of Theorem \ref{thm 1.3}]
We only provide a sketch of the proof of this statement, which is identical to \cite[Theorem 1.3]{TerTzvVis}. 
	Considering the scaling $T_{\mu}$ defined in \eqref{trasformazione thm 3.1} and its properties stated in \eqref{proprieta trasformazione thm 3.1}, $u_{\mu} \in H^1_{\mu}(\R^N \x \cG)$ is a ground state of $E_p(\cdot, \R^N \x \cG)$ if and only if $T_{\mu}(u_{\mu})$ is a mass-one ground state of $E_{p, \lambda_{\mu}}(\cdot, \R^N \x \cG)$, where
	\[
	\lambda_{\mu}:=\mu^{\frac{-2(p-2)}{4-N(p-2)}}.
	\]
	Furthermore, $T_{\mu}(u_{\mu})$ is independent on the $y$-variable if and only if $u_{\mu}$ is. Proposition \ref{thm 3.1} therefore implies that $u_{\mu}$ is semi-trivial whenever $\lambda_{\mu}>\lambda^*_{\R^N \x \cG, p}$, whereas it depends on the $y$-variable whenever
	$\lambda_{\mu}<\lambda^*_{\R^N \x \cG, p}$. Setting 
	\begin{equation}\label{da lambda a mu}
		\mu_2\left(\R^N \x \cG,p\right):=\left(\lambda^*_{\R^N \x \cG,p}\right)^{-\frac{4-pN+2N}{2(p-2)}}
	\end{equation}
	and observing that $-2(p-2)/(4-N(p-2))<0$ concludes the proof.
\end{proof}

\begin{remark}
The scaling \eqref{trasformazione thm 3.1} also provides a variational characterization of $\mu_2\left(\R^N \x \cG,p\right)$. Indeed, by \eqref{proprieta trasformazione thm 3.1}, $\mu_2\left(\R^N \x \cG,p\right)$ is the largest mass for which $E_{p,\lambda_\mu}(\cdot, \R^N \x \cG)$, with $\lambda_\mu:=\mu^{-2(p-2)/(4-N(p-2))}$, admits a semi-trivial ground state of mass one, i.e. for which $\cE_{p,\lambda_\mu}(\R^N \x \cG)=l(\cG)\cE_p\left(1/l(\cG), \R^N\right)$. Equivalently, $\lambda^*_{\R^N \x \cG,p}$ is the smallest $\lambda$ for which $E_{p,\lambda}(u,\R^N \x \cG) \geq l(\cG)\cE_p\left(1/l(\cG),\R^N\right)$ for every $u \in H^1_1(\R^N \x \cG)$. Since the inequality is trivially satisfied when $\partial_yu \equiv 0$, solving it for $\lambda$ on the remaining functions gives
	\[
	\lambda^*_{\R^N \x\cG,p}= \sup\limits_{\substack{u \in H^1_1(\R^N \x \cG) \\ \partial_y u \not\equiv 0}}\frac{l(\cG)\cE_p\left(\frac{1}{l(\cG)}, \R^N\right)-\frac{1}{2}\|\nabla_x u\|_2^2+\frac{1}{p}\|u\|_p^p}{\frac{1}{2}\|\partial_y u\|_2^2},
	\]
	and hence, by \eqref{da lambda a mu}, $\mu_2\left(\R^N \x \cG,p\right)=\left(\lambda^*_{\R^N \x\cG,p}\right)^{-\frac{4-N(p-2)}{2(p-2)}}$.
\end{remark}

\section{The supercritical case}\label{supercritical case}
We now address the problem of existence of solutions of \eqref{schroedinger equation} for $p \in (2+4/(N+1), 2+4/N)$. This range is $L^2$-supercritical for the NLS equation in $\R^N \times \cG$ and subcritical for the NLS equation in $\R^N$; therefore, $E_p(\cdot, \R^N \x \cG)$ does not admit global minimizers in $H^1_{\mu}(\R^N \x \cG)$, and we search for solutions as local minimizers, in the spirit of \cite{PieVerYu}, where the problem is studied on $\R^N \x M^k$, being $M^k$ a $k$-dimensional compact Riemannian manifold. Actually, our method is slightly different and closer to the one introduced in \cite{BelBouJeaVis} to deal with the NLS equation with partial confinement. We often use the Gagliardo-Nirenberg inequality \eqref{GN con costanti}; more precisely, we fix arbitrarily $B>0$, and we denote $K:=K_{\R^N \x \cG, p, B}$. 
\begin{definition}
	We introduce the following family of open sets in $H^1_{\mu}$:
	\[
	U_{\mu, t}:=\left\{u \in H^1\left(\R^N \x \cG\right) \text{ such that } \|u\|_2^2=\mu, \, \|\nabla_{x,y} u\|_2^2<t\right\}.
	\]
\end{definition}

\begin{lemma}\label{esistenza minimo locale Jea}
Let $\cG$ be a compact metric graph, $t>0$, $N\geq 1$, $p<2+4/N$. Moreover, assume that there exists $\mu^* \in (0,1]$ such that 
\[
\inf\limits_{U_{\mu, t}}E_p(\cdot, \R^N \x \cG)	< \inf\limits_{U_{\mu, t}\setminus U_{\mu, t\mu}}E_p(\cdot, \R^N \x \cG) \qquad \text{for every }\mu<\mu^*. 
\]
Then, for every $\mu<\mu^*$ there exists a mass-constrained local minimizer of $E_p(\cdot, \R^N \x \cG)$ of mass $\mu$, which is a global minimizer in the open set $U_{\mu,t}$.
\end{lemma}
\begin{proof} The proof is divided into the following steps. \\
\textit{First step: $\inf_{U_{\mu,t}} E_p(\cdot, \R^N \x \cG)<0$.} Let $Z(x,y):=Z_{\mu/l(\cG),p}(x) \in H^1_{\mu}(\R^N \x \cG)$ be the semi-trivial solution of \eqref{schroedinger equation} of mass $\mu$, defined in Remark \ref{soluzione triviale}. Then, for $s>0$ sufficiently small we have
\[
s \star Z(x,y):=s^{\frac{N}{2}}Z(sx,y) \in U_{\mu, t},
\]
since $\left\|s \star Z\right\|_{L^2(\R^N \x \cG)}^2=\mu$ for every $s>0$ and
\begin{equation}\label{condizione sulla s}
\left\|\nabla_x \left(s \star Z\right)\right\|_{L^2(\R^N \x \cG)}^2=s^2 l(\cG) \left\| \nabla_x Z\right\|_{L^2(\R^N)}^2<t \qquad \text{for } s<\left(\frac{t}{l(\cG)\left\|\nabla_x Z\right\|_{2}^2}\right)^{\frac{1}{2}}.
\end{equation}
Moreover, by trivial rescalings and the Pohozaev's identity \eqref{identita di Pohozaev}:
\[
\begin{split}
E_p\left(s\star Z, \R^N \x \cG\right)&=\frac{l(\cG)}{2}s^2\left\|\nabla_x Z\right\|_2^2-\frac{l(\cG)}{p}s^{\frac{N(p-2)}{2}}\left\|Z\right\|_p^p\\
&=\frac{l(\cG)}{2}s^2\left\|\nabla_x Z\right\|_2^2 \left(1-\frac{4}{N(p-2)}s^{\frac{N(p-2)}{2}-2}\right).
\end{split}
\]
Thus, 
\[
\inf\limits_{U_{\mu,t}}E_p(\cdot, \R^N \x \cG) \leq E_p(s\star Z)=\frac{l(\cG)}{2}s^2\|\nabla_x Z\|_2^2 \left(1-\frac{4}{N(p-2)}s^{\frac{N(p-2)}{2}-2}\right)
\]	
for every $s$ satisfying \eqref{condizione sulla s}, and the last quantity is negative for $s$ sufficiently small, as $N(p-2)/2-2<0$. This proves the first step. \\

\textit{Second step: properties of minimizing sequences.} Let $\{u_n\}\subseteq U_{\mu,t}$ be a nonnegative minimizing sequence for $E_p(\cdot, \R^N \x \cG)$ such that each $u_n$ is radially decreasing with respect to the $x$-variable. By the hypothesis, $\{u_n\}\subseteq U_{\mu,t\mu}$. 
Since $\|u_n\|_{H^1}$ is bounded, up to a subsequence $u_n \rightharpoonup u$ for some $u \in H^1(\R^N \x \cG)$, and 
\[
\|u\|_2^2 \leq \mu \qquad \text{ and } \qquad \left\|\nabla_{x,y}u\right\|_2^2 \leq t\mu.
\] 
Additionally, since $\inf_{U_{\mu,t}}E_p(\cdot, \R^N \x \cG) <0$, there exists a constant $C>0$ such that 
\[
\|u_n\|_p^p \geq C, \qquad \text{ for every } n\in \mathbb{N}.
\]
By the localized Gagliardo-Nirenberg inequality \eqref{GN localizzata} and arguing exactly as in the proof of the second step of Theorem \ref{caso sottocritico e critico}, we deduce that $u \not\equiv 0$ (this time, we use \eqref{GN localizzata} for $p>2+4/(N+1)$). We aim at proving that $\|u\|_2^2=\mu$. \\
\textit{Third step: the function $s \in (0, \mu^*] \mapsto \left(\inf_{U_{s,t}}E_p(\cdot, \R^N \x \cG)\right)/s$ is strictly decreasing.} Let $r<s\leq \mu^*$ and let $\{v_n\} \subseteq U_{r,t}$ be a nonnegative minimizing sequence for $E_p(\cdot, \R^N \x \cG)$, radially decreasing in $x$. By the second step, $\{v_n\} \subseteq U_{r, rt}$ and $\|v_n\|_p^p \geq C>0$; consequently $\{\sqrt{s/r}\,v_n\}\subseteq U_{s, st}\subseteq U_{s,t}$, and since
\[
E_p\left(\sqrt{\tfrac{s}{r}}\, v_n, \R^N \x \cG\right)=\frac{s}{r}E_p(v_n, \R^N\x\cG)-\left(\left(\frac{s}{r}\right)^{\frac{p}{2}}-\frac{s}{r}\right)\frac{\|v_n\|_p^p}{p},
\]
passing to the $\liminf$ we obtain $\inf_{U_{s,t}}E_p \leq \frac{s}{r}\inf_{U_{r,t}}E_p-\left(\left(\frac sr\right)^{p/2}-\frac sr\right)\frac Cp<\frac{s}{r}\inf_{U_{r,t}}E_p$.\\
\textit{Fourth step: $u$ is the desired local minimizer.} If $\|u\|_2^2=\tau<\mu$, then $u \in U_{\tau,t}$ and, by the previous step applied with $r=\tau$ and $s=\mu$,
\[
\inf\limits_{U_{\tau, t}}E_p \leq E_p(u, \R^N \x \cG)\leq \liminf_n E_p(u_n, \R^N \x \cG) =\inf\limits_{U_{\mu, t}}E_p< \frac{\mu}{\tau}\inf\limits_{U_{\tau, t}}E_p,
\]
which contradicts $\inf_{U_{\tau, t}}E_p<0$. Hence $\|u\|_2^2=\mu$ and, arguing as in the fifth step of Theorem \ref{caso sottocritico e critico}, $\inf_{U_{\mu,t}}E_p(\cdot, \R^N \x\cG) = E_p(u, \R^N \x\cG)$. Since $\|\nabla_{x,y}u\|_2^2 \leq t \mu<t$, the function $u$ is an interior point of $U_{\mu,t}$, and is therefore the desired local minimizer.
\end{proof}

Theorem \ref{minimo locale intro} is a direct corollary of the next result.

\begin{thm}
	\label{minimo locale}
	Let $\cG$ be a compact metric graph, $N \geq 1$, $p \in \left(2+4/(N+1), 2+4/N\right)$. Then, there exists a critical mass $\mu^*_{ex}\left(\R^N \x \cG, p\right)>0$ such that, for every $0<\mu<\mu_{ex}^*\left(\R^N \x \cG, p\right)$, the functional $E_p(\cdot, \R^N \x \cG)$ admits a local minimum $u_{\mu}$ in $H^1_{\mu}$. More precisely, there exists $t^*>0$ such that $u_{\mu}$ is a minimum in $U_{\mu, t^*\mu}$ for every $\mu \in (0, \mu^*_{ex})$. Moreover, we have the estimate
	\begin{equation}\label{stima mex}
		\mu_{ex}^*\left(\R^N \x \cG, p\right) \ge \left(\frac{p\left(1-\frac{2}{\theta(p)}\right)^{\frac{\theta(p)}{2}-1}}{\theta(p)KB^{\frac{\theta(p)}{2}-1}}\right)^{\frac{2}{p-2}}l(\cG)^{-\frac{4}{p-2}+N+1}.
	\end{equation}
\end{thm}
\begin{proof}
	\textit{First step: application of Lemma \ref{esistenza minimo locale Jea}.} We fix $t>0$ and search $\tilde{\mu}_{ex}\left(\R^N \x \cG,p\right)\in (0,1]$ such that, for every $\mu <\tilde{\mu}_{ex}\left(\R^N \x \cG,p\right)$,
\begin{equation}
\label{condizione da verificare}
\inf\limits_{U_{\mu, t}}E_p(\cdot, \R^N \x \cG)	<0< \inf\limits_{U_{\mu, t}\setminus U_{\mu, t\mu}}E_p(\cdot, \R^N \x \cG).
\end{equation}
By the Gagliardo-Nirenberg inequality \eqref{GN con costanti}, for every $u \in U_{\mu, t}\setminus U_{\mu, t\mu}$, 
\[
\begin{split}
E_p(u, \R^N \x \cG) &\geq \frac{1}{2}\|\nabla_{x,y}u\|_2^2-\frac{K}{p}\mu^{\frac{p-\theta(p)}{2}}\left(\|\nabla_{x,y}u\|_2^2+\frac{B\mu}{l(\cG)^2}\right)^{\frac{\theta(p)}{2}} \\
&=\frac{1}{2}\|\nabla_{x,y}u\|_2^2\left(1-\frac{2K}{p}\mu^{\frac{p-\theta(p)}{2}}\left(1+\frac{B\mu}{l(\cG)^2\|\nabla_{x,y}u\|_2^2}\right)^{\frac{\theta(p)}{2}}\|\nabla_{x,y}u\|_2^{\theta(p)-2}\right) \\
&\geq \frac{1}{2}\|\nabla_{x,y}u\|_2^2\left(1-\frac{2K}{p}\mu^{\frac{p-\theta(p)}{2}}\left(1+\frac{B}{l(\cG)^2t}\right)^{\frac{\theta(p)}{2}}t^{\frac{\theta(p)}{2}-1}\right).
\end{split}
\]
The right hand side is positive for $\mu<\mu'(t)$. In order to obtain a larger threshold for $\mu$, we choose $t=t_1$ defined by
\[
t_1:=\frac{2B}{l(\cG)^2\left(\theta(p)-2\right)} = \mathrm{argmin}\left\{ \left(1+\frac{B}{l(\cG)^2t}\right)^{\frac{\theta(p)}{2}}t^{\frac{\theta(p)}{2}-1}: \ t>0\right\}.
\]
In this way, the term in the brackets is positive if 
\[
\mu< \left(\frac{p\left(1-\frac{2}{\theta(p)}\right)^{\frac{\theta(p)}{2}-1}}{\theta(p)KB^{\frac{\theta(p)}{2}-1}}\right)^{\frac{2}{p-\theta(p)}}l(\cG)^{\frac{2\theta(p)-4}{p-\theta(p)}}=: \mu'(\R^N \x \cG,p).
\]
Thus, the thesis follows for $\tilde{\mu}_{ex}(\R^N \x \cG,p):= \min\{1, \mu'(\R^N \x \cG,p)\}$.\\
\textit{Second step: improvement of the estimate for $\mu_{ex}^*(\R^N \x \cG, p)$.} In the first step, we showed the existence of $\mu_{ex}^*(\R^N \times \cG,p)$ and obtained the estimate $\mu_{ex}^*(\R^N \times \cG,p) \ge \tilde{\mu}_{ex}(\R^N \times \cG,p)$. However, recalling the definition of $u_L$ given in Remark \ref{remark scaling}, it is possible to deduce a relation between $\mu_{ex}^*(\R^N \x \cG,p)$ and $\mu_{ex}^*(\R^N \x L\cG,p)$; exploiting this relation, in this step we obtain a stronger lower bound.

Let
\[
C:=\frac{p\left(1-\frac{2}{\theta(p)}\right)^{\frac{\theta(p)}{2}-1}}{\theta(p)KB^{\frac{\theta(p)}{2}-1}}.
\] 
Choose $L\geq 0$ such that, for the scaled graph $L \cG$,
\[
\left(\frac{p\left(1-\frac{2}{\theta(p)}\right)^{\frac{\theta(p)}{2}-1}}{\theta(p)KB^{\frac{\theta(p)}{2}-1}}\right)^{\frac{2}{p-\theta(p)}}l(L\cG)^{\frac{2\theta(p)-4}{p-\theta(p)}}=1 \iff L=\frac{1}{l(\cG)C^{\frac{1}{\theta(p)-2}}}.
\]
Thus, $\tilde{\mu}_{ex}\left(\R^N \x L\cG,p\right)$=1. We now apply the scaling \eqref{scaling} from \(L\cG\) to \(\cG\), namely with scaling factor \(L^{-1}\). By Remark \ref{remark scaling}, the existence of a local minimizer of $E_p(\cdot,\R^N\times L\cG)$ for every
$\mu<\tilde{\mu}_{ex}(\R^N\times L\cG,p)$, which is global in
$U_{\mu,t_1\mu}$, implies the existence of a local minimizer of $E_p(\cdot,\R^N\times\cG)$ for every
$\mu<\mu_{ex}^*(\R^N\times\cG,p)$ which is global in
$U_{\mu,t^*\mu}$, where
\[
t^*=L^2t_1,
\qquad
\mu_{ex}^*(\R^N\times\cG,p)
=
\tilde{\mu}_{ex}(\R^N\times L\cG,p)
L^{\frac{4}{p-2}-N-1}
=
L^{\frac{4}{p-2}-N-1}.
\]
Thus, 
\[
\mu_{ex}^*\left(\R^N \x \cG,p\right)=\left(\frac{1}{l(\cG)C^{\frac{1}{\theta(p)-2}}}\right)^{\frac{4}{p-2}-N-1}=C^{\frac{2}{p-2}}l(\cG)^{-\frac{4}{p-2}+N+1},
\]
which is the claimed critical mass.
\end{proof}
\begin{remark}
The thesis of Theorem \ref{minimo locale} is analogous to that of \cite[Theorem 1.2, part 1]{PieVerYu}, but the proof is rather different. The strategy developed here can be used also in the context of \cite{PieVerYu}, to simplify the proof. The estimate for $\mu_{ex}^*$ found here is coherent with the threshold found in the proof of \cite[Lemma 3.2]{PieVerYu}. 
\end{remark}	

In the remainder of this section we investigate the semi-triviality of the local minimizer found in Theorem \ref{minimo locale}. This is motivated by the following preliminary statement.
\begin{lem}\label{lemma stima minimo locale}
	Under the assumptions of Theorem \ref{minimo locale}, $Z_{\mu/l(\cG),p} \in U_{\mu, t^*\mu}$ for every $\mu \in (0, \mu^*_{ex}(\R^N \x \cG,p))$.
\end{lem}
The proof of this fact is analogous to that of \cite[Lemma 3.4]{PieVerYu}, and is therefore omitted. At this point, it is natural to investigate whether the local minimizers obtained in Theorem \ref{minimo locale} exhibit a nontrivial dependence on the $y$-variable, or whether they are simply given by $Z_{\mu/l(\cG),p}$. We aim at proving Theorem \ref{thm 1.3 sopracritico}.
\begin{proof}[Proof of Theorem \ref{thm 1.3 sopracritico}]
This theorem is adapted from \cite[Theorem 1.2, parts 2 and 3]{PieVerYu}, and the arguments trace step by step \cite[Section 4]{PieVerYu}. For this reason, we only provide a sketch of the proof. \\
We proved that, for some $t^*>0$, there exists a local minimum $u_{\mu} \in U_{\mu, t^*\mu}$ for $E_p(\cdot, \R^N \x \cG)$, for every $\mu<\mu_{ex}^*\left(\R^N \x \cG,p\right)$. By Lemmas \ref{esistenza minimo locale Jea} and \ref{lemma stima minimo locale}, for all such $\mu$
\[
Z_{\frac{\mu}{l(\cG)},p} \in U_{\mu, t^*\mu} \qquad \text{ and }\qquad \inf\limits_{U_{\mu, t^*}}E_p(\cdot, \R^N \x \cG)=\inf\limits_{U_{\mu, t^*\mu}}E_p(\cdot, \R^N \x \cG) \leq l(\cG)\cE_p\left(\frac{\mu}{l(\cG)}, \R^N\right)<0.
\]
Following the arguments of \cite[Corollary 4.2]{PieVerYu}, we define
\[
\mu_2(\R^N \x \cG,p):=\inf\left\{0<\mu<\mu^*_{ex}(\R^N \x \cG, p) \text{ such that } \inf\limits_{U_{\mu, t}}E_p(\cdot, \R^N \x \cG) < l(\cG)\cE_p\left(\frac{\mu}{l(\cG)}, \R^N\right)\right\},
\]
with the convention that $\mu_2(\R^N \x \cG,p)=\mu^*_{ex}(\R^N \x \cG, p)$ if the above set is empty. As claimed in \cite[Lemma 4.1]{PieVerYu}, this mass satisfies the requests of the two points of the thesis, provided that $\mu_2>0$. To prove this last fact, we can adapt the arguments of \cite[Lemmas 4.3 and 4.4]{PieVerYu} without modifications. We introduce the transformation 
\[
u(x,y)=\mu^{\frac{2}{4-N(p-2)}}v\left(\mu^{\frac{p-2}{4-N(p-2)}}x, y\right).
\]
Let $0<\mu<\mu_{ex}^*$, and let 
\[
\lambda=\lambda(\mu):=\mu^{-\frac{2(p-2)}{4-N(p-2)}}.
\]
We assert that $u_{\mu}$ is a minimizer of $E_p(\cdot, \R^N \x \cG)$ in $U_{\mu, t^*\mu}$ if and only if the corresponding $v_{\lambda}$ is a minimizer of 
\[
E_{p,\lambda}(w, \R^N \x \cG):=\frac{1}{2}\|\nabla_x w\|_2^2 + \frac{\lambda}{2}\|\partial_y w\|_2^2-\frac{1}{p}\|w\|_p^p
\]
in 
\[
V_{\lambda, t^*}:=\left\{w \in H^1_{1}(\R^N \times \cG) \text{ such that } \frac{1}{\lambda}\|\nabla_x w\|_2^2 + \|\partial_y w\|_2^2<t^*\right\}.
\]
For each $\lambda_n \rightarrow +\infty$ the sequence $\{v_{\lambda_n}\}$ of minimizers of $E_{p, \lambda_n}$ in $V_{\lambda_n, t}$ satisfies all the properties collected in Lemma \ref{lem: prop v_n}.
The proofs can be adapted from \cite[Lemmas 4.3 and 4.4]{PieVerYu} and \cite[Lemmas 3.3-3.5]{TerTzvVis}. Moreover, proceeding as in the proofs of Lemmas \ref{u in W2,2} and \ref{lemma 3.6}, we conclude that $\|\partial_y v_{\lambda_n}\|_2^2=0$, for $n$ sufficiently large. This fact implies that $\mu_2(\R^N \x \cG,p)>0$, completing the proof. 
\end{proof}

\section{Local minimality of the trivial solution}\label{sez minimalita}
In this section we find a condition for which the trivial solution $Z_{\mu/l(\cG),p}$ is a local minimizer in $H^1_{\mu}(\R^N \x \cG)$. Let us introduce the first positive eigenvalue of the Kirchhoff Laplacian on $\cG$.

\begin{definition}
	\label{primo autovettore}
	For each $\cG$ compact metric graph, we introduce 
	\[
	\alpha_1(\cG):=\inf\left\{\frac{\|w'\|_{L^2(\cG)}^2}{\|w\|_{L^2(\cG)}^2} \text{ such that }w \in H^1(\cG) \setminus\{0\} \text{ and } \int_{\cG}w=0 \right\}
	\]
	the first nonzero eigenvalue of the second derivative on $\cG$. By \cite[Theorem 3.1.1]{intro_graphs}, we know that $\alpha_1>0$. Moreover, this infimum is attained by a normalized eigenfunction $\phi_1 \in H^1(\cG)$ corresponding to $\alpha_1$, characterized by
	\[
	\begin{cases}
			- \phi_1''=\alpha_1\phi_1  &\text{ in }\cG \\
			\sum\limits_{e \text{ incident at } v }\left(\phi_1\right)_e'(v)=0 &\forall v \in \mathcal{V}, 
	\end{cases} \qquad \left\|\phi_1\right\|_{L^2(\cG)}^2=1 \quad \text{and} \quad \left\|\phi_1'\right\|_{L^2(\cG)}^2=\alpha_1(\cG).
	\]
\end{definition}

\begin{remark}\label{rem: EZZ<0}
	In the next theorem, the quantity $E_p'(Z_{\mu/l(\cG),p}, \R^N)(Z_{\mu/l(\cG),p})$ plays a key role. Observe that, by Remark \ref{soluzione triviale},
	\[
	E_p'(Z_{\frac{\mu}{l(\cG)},p}, \R^N)(Z_{\frac{\mu}{l(\cG)},p})=\|\nabla Z_{\frac{\mu}{l(\cG)},p}\|_{L^2\left(\R^N\right)}^2-\|Z_{\frac{\mu}{l(\cG)},p}\|_{L^p(\R^N)}^p=-\cE_p\left(\frac{\mu}{l(\cG)},\R^N\right)\left(\frac{2N(p-2)-4p}{4-N(p-2)}\right),
	\]
	which is negative. 
\end{remark}

\begin{definition}\label{primo autovalore E''}
	Another quantity involved in the next theorem is the first eigenvalue of the quadratic form $E_p''(Z_{\frac{\mu}{l(\cG)},p}, R^N)$, denoted $\beta_1(p, N, \mu / l(\cG))$, which can be expressed variationally as:
	\[
	\beta_1(p, N, \mu / l(\cG))=\inf\limits_{v \in H^1(\R^N) \setminus\{0\}}\left(\frac{E_p''\left(Z_{\frac{\mu}{l(\cG)},p}, \R^N\right)(v, v)}{\|v\|_{L^2(\R^N)}^2}\right).
	\]
\end{definition}

\begin{remark}\label{E''}
	Observe that $\beta_1(p, N, \mu / l(\cG))$ is strictly negative, since 
	\[\begin{split}
		\inf\limits_{v \in H^1(\R^N) \setminus \{0\}}\left(\frac{E_p''(Z_{\frac{\mu}{l(\cG)},p}, \R^N)(v, v)}{\|v\|_{L^2(\R^N)}^2}\right) &\leq \frac{E_p''(Z_{\frac{\mu}{l(\cG)},p}, \R^N)(Z_{\frac{\mu}{l(\cG)},p}, Z_{\frac{\mu}{l(\cG)},p})}{\|Z_{\frac{\mu}{l(\cG)},p}\|_{L^2(\R^N)}^2}\\
		&=\frac{l(\cG)}{\mu}\left(-\cE_p\left(\frac{\mu}{l(\cG)},\R^N\right)\right)\left(\frac{2N(p-2)-4p(p-1)}{4-N(p-2)}\right)<0,
	\end{split}\]
	by Remark \ref{soluzione triviale}. Moreover, such infimum is bounded from below, since
	\[\begin{split}
		\frac{E_p''\left(Z_{\frac{\mu}{l(\cG)},p}, \R^N\right)(v, v)}{\|v\|_{L^2(\R^N)}^2} = \frac{\|\nabla_{x}v\|_2^2-(p-1)\int_{\R^N}\left(Z_{\frac{\mu}{l(\cG)},p}\right)^{p-2}v^2\, dx}{\|v\|_{L^2(\R^N)}^2}\geq -(p-1)\|Z_{\frac{\mu}{l(\cG)},p}\|_{\infty}^{p-2}
	\end{split}\]
	for every $v \in H^1(\R^N) \setminus \{0\}$.
\end{remark}

The next theorem provides necessary and sufficient conditions for $Z_{\mu/l(\cG),p}$ to be a local minimizer.

\begin{proposition}\label{prop mu1}
	Let $\cG$ be a compact metric graph, $N \geq 1$, 
	$p \in \left(2,2+4/N\right)$, and consider the semi-trivial 
	solution $Z:=Z_{\mu/l(\cG),p}$ introduced in 
	Remark \ref{soluzione triviale}. Then:
	\[
	\begin{split}
	\text{($i$) } &\beta_1(p,N,\mu/l(\cG))
		-\frac{l(\cG)E_p'\left(Z,\R^N\right)(Z)}{\mu}
		+\alpha_1(\cG)>0 \quad \implies \quad \text{$Z$ is a local minimizer};\\
	\text{($ii$) } &\beta_1(p,N,\mu/l(\cG))
		-\frac{l(\cG)E_p'\left(Z,\R^N\right)(Z)}{\mu}
		+\alpha_1(\cG)<0 \quad \implies \quad \text{$Z$ is not a local minimizer}.
	\end{split}
	\]
\end{proposition}

Note that the limit case is left open.

\begin{proof}
	Throughout the proof we write $\|\cdot\|_2$ and 
	$(\cdot,\cdot)_2$ for the norm and scalar product of
	$L^2(\R^N\x\cG)$, and we set
	\[
	\kappa:=
	-\frac{l(\cG)E_p'\left(Z,\R^N\right)(Z)}{\mu},
	\qquad
	\Lambda:=
	\beta_1(p,N,\mu/l(\cG))
	+\kappa+\alpha_1(\cG).
	\]
	Thus, hypotheses $(i)$ and $(ii)$ read respectively
	$\Lambda>0$ and $\Lambda<0$. By Remarks
	\ref{rem: EZZ<0} and \ref{E''}, we have $\kappa>0$.

	We introduce the quadratic form
	\begin{equation}\label{forma Q}
		\begin{split}
			Q(u)
			:&=
			E_p''\left(Z,\R^N\x\cG\right)(u,u)
			+\kappa\|u\|_2^2
			\\
			&=
			\left\|\nabla_{x,y}u\right\|_2^2
			+\kappa\|u\|_2^2
			-(p-1)\int_{\R^N\x\cG}Z^{p-2}u^2,
		\end{split}
	\end{equation}
	and denote by $\mathcal B$ its associated symmetric bilinear
	form. Since $Z\in L^\infty(\R^N)$, there exists $M>0$ such that
	\[
	|\mathcal B(u,v)|
	\leq M\|u\|_{H^1}\|v\|_{H^1}
	\qquad
	\forall u,v\in H^1(\R^N\x\cG).
	\]

	Since $Z$ is a constrained critical point of
	$E_p(\cdot,\R^N\x\cG)$ and
	\[
	E_p'\left(Z,\R^N\x\cG\right)(Z)
	=
	l(\cG)E_p'\left(Z,\R^N\right)(Z)
	=
	-\kappa\mu,
	\]
	decomposing
	\[
	\phi=\mu^{-1}(Z,\phi)_2Z+\phi^\perp,
	\qquad
	\phi^\perp\in T_ZH^1_\mu(\R^N\x\cG),
	\]
	we obtain the Lagrange multiplier identity
	\begin{equation}\label{moltiplicatore}
		E_p'\left(Z,\R^N\x\cG\right)(\phi)
		=
		-\kappa(Z,\phi)_2
		\qquad
		\forall\phi\in H^1(\R^N\x\cG).
	\end{equation}

	Consequently, if
	$\gamma:[-1,1]\to H^1_\mu(\R^N\x\cG)$ is a regular curve
	such that $\gamma(0)=Z$ and $\gamma'(0)=u$, differentiating
	twice the identity $\|\gamma(t)\|_2^2=\mu$ gives $(\gamma''(0),Z)_2=-\|u\|_2^2$. Therefore, by \eqref{moltiplicatore},
	\begin{equation}\label{Ep'' da stimare}
		\left.
		\frac{d^2}{dt^2}
		E_p\left(\gamma(t),\R^N\x\cG\right)
		\right|_{t=0}
			=
		E_p''\left(Z,\R^N\x\cG\right)(u,u)
		+
		E_p'\left(Z,\R^N\x\cG\right)(\gamma''(0))
		=
		Q(u).
	\end{equation}

	Notice that $Q$ cannot be coercive on the whole tangent space.
	Indeed, for every $j=1,\ldots,N$, the function $\partial_{x_j}Z\in T_ZH^1_\mu(\R^N\x\cG)$, 	and the curve $t\mapsto Z(\cdot-te_j)$ has constant mass and energy. Hence, by
	\eqref{Ep'' da stimare}, we have $Q(\partial_{x_j}Z)=0$. We shall instead prove coercivity on the subspace of functions
	having zero average with respect to the graph variable, and
	use the global minimality of the Euclidean soliton to control
	the remaining directions.

	\textit{First step: decomposition with respect to the graph
	variable.}
	For $u\in H^1(\R^N\x\cG)$, write
	\[
	u(x,y)=u_1(x)+u_2(x,y),
	\]
	where
	\[
	u_1(x):=
	\frac{1}{l(\cG)}\int_{\cG}u(x,y)\,dy,
	\qquad
	u_2(x,y):=u(x,y)-u_1(x).
	\]
	Then $u_1\in H^1(\R^N)$,
	$u_2\in H^1(\R^N\x\cG)$, and
	\begin{equation}\label{1571}
		\int_{\cG}u_2(x,y)\,dy=0
				\quad\text{for a.e. }x\in\R^N,
		\qquad
		\int_{\cG}\nabla_xu_2(x,y)\,dy=0
		\quad\text{for a.e. }x\in\R^N.
	\end{equation}
	In particular,
	\begin{equation}\label{1572}
		\|u\|_2^2
		=
		l(\cG)\|u_1\|_{L^2(\R^N)}^2+\|u_2\|_2^2.
	\end{equation}
	The analogous orthogonal decomposition holds for the
	Dirichlet integral. Moreover, if
	$u\in T_ZH^1_\mu(\R^N\x\cG)$, then $u_1 \perp Z$ in $L^2(\R^N)$.

	\textit{Second step: positivity of the constrained Hessian.}
	Let $u\in T_ZH^1_\mu(\R^N\x\cG)$ and let
	$u=u_1+u_2$ be the decomposition above. Then
	\begin{equation}\label{termini da stimare}
		Q(u)
		=
		Q(u_1)+2\mathcal B(u_1,u_2)+Q(u_2).
	\end{equation}
	Since $u_1$ is independent on $y$ and is
	$L^2(\R^N)$-orthogonal to $Z$,
	\begin{equation}\label{stima u1}
		\begin{split}
			Q(u_1)
			&=
			l(\cG)
			\left[
			E_p''\left(Z,\R^N\right)(u_1,u_1)
			+\kappa\|u_1\|_{L^2(\R^N)}^2
			\right]
			\geq 0.
		\end{split}
	\end{equation}
	Indeed, the expression in square brackets is the constrained
	second variation of the Euclidean energy at the global minimizer
	$Z$.

	By \eqref{1571}, all mixed terms vanish, including the
	$L^2$ term, and hence
	\[
	\mathcal B(u_1,u_2)=0.
	\]
	Finally, the definition of $\beta_1$, the Poincar\'e inequality
	on $\cG$, and \eqref{1571} give
	\begin{equation}\label{stima u2}
		\begin{split}
			Q(u_2)
			&=
			\int_{\cG}
			E_p''\left(Z,\R^N\right)
			\left(u_2(\cdot,y),u_2(\cdot,y)\right)\,dy
			+
			\|\partial_yu_2\|_2^2
			+
			\kappa\|u_2\|_2^2
			\\
			&\geq
			\left[
			\beta_1(p,N,\mu/l(\cG))
			+\alpha_1(\cG)+\kappa
			\right]\|u_2\|_2^2
			=
			\Lambda\|u_2\|_2^2.
		\end{split}
	\end{equation}
	Thus, if $\Lambda>0$, the constrained Hessian is nonnegative
	on the whole tangent space and is strictly positive in every
	nonzero direction having nontrivial dependence on $y$.

	\textit{Third step: $L^2$-coercivity on the transverse
	subspace.}
	Let
	\[
	\mathcal X:=
	\left\{
	v\in H^1(\R^N\x\cG):
	\int_{\cG}v(x,y)\,dy=0
	\ \text{for a.e. }x\in\R^N
	\right\}.
	\]
	Assume $\Lambda>0$ and set $c_0:=\Lambda$. Repeating
	\eqref{stima u2}, we obtain
	\begin{equation}\label{coercivita L2}
		Q(v)\geq c_0\|v\|_2^2
		\qquad
		\forall v\in\mathcal X.
	\end{equation}

	\textit{Fourth step: $H^1$-coercivity on the transverse
	subspace.}
	Set $C_Z:=(p-1)\|Z\|_\infty^{p-2}$. By \eqref{forma Q},
	\[
	Q(v)
	\geq
	\|\nabla_{x,y}v\|_2^2-C_Z\|v\|_2^2
	\qquad
	\forall v\in H^1(\R^N\x\cG).
	\]
	Hence, for $\eta\in(0,1)$ and $v\in\mathcal X$, writing
	$Q=\eta Q+(1-\eta)Q$ and applying
	\eqref{coercivita L2} to the second summand, we obtain
	\[
	Q(v)
	\geq
	\eta\|\nabla_{x,y}v\|_2^2
	+
	\left[(1-\eta)c_0-\eta C_Z\right]\|v\|_2^2.
	\]
	Choosing $\eta:=c_0/(2(c_0+C_Z))$, 	we have
	\[
	(1-\eta)c_0-\eta C_Z=\frac{c_0}{2}.
	\]
	Therefore,
	\begin{equation}\label{coercivita H1}
		Q(v)
		\geq
		c_1\|v\|_{H^1(\R^N\x\cG)}^2
		\qquad
		\forall v\in\mathcal X,
		\qquad
		c_1:=
		\min\left\{\eta,\frac{c_0}{2}\right\}>0.
	\end{equation}

	\textit{Fifth step: conclusion of Case $(i)$.}
Let $w\in H^1_\mu(\R^N\x\cG)$ be sufficiently close to $Z$
in $H^1(\R^N\x\cG)$, and decompose
\[
w(x,y)=w_1(x)+w_2(x,y),
\]
where
\[
w_1(x):=
\frac{1}{l(\cG)}\int_{\cG}w(x,y)\,dy,
\qquad
\int_{\cG}w_2(x,y)\,dy=0.
\]
Set $s:=\|w_2\|_2^2$, and $m:=\mu/l(\cG)$. By the orthogonality of the decomposition and the mass
constraint,
\[
\|w_1\|_{L^2(\R^N)}^2
=
m-\frac{s}{l(\cG)}.
\]
Moreover, $\|w-Z\|_{H^1}^2=\|w_1-Z\|_{H^1}^2+\|w_2\|_{H^1}^2$. A second-order Taylor expansion of $E_p(w, \R^N \x \cG)$ with respect to $w_2$, around $w_1$, yields
\begin{equation}\label{sviluppo}
	\begin{split}
		E_p\left(w,\R^N\x\cG\right)&=E_p\left(w_1,\R^N\x \cG\right)+\frac12E_p''\left(w_1,\R^N\x\cG\right)(w_2,w_2)+r(w) \\
		&\geq l(\cG)\cE_p\left(m-\frac{s}{l(\cG)},\R^N\right)+\frac12E_p''\left(w_1,\R^N\x\cG\right)(w_2,w_2)+r(w),
	\end{split}
\end{equation}
where
\[
r(w)=o\left(\|w_2\|_{H^1(\R^N\x\cG)}^2\right)\qquad\text{as }w\to Z\text{ in }H^1(\R^N\x\cG).
\] 
Indeed, the linear term $E_p'(w_1, \R^N \x \cG)(w_2)$ vanishes by \eqref{1571}. Combining the H\"older inequality with exponents $p/(p-2)$, $p/2$, the Sobolev inequality $\|v\|_p \leq \|v\|_{H^1}$, and the continuity of the Nemytskii operator $v\mapsto |v|^{p-2}$ from $L^p$ to $L^{p/(p-2)}$, we infer that
\begin{equation}\label{transf of E_p''}
	\begin{split}
		E_p''(w_1, \R^N \x \cG)(w_2, w_2)&=E_p''(Z, \R^N \x \cG)(w_2, w_2)+(p-1)\int_{\R^N \x \cG}\left(Z^{p-2}-w_1^{p-2}\right)w_2^2\\
		&\geq E_p''(Z, \R^N \x \cG)(w_2, w_2)-(p-1)\|Z^{p-2}-w_1^{p-2}\|_{p/(p-2)}\|w_2\|_p^2 \\
		&\geq E_p''(Z, \R^N \x \cG)(w_2, w_2)-C\|w_2\|_{H^1}^2o(1),
	\end{split}
\end{equation}
for some $C>0$. This identity will be useful below. Now, letting $e(\nu):=\cE_p(\nu,\R^N)$, the identity \eqref{scalamento energia} shows that $e$ is smooth on $(0,+\infty)$, and,
\[
e'(m)=\left(1+\frac{2(p-2)}{4-N(p-2)}\right)m^{\frac{2(p-2)}{4-N(p-2)}}\cE_p(1, \R^N)=\frac{1}{m}\left(1+\frac{2(p-2)}{4-N(p-2)}\right)\cE_p(m, \R^N).
\]
Employing the identities in \eqref{identita di Pohozaev},
\[
e'(m)=\frac{1}{2m}E_p'(Z, \R^N)(Z)=-\frac{\kappa}{2}.
\]
It follows that, as $s\to 0$,
\[\begin{split}
	l(\cG)\cE_p\left(m-\frac{s}{l(\cG)},\R^N\right)
	&=l(\cG)\cE_p\left(m,\R^N\right)+\frac{\kappa}{2}s+o(s)\\
	&=E_p\left(Z,\R^N\x\cG\right)+\frac{\kappa}{2}\|w_2\|_2^2+o\left(\|w_2\|_2^2\right),
\end{split}
\]
recalling that $s=\|w_2\|_2^2$. Combining this estimate with \eqref{transf of E_p''} and \eqref{sviluppo}, we infer that
\[\begin{split}
	E_p\left(w,\R^N\x\cG\right)&-E_p\left(Z,\R^N\x\cG\right)\\
	&\geq\frac12\left(E_p''\left(Z,\R^N\x\cG\right)(w_2,w_2)+\kappa\|w_2\|_2^2\right)-C\|w_2\|_{H^1}^2o(1)+o\left(\|w_2\|_2^2\right)+o\left(\|w_2\|_{H^1}^2\right) \\
	&=\frac{1}{2}Q(w_2)-C\|w_2\|_{H^1}^2o(1)+o\left(\|w_2\|_{H^1}^2\right),
\end{split}
\]
for some $C>0$. Since $w_2\in\mathcal X$, estimate \eqref{coercivita H1} yields
\[
E_p\left(w,\R^N\x\cG\right)-E_p\left(Z,\R^N\x\cG\right)\geq \left(\frac{c_1}{2}+o(1)\right)\|w_2\|_{H^1}^2\geq \frac{c_1}{4}\|w_2\|_{H^1}^2\geq 0
\]
whenever $\|w-Z\|_{H^1}$ is sufficiently small. If
$w_2=0$, the same conclusion follows directly from the global
minimality of $Z$ on $H^1_m(\R^N)$. Thus $Z$ is a local
minimizer, proving $(i)$. It is not a strict local minimizer,
since all its sufficiently small translates have the same mass
and energy.

	\textit{Sixth step: Case $(ii)$.}
	Assume $\Lambda<0$. By the definition of $\beta_1$, there
	exists $v\in H^1(\R^N)\setminus\{0\}$ such that
	\[
	\frac{
	E_p''\left(Z,\R^N\right)(v,v)
	}{
	\|v\|_{L^2(\R^N)}^2
	}
	+\kappa+\alpha_1(\cG)<0.
	\]
	Let $\phi_1$ be a normalized eigenfunction associated with
	$\alpha_1(\cG)$ and set
	\[
	u(x,y):=v(x)\phi_1(y).
	\]
	Since $\int_{\cG}\phi_1=0$, we have
	$u\in T_ZH^1_\mu(\R^N\x\cG)$. Moreover,
	$\|\phi_1\|_{L^2(\cG)}=1$ and
	$\|\phi_1'\|_{L^2(\cG)}^2=\alpha_1(\cG)$, and hence
	\[
	Q(u)
	=
	\left[
	\frac{
		E_p''\left(Z,\R^N\right)(v,v)
	}{
		\|v\|_{L^2(\R^N)}^2
	}
	+\alpha_1(\cG)+\kappa
	\right]
	\|v\|_{L^2(\R^N)}^2
	<0.
	\]

	The curve
	\[
	\gamma(t):=
	\sqrt{\mu}\,
	\frac{Z+tu}{\|Z+tu\|_2}
	\in H^1_\mu(\R^N\x\cG)
	\]
	satisfies $\gamma(0)=Z$ and, since $(Z,u)_2=0$,
	$\gamma'(0)=u$. The first derivative of the energy at
	$t=0$ vanishes, while \eqref{Ep'' da stimare} gives
	\[
	\left.
	\frac{d^2}{dt^2}
	E_p\left(\gamma(t),\R^N\x\cG\right)
	\right|_{t=0}
	=
	Q(u)<0.
	\]
	Therefore,	 $E_p\left(\gamma(t),\R^N\x\cG\right)
	<
	E_p\left(Z,\R^N\x\cG\right)$ for every sufficiently small $t\neq0$. Since
	$\gamma(t)\to Z$ in $H^1(\R^N\x\cG)$ as $t \to 0$, the function $Z$ is
	not a local minimizer.
\end{proof}

	Thanks to Proposition \ref{prop mu1}, we can prove Theorem \ref{thm: loc min semitriv}
\begin{proof}[Proof of Theorem \ref{thm: loc min semitriv}]
	We first show that the denominator appearing in the expression of $\mu_1(\R^N \x \cG,p)$ is strictly positive. By Remarks \ref{rem: EZZ<0} and \ref{E''}, 
	\[\begin{split}
	E_p'\left(Z_{1, p}, \R^N\right)\left(Z_{1, p}\right)- &\beta_1(p,N,1) 	\\
	& \geq -\cE_p\left(1,\R^N\right)\left(\frac{2N(p-2)-4p}{4-N(p-2)}\right)-\left(-\cE_p\left(1,\R^N\right)\right)\left(\frac{2N(p-2)-4p(p-1)}{4-N(p-2)}\right)\\ &=-\cE_p\left(1,\R^N\right)\frac{4p(p-2)}{4-N(p-2)},
	\end{split}\]
	which is strictly positive. To apply Proposition \ref{prop mu1}, it is enough to make the dependence on the mass $\mu$ explicit in the quantity
	\[
	\beta_1(p, N, \mu / l(\cG))-\frac{l(\cG)E_p'\left(Z_{\frac{\mu}{l(\cG)}, p}, \R^N\right)\left(Z_{\frac{\mu}{l(\cG)}, p}\right)}{\mu}+\alpha_1(\cG).
	\]
	By \eqref{scalamento norme},
	\[\begin{split}
	E_p'\left(Z_{\frac{\mu}{l(\cG)}, p}, \R^N\right)\left(Z_{\frac{\mu}{l(\cG)}, p}\right)&=\left\|Z_{\frac{\mu}{l(\cG)}, p}\right\|_{L^2(\R^N)}^2-\left\|Z_{\frac{\mu}{l(\cG)}, p}\right\|_{L^p(\R^N)}^p =\left(\frac{\mu}{l(\cG)}\right)^{1+\frac{2(p-2)}{4-N(p-2)}}E_p'\left(Z_{1, p}, \R^N\right)\left(Z_{1, p}\right).
	\end{split}\]
For every $\nu>0$, define the rescaling operator $R_{\nu}:H^1(\R^N) \rightarrow H^1(\R^N)$ by 
\[
u \longmapsto R_{\nu}(u):=\nu^{\frac{2}{4-N(p-2)}}u\left(\nu^{\frac{p-2}{4-N(p-2)}}x\right).
\]
The map $R_{\nu}$ is a bijection and, by \eqref{da Znup a Z1p}, satisfies $R_{\nu}(Z_{1,p})=Z_{\nu, p}$. A change of variables gives 
\[\begin{split}
\inf\limits_{v \in H^1(\R^N) \setminus\{0\}}\left(\frac{E_p''\left(Z_{\frac{\mu}{l(\cG)}, p}, \R^N\right)(v, v)}{\|v\|_{L^2(\R^N)}^2}\right)&=\inf\limits_{v \in H^1(\R^N) \setminus\{0\}}\left(\frac{E_p''\left(R_{\frac{\mu}{l(\cG)}}(Z_{1, p}, \R^N)\right)\left(R_{\frac{\mu}{l(\cG)}}(v), R_{\frac{\mu}{l(\cG)}}(v)\right)}{\left\|R_{\frac{\mu}{l(\cG)}}(v)\right\|_{L^2(\R^N)}^2}\right) \\
&=\inf\limits_{v \in H^1(\R^N) \setminus\{0\}}\left(\frac{\left(\frac{\mu}{l(\cG)}\right)^{1+\frac{2(p-2)}{4-N(p-2)}}E_p''\left((Z_{1, p}, \R^N)\right)\left(v, v\right)}{\frac{\mu}{l(\cG)}\|v\|_{L^2(\R^N)}^2}\right).
\end{split}\]
Combining these identities, we conclude that
\[\begin{split}
\beta_1(p, N, \mu / l(\cG)) &-\frac{l(\cG)E_p'\left(Z_{\frac{\mu}{l(\cG)}, p}, \R^N\right)\left(Z_{\frac{\mu}{l(\cG)}, p}\right)}{\mu}+\alpha_1(\cG) \\
& =\left(\frac{\mu}{l(\cG)}\right)^{\frac{2(p-2)}{4-N(p-2)}}\left(\beta_1(p, N, 1)-E_p'\left(Z_{1, p}, \R^N\right)\left(Z_{1, p}\right)\right)+\alpha_1(\cG).
\end{split}
\]
Proposition \ref{prop mu1} yields the desired result.
\end{proof} 

\section*{Acknowledgment} 
The authors are affiliated to the INDAM-GNAMPA group. GV acknowledges the MUR grant Dipartimento di
Eccellenza 2023-2027.

\medskip

\noindent \textbf{Data availability:} No data were used for the research described in the article.

\medskip

\noindent \textbf{Conflict of interest:} The authors declare that they have no conflict of interest.

\end{document}